\documentclass[a4paper]{article}

\usepackage{amsfonts}
\usepackage{graphicx}
\usepackage{amsthm}
\usepackage{amsmath}
\usepackage{amssymb}
\usepackage{enumerate}
\allowdisplaybreaks

\usepackage{cancel}
\usepackage{dsfont}

\newtheorem{proposition}{Proposition}[section]
\newtheorem{lemma}[proposition]{Lemma}
\newtheorem{corollary}[proposition]{Corollary}
\newtheorem{theorem}[proposition]{Theorem}

\theoremstyle{definition}
\newtheorem{definition}[proposition]{Definition}

\newtheorem{remark}[proposition]{Remark}
\numberwithin{equation}{section}
\newtheorem*{example*}{Example}

\makeatletter
\newcommand{\customlabel}[2]{%
	\protected@write \@auxout {}{\string \newlabel {#1}{{#2}{}}}}
\makeatother

\newtheorem{reduction}{Reduction}

\begin{document}

\begin{center}
\LARGE
\textbf{Products of three conjugacy classes in the alternating group}
\bigskip\bigskip

\large
Daniele Dona
\bigskip

\normalsize
HUN-REN Alfr\'ed R\'enyi Institute of Mathematics

Re\'altanoda utca 13-15, Budapest 1053, Hungary

\texttt{dona@renyi.hu}
\bigskip\medskip
\end{center}

\begin{minipage}{110mm}
\small
\textbf{Abstract.} We prove that for $\delta$ small, $n$ large, and any three conjugacy classes $C_{1},C_{2},C_{3}$ of $G=\mathrm{Alt}(n)$ of size at least $|G|^{1-\delta}$ we have $C_{1}C_{2}C_{3}=G$.

The result provides a positive answer to Problem 20.23 of the Kourovka Notebook \cite{KM22}, improves theorems of Garonzi and Mar\'oti \cite{GM21} (using $4$ classes) and Rodgers \cite{Rod02} (using larger classes), complements the known result for $G$ a simple group of Lie type \cite{MP21} \cite{LST24} \cite{FM25}, and is tight in several senses. Furthermore, since no character theory is involved, the proof can be used in principle to build a constructive algorithm that, given $g\in G$, outputs $c_{i}\in C_{i}$ such that $c_{1}c_{2}c_{3}=g$.
\medskip

\medskip

\end{minipage}
\medskip

\section{Introduction}

Garonzi and Mar\'oti proved the following fact about \textit{normal} (i.e.\ conjugation-invariant) subsets of $\mathrm{Alt}(n)$.

\begin{theorem}[\cite{GM21}, Thm.~1.1]
For any $\varepsilon>0$, any $n$ large enough depending on $\varepsilon$, and any four normal subsets $S_{1},S_{2},S_{3},S_{4}$ of $G=\mathrm{Alt}(n)$ satisfying $|S_{i}||S_{j}|\geq|G|^{1+\varepsilon}$ for all choices of $1\leq i<j\leq 4$, we have $S_{1}S_{2}S_{3}S_{4}=G$.
\end{theorem}

In particular, for any four conjugacy classes $C_{i}$ with $|C_{i}|\geq|G|^{\frac{1}{2}+\varepsilon}$ we have $C_{1}C_{2}C_{3}C_{4}=G$. Then, they asked whether it was possible to do the same with \textit{three} classes rather than four, up to relaxing the exponent $\frac{1}{2}+\varepsilon$ to $1-\delta$. The question appears as Problem 20.23 in the Kourovka Notebook \cite{KM22}. We answer the question in the affirmative.

\begin{theorem}\label{th:main}
There exist constants $\delta>0$ and $n_{0}$ such that, for any $n\geq n_{0}$ and any three conjugacy classes $C_{1},C_{2},C_{3}$ of $G=\mathrm{Alt}(n)$ with $|C_{i}|\geq|G|^{1-\delta}$, we have $C_{1}C_{2}C_{3}=G$.
\end{theorem}

One can quickly drop the assumption $n\geq n_{0}$ and extend the result to normal subsets. Together with the corresponding result for groups of Lie type (see \cite[Thm.~1.3]{MP21} and \cite[Thm.~7.4]{LST24}), we obtain the following more general theorem.

\begin{theorem}\label{th:mainmore}
There exists a constant $\delta>0$ such that, for any non-abelian finite simple group $G$ and any three normal subsets $S_{1},S_{2},S_{3}\subseteq G$ with $|S_{i}|\geq|G|^{1-\delta}$, we have $S_{1}S_{2}S_{3}=G$.
\end{theorem}

Most of the paper is devoted to proving the following intermediate result.

\begin{theorem}\label{th:main2}
There exist constants $\delta>0$ and $n_{0}$ such that, for any $n\geq n_{0}$ and any three conjugacy classes $C_{1},C_{2},C_{3}$ of $G=\mathrm{Alt}(n)$ with $|C_{i}|\geq|G|^{1-\delta}$, we have $C_{1}C_{2}\supseteq C_{3}$.
\end{theorem}

Again we can drop the assumption $n\geq n_{0}$, but we cannot extend Theorem~\ref{th:main2} to normal subsets: taking $S_{1}=C_{1}$, $S_{2}=C_{2}$ with $C_{1}\neq C_{2}^{-1}$, and $S_{3}=C_{3}\cup\{e\}$ provides an easy counterexample.

\subsection{Relation to previous literature}\label{se:literature}

Theorem~\ref{th:main} improves in full or in part several existing results and complements our knowledge of many others. We mention some of them, although the literature is too rich for the list below to be exhaustive.
\smallskip

\textit{Many classes.} Famously by \cite[Thm.~1.1]{LS01}, for any class $C$ in a non-abelian finite simple group $G$ we have $C^{k}=G$ for $k=O(\log|G|/\log|C|)$, so in particular when $|C|\geq|G|^{\eta}$ there is a constant $k_{\eta}$ such that $C^{k_{\eta}}=G$. By \cite[Thm.~1.3]{MP21} we can take $k_{\eta}=8$ for some $\eta=1-\delta$, and by \cite[Thm.~1.1]{GM21} for $G=\mathrm{Alt}(n)$ large enough we can take $k_{\eta}=4$ for any $\eta=\frac{1}{2}+\varepsilon$. Theorem~\ref{th:main} shows that $k_{\eta}=3$ holds for $\mathrm{Alt}(n)$ and some $\eta=1-\delta$.

Our result is tight in the sense that $k_{1-\delta}=2$ is not achievable, even if we aim to cover $G\setminus\{e\}$ rather than $G$. In fact, for $G=\mathrm{Alt}(n)$ and $n$ large, using Proposition~\ref{pr:gmcycles} there exist three classes of size $1<|C_{1}|<|G|^{\frac{1}{2}\delta}$, $|G|^{1-\delta}<|C_{2}|<|G|^{1-\frac{3}{4}\delta}$, $|C_{3}|>|G|^{1-\frac{1}{4}\delta}$. This means that $C_{2}^{-1}C_{1}\not\supseteq C_{3}$, implying in turn by Remark~\ref{re:littleplay} that $C_{2}C_{3}\not\supseteq C_{1}$. Hence, $|C_{3}|>|C_{2}|>|G|^{1-\delta}$ and $C_{2}C_{3}\subsetneq G\setminus\{e\}$.
\smallskip

\textit{Two large classes.} There are known nontrivial choices of $\eta=1-o(1)$ for which $C^{2}\supseteq\mathrm{Alt}(n)\setminus\{e\}$ whenever $|C|\geq|\mathrm{Alt}(n)|^{\eta}$ (asking the same for every $G$ simple would be stronger than Thompson's conjecture). By \cite[Thm.~1.5]{KLS24} it is enough to ask $|C|\geq e^{-n^{\alpha}}|\mathrm{Alt}(n)|$ for any  $0<\alpha<\frac{2}{5}$ and $n$ large, improving on \cite[Cor.~2.6]{LM23} and \cite[Thm.~4]{LT23}.
\smallskip

\textit{Special classes.} We have results for classes with specific properties. By \cite[Thm.~2.3]{Rod02}, for any three classes $C_{1},C_{2},C_{3}\subseteq\mathrm{Alt}(n)$ such that there are at most $6$ cycles among all three, we have $C_{1}C_{2}C_{3}=\mathrm{Alt}(n)$.

There are also many known sufficient conditions on $C$ in order to have $C^{2}\supseteq\mathrm{Alt}(n)\setminus\{e\}$. This is by no means a rare occurrence: for any $\alpha<\frac{1}{4}$ and $n$ large, by \cite[Thm.~1.13]{LS08} picking a random $\sigma\in\mathrm{Sym}(n)$ and taking $C=\sigma^{\mathrm{Sym}(n)}$ yields $C^{2}=\mathrm{Alt}(n)$ with probability $\geq 1-e^{-n^{\alpha}}$. Using the orbit growth and cycle growth sequences $E(\sigma),B(\sigma)$, one can create conditions on the number of cycles of $\sigma$ of given lengths so that its class $C$ satisfies $C^{2}=\mathrm{Alt}(n)$: see for instance \cite[Thms.~1.9--1.10]{LS08} and \cite[Thm.~1.3]{KLS24}. By \cite[Thm.~3]{LT23}, when $C\subseteq\mathrm{Alt}(n)$ is such that $C^{\mathrm{Sym}(n)}$ is the union of two\footnote{See Proposition~\ref{pr:cyclestr}, which is elementary and classical. These classes are called \textit{split} in \cite[\S 5.1]{FH04} and \textit{exceptional} in \cite[\S 4]{GM21}.} distinct classes $C,C'$ of $\mathrm{Alt}(n)$, both $C^{2}$ and $CC'$ contain $\mathrm{Alt}(n)\setminus\{e\}$.

Yet more results are available if we work with classes of $l$-cycles. By \cite[Thm.~3.4]{HKL08} and \cite[Thm.~A]{KKM24} there are near-exact expressions for $n(k,l)$, the largest $n$ such that every $\sigma\in\mathrm{Alt}(n)$ is the product of $k$ many $l$-cycles; the particular case of four $\left(\left\lfloor\frac{3}{8}n\right\rfloor+1\right)$-cycles was given already in \cite[Cor.~2.4]{Ber72} (and it surpasses the $\frac{1}{2}+\varepsilon$ threshold of \cite{GM21}). By \cite[Thm.~7]{HKL04}, for any $\sigma\in\mathrm{Sym}(n)$ and any $C_{1},C_{2}\subseteq\mathrm{Sym}(n)$ made of $l_{i}$-cycles, there are necessary and sufficient conditions in terms of the support and the number of cycles of $\sigma$ in order to decide whether $\sigma\in C_{1}C_{2}$. By \cite[Thm.~5.1]{Dvi85} we know when $C_{1}C_{2}$ contains the class $O_{l}$ of $l$-cycles (see Proposition~\ref{pr:dvir} below), and by \cite[Cors.~2.1--3.1]{Ber72} we know when $(O_{l})^{2}$ (resp.\ $O_{l}O_{l+1}$) contains $\mathrm{Alt}(n)$ (resp.\ $\mathrm{Sym}(n)\setminus\mathrm{Alt}(n)$).
\smallskip

\textit{Groups of Lie type.} As already mentioned before Theorem~\ref{th:mainmore}, Theorem~\ref{th:main} complements the corresponding result for groups of Lie type. By \cite[Thm.~1.3]{MP21}, or by \cite[Thm.~7.4]{LST24} in the special case of classical groups, if $G$ is a finite simple group of Lie type of large enough rank and we have three classes with $|C_{i}|\geq|G|^{1-\delta}$, then $C_{1}C_{2}C_{3}=G$.

In fact, in this case we know even more, namely that there is good mixing: by \cite[Thm.~1.1]{FM25}, for any $r>16$, any two classes $C_{1},C_{2}$, and any set $A$ with $|C_{1}||C_{2}||A|\geq r|G|^{3-\delta}$, the probability of having $c_{1}c_{2}\in A$ by picking $c_{i}\in C_{i}$ randomly is bounded in the range $\left(1\pm\frac{1}{\sqrt{r}}\right)|A|/|G|$. See the next subsection for comments about mixing in $\mathrm{Alt}(n)$.
\smallskip

\subsection{Remarks on tightness and constructibility}

We make a few observations on these two points.
\smallskip

\textit{Tightness.} As explained in \S\ref{se:literature}, a version of Theorem~\ref{th:main} with two classes cannot hold, even avoiding the identity element.

Our proof does not provide explicit values for $\delta,n_{0}$ but they can be computed in principle, keeping track of quantities in the proof of Theorem~\ref{th:main2} and in parts of \cite{Dvi85} and \cite{GM21} (which enter our computations through Propositions~\ref{pr:gmcycles}--\ref{pr:dvir}). It is known that we cannot take $\delta>\frac{1}{2}$: even for a single class $C$, \cite[Lemma~3.06]{Bre78} shows that for any $\varepsilon>0$ and $n$ large we have $C^{3}\not\supseteq\mathrm{Alt}(n)\setminus\{e\}$ for at least one class of size $|C|\geq|\mathrm{Alt}(n)|^{\frac{1}{2}-\varepsilon}$. Almost surely, if $\delta$ were to be made explicit in our proof, it would not be the sharpest possible value: one look at \cite[Figs.~1--2]{Ber72} should make clear the amount of work one needs to arrange cycles to achieve sharp values by elementary means, even in the simplest case.

Although Theorem~\ref{th:main} can be extended to normal sets, one cannot further relax the condition to allow general sets: by \cite[Thm.~6.2]{Ked09}, for any large finite group $G$ with a transitive action on $\{1,\ldots,n\}$ there are three sets $A,B,C\subseteq G$ with $|A||B||C|\geq\frac{1}{3n}|G|^{3}$ and such that $ABC\neq G$. The so-called Gowers trick however allows us to get close to that, and even achieve good mixing: see \cite[Cor.~1]{NP11}.

Another way in which Theorem~\ref{th:main} is tight is that, unlike for groups of Lie type, in $G=\mathrm{Alt}(n)$ three classes of size $\geq|G|^{1-\delta}$ do not necessarily mix well. Even more strikingly, for any given $\delta>0$, the class $C$ made of one $\lfloor(1-\delta)n\rfloor$-cycle and $\lceil\delta n\rceil$ fixed points, which by Proposition~\ref{pr:gmcycles} has size $\geq|G|^{1-2\delta}$ for $n$ large, does not even have \textit{constant} mixing time when $n\rightarrow\infty$: the $L^{1}$-mixing time for $C$ is bounded above and below by functions of order $\log n$ by \cite[Thm.~6.1]{Roi96b}, and thus the $L^{\infty}$-mixing time is also at least as large. See also results on $L^{2}$-mixing times in \cite[Thm.~2.4]{Vis98} and \cite[Thms.~1--2]{MS07}, and on $L^{1}$ and $L^{2}$-uniform distributions involving different sets in \cite[Thm.~6.1]{LS08} and \cite[Thm.~1.4]{LM23} respectively.
\smallskip

\textit{Constructibility.} Many of the aforementioned results on products of conjugacy classes rely on character theory. Characters have shown their power at least since \cite{LS01}, and are by now the standard tool for dealing with these problems. Power however comes at the cost of non-constructive proofs: showing that for given $g\in G$ there are $c_{i}\in C_{i}$ with $\prod_{i}c_{i}=g$ via these methods does not give a way to construct the $c_{i}$ themselves in the process. An example of the contrast between power and elegance on one side and the use of elementary constructive tools on the other is offered by \cite[\S 2 and \S 5]{DLR24}.

At least in comparison to $G$ of Lie type, the case $G=\mathrm{Alt}(n)$ ought not to be as dependent on character-theoretic tools\footnote{Thanks are due to Pham Huu Tiep, who knows immeasurably more about character theory, for supporting or at least not undermining this feeling of the author (personal communication).}. Elementary methods as the ones used in the past century should be the natural choice, and this is the avenue taken in this paper, which is independent from any character theory. The entire proof is elementary and constructive, and the process of finding $c_{i}$ such that $c_{1}c_{2}c_{3}=g$ is almost self-contained (the exception being Proposition~\ref{pr:dvir}). It is thus possible in principle to build a constructive algorithm outputting the $c_{i}$, given $g$ and the $C_{i}$ as inputs.

Our choice is not purely aesthetic. Computing the number of tuples of $c_{i}\in C_{i}$ satisfying $c_{1}c_{2}\ldots c_{k-1}=c_{k}$ through character methods involves bounding summands of the form $\chi(C_{1})\ldots\chi(C_{k-1})\chi(C_{k}^{-1})/\chi(1)^{k-2}$ for $\chi\in\mathrm{Irr}(G)$. The method is shown to succeed when $k\geq 4$, or when $k\leq 3$ and the classes $C_{i}$ have rather strong constraints on their cycle structure that make $\chi(C_{i})$ easier to compute, for instance analyzing the (exact) Frobenius formula given the cycles of $C_{i}$ and the parts of the partition $\lambda$ defining $\chi=\chi^{\lambda}$. Without those constraints the $\chi(C_{i})$ are too large, and without $k$ large they are not cancelled by the copies of $\chi(1)$ in the denominator; sometimes such a failure is not just an artifact of the method, but leads to concrete counterexamples, as in the non-mixing results of \cite[\S 8.1]{LM23}. It is unclear to the author whether, in the context of the present paper, the main theorem is inherently out of character theory's reach.

\subsection{Spirit of the proof of Thm.~\ref{th:main2}: a baby example}\label{se:example}

After two comparatively short sections on preliminary results (\S\ref{se:prelim}) and on reducing Theorems~\ref{th:main}--\ref{th:mainmore} to Theorem~\ref{th:main2} (\S\ref{se:maintomain}), the bulk of the paper is dedicated to proving Theorem~\ref{th:main2}. Since there are numerous steps and considerable notation, let us try and illustrate with an example the spirit behind the proof.

Suppose we want to build elements $\alpha_{i}\in C_{i}\subseteq\mathrm{Sym}(10)$ such that $\alpha_{1}\alpha_{2}=\alpha_{3}$, where $C_{1}$ is the class of elements made of a $4$-cycle and a $6$-cycle, and $C_{2}=C_{3}$ is the class of $10$-cycles. This is already enough to show that $C_{1}C_{2}\supseteq C_{3}$, thanks to Remark~\ref{re:littleplay}.

We know obvious solutions to simpler problems. For instance, if $k$ is odd then the square of a $k$-cycle is a $k$-cycle, so we know how to solve the problem for $C_{1}=C_{2}=C_{3}$ the class of $k$-cycles in $\mathrm{Sym}(k)$. From that, one may hope to build solutions to more complicated problems. For example, if we have $\alpha\in\mathrm{Sym}(k+h)$ made of two disjoint cycles of length $k$ and $h$, say $\alpha=(x_{1}\cdots x_{k})(y_{1}\cdots y_{h})$, then $\alpha(x_{k}\, y_{1})$ is a $(k+h)$-cycle. This banal observation is useful in that we may manipulate the cycle structure of \textit{two} of the $\alpha_{i}$ without touching the third: if $\alpha_{1}\alpha_{2}=\alpha_{3}$ then also $(\sigma\alpha_{1})\alpha_{2}=(\sigma\alpha_{3})$, $(\alpha_{1}\sigma)(\sigma^{-1}\alpha_{2})=\alpha_{3}$, and so on.

Returning to our baby example, we might try to break the $10$-cycles of $C_{2},C_{3}$ into pairs of $4$-cycles and $6$-cycles, and reduce ourselves to work in $\mathrm{Sym}(4)\times\mathrm{Sym}(6)$ in the hope of simplifying the problem, by glueing back the cycles later. The obvious flaw is that $4$ and $6$ are even, and the product of two $2k$-cycles is never a $2k$-cycle. However, we might confide for now in our ability to glue back more complicated configurations, and thus try and chip away one point from every cycle of length $4$ or $6$. Visually, our cycle manipulations so far look as follows:
\begin{align*}
\begin{aligned}
C_{1}: & & (\text{\textunderscore\ \textunderscore\ \textunderscore\ \textunderscore} )( \text{\textunderscore\ \textunderscore\ \textunderscore\ \textunderscore\ \textunderscore\ \textunderscore}) & & & & (\text{\textunderscore\ \textunderscore\ \textunderscore\ \textunderscore} )( \text{\textunderscore\ \textunderscore\ \textunderscore\ \textunderscore\ \textunderscore\ \textunderscore}) & & & & (\text{\textunderscore\ \textunderscore\ \textunderscore} )( \text{\textunderscore} )( \text{\textunderscore\ \textunderscore\ \textunderscore\ \textunderscore\ \textunderscore} )( \text{\textunderscore}) \\
C_{2}: & & (\text{\textunderscore\ \textunderscore\ \textunderscore\ \textunderscore} \hphantom{)(} \text{\textunderscore\ \textunderscore\ \textunderscore\ \textunderscore\ \textunderscore\ \textunderscore}) & & \longrightarrow & & (\text{\textunderscore\ \textunderscore\ \textunderscore\ \textunderscore} )( \text{\textunderscore\ \textunderscore\ \textunderscore\ \textunderscore\ \textunderscore\ \textunderscore}) & & \longrightarrow & & (\text{\textunderscore\ \textunderscore\ \textunderscore} )( \text{\textunderscore} )( \text{\textunderscore\ \textunderscore\ \textunderscore\ \textunderscore\ \textunderscore} )( \text{\textunderscore}) \\
C_{3}: & & (\text{\textunderscore\ \textunderscore\ \textunderscore\ \textunderscore}\hphantom{)(} \text{\textunderscore\ \textunderscore\ \textunderscore\ \textunderscore\ \textunderscore\ \textunderscore}) & & & & (\text{\textunderscore\ \textunderscore\ \textunderscore\ \textunderscore} )( \text{\textunderscore\ \textunderscore\ \textunderscore\ \textunderscore\ \textunderscore\ \textunderscore}) & & & & (\text{\textunderscore\ \textunderscore\ \textunderscore} )( \text{\textunderscore} )( \text{\textunderscore\ \textunderscore\ \textunderscore\ \textunderscore\ \textunderscore} )( \text{\textunderscore})
\end{aligned}
\end{align*}
Solutions for each subproblem are immediate, and we can put together $(1\, 2\, 3)^{2}=(1\, 3\, 2)$, $(4)^{2}=(4)$, and so on to compose a solution for the simplified problem on the right.

Once we have the simplified solution, we need to go backwards. We first need to glue back each $1$-cycle to the cycle to its left. Fortunately, there are elements that are common to all triples of cycles vertically aligned above (say, $1$ is shared by all $3$-cycles, et cetera), and playing with a handful of them is enough to do the trick, independently from the actual length or composition of the cycles themselves. After a few attempts, we can figure out that, using only the elements $3,4,9,10$, we may perform the following:
\begin{align*}
\begin{aligned}
\gamma_{1}:\!\! & & ( 1\, 2\, \mathbf{3} )( \mathbf{4} )( 5\, 6\, 7\, 8\, \mathbf{9} )( \mathbf{10} ) & & & & \gamma_{1}( 3\, 4 )( 9\, 10 ):\!\! & & ( 1\, 2\, \mathbf{4}\, \mathbf{3} )( 5\, 6\, 7\, 8\, \mathbf{10}\, \mathbf{9} ) \\
\gamma_{2}:\!\! & & ( 1\, 2\, \mathbf{3} )( \mathbf{4} )( 5\, 6\, 7\, 8\, \mathbf{9} )( \mathbf{10} ) & & \longrightarrow & & ( 3\, 9\, 4 )\gamma_{2}( 3\, 10\, 4 ):\!\! & & ( 1\, 2\, \mathbf{10}\, \mathbf{4} )( 5\, 6\, 7\, 8\, \mathbf{9}\, \mathbf{3} )\ \, \\
\gamma_{1}\gamma_{2}:\!\! & & ( 2\, 1\, \mathbf{3} )( \mathbf{4} )( 6\, 8\, 5\, 7\, \mathbf{9} )( \mathbf{10} ) & & & & \gamma_{1}( 4\, 9\, 10 )\gamma_{2}( 3\, 10\, 4 ):\!\! & & ( 2\, 1\, \mathbf{10}\, \mathbf{3} )( 6\, 8\, \mathbf{4}\, 5\, 7\, \mathbf{9} )\ \,
\end{aligned}
\end{align*}
Then we go backwards once more, and glue together the $4$-cycles and $6$-cycles in the second and third line. Even if the elements $3,4,9,10$ have been moved around, making some of them unusable, there are still enough other elements to allow the glueing technique we proposed at the start. Using $2,8$ we find that
\begin{align*}
\begin{aligned}
\beta_{1}:\!\! & & ( 4\, 3\, 1\, \mathbf{2} )( \mathbf{8}\, 10\, 9\, 5\, 6\, 7 ) & & & & \beta_{1}:\!\! & & ( 4\, 3\, 1\, \mathbf{2} )( \mathbf{8}\, 10\, 9\, 5\, 6\, 7 ) \\
\beta_{2}:\!\! & & ( 10\, 4\, 1\, \mathbf{2} )( \mathbf{8}\, 9\, 3\, 5\, 6\, 7 )\ \, & & \longrightarrow & & \beta_{2}(2\, 8):\!\! & & ( 10\, 4\, 1\, \mathbf{8}\, 9\, 3\, 5\, 6\, 7\, \mathbf{2} )\ \, \\
\beta_{1}\beta_{2}:\!\! & & ( 1\, 10\, 3\, \mathbf{2} )( \mathbf{8}\, 4\, 5\, 7\, 9\, 6 )\ \, & & & & \beta_{1}\beta_{2}(2\, 8):\!\! & & ( 1\, 10\, 3\, \mathbf{8}\, 4\, 5\, 7\, 9\, 6\, \mathbf{2} )\ \,
\end{aligned}
\end{align*}
and we are done.

The complete proof of Theorem~\ref{th:main2}, which spans \S\S\ref{se:re1}--\ref{se:solutions} (with a preparatory \S\ref{se:notation} on notation), is a more elaborate version of the example above. We progressively reduce the problem into more manageable subproblems, making sure at each step that finding a solution for the simpler problem means finding a solution for the more complicated one. We identify seven reductions of that sort in \S\S\ref{se:re1}--\ref{se:re7}, one per section: especially from \S\ref{se:re3} the structure becomes evident, with one lemma dedicated to reducing the problem and one lemma dedicated to transforming the solution in each section. In \S\ref{se:solutions}, we find appropriate solutions for the final reduction.

To use the language of later sections for the example above: looking at $C_{i}$ as strings of $(\text{\textunderscore\ \textunderscore\ \textunderscore\ \textunderscore} )$ corresponds to working with \textit{class strings} (\S\ref{se:notation}); the two cycle manipulations are performed in the course of achieving \textit{Reduction~\ref{re:6}} (Lemma~\ref{le:theta6}) and \textit{Reduction~\ref{re:7}} (Lemma~\ref{le:theta7}); vertically aligned cycles \textit{share positions}, and the chosen solutions have common values in them by property $\mathfrak{C}_{1}$ (Definition~\ref{de:aligned}); the solutions are transformed backwards in the course of undoing the two reductions (Lemmas~\ref{le:7to6} and~\ref{le:6to5}).

\section{Preliminaries}\label{se:prelim}

Let us collect here known structural results on conjugacy classes of $\mathrm{Alt}(n)$ and $\mathrm{Sym}(n)$. Before those however, here is a trivial remark.

\begin{remark}\label{re:littleplay}
For any group $G$ and any three conjugacy classes $C_{1},C_{2},C_{3}$, one has $C_{1}C_{2}\supseteq C_{3}$ if and only if there are $\alpha_{i}\in C_{i}$ such that $\alpha_{1}\alpha_{2}=\alpha_{3}$. The ``only if'' direction is obvious for every triple of sets, the ``if'' direction uses the fact that the $C_{i}$ are classes. Another consequence is that $C_{1}C_{2}\supseteq C_{3}$ implies $C_{1}^{-1}C_{3}\supseteq C_{2}$.
\end{remark}

For $x,y\in G$ we write $x^{y}:=y^{-1}xy$. Similarly for $S\subseteq G$ we write $x^{S}:=\{x^{y}|y\in S\}$ and $S^{y}:=\{x^{y}|x\in S\}$. In particular, $x^{G}$ is the conjugacy class of $x$ in $G$.

For $x\in\mathrm{Sym}(n)$, we can write $x$ as product of disjoint cycles, and this product is unique up to $1$-cycles and ordering. For the rest of the paper, we consider the decomposition that includes all the $1$-cycles, so that all $n$ elements appear in the decomposition of $x$. For $1\leq i\leq n$, let $n_{i}(x)$ be the number of cycles of length $i$ in the decomposition. By \textit{cycle structure} of $x$ we mean the sequence $(n_{i}(x))_{1\leq i\leq n}$. The following statement sums up a few textbook facts on elements and classes.

\begin{proposition}\label{pr:cyclestr}
Let $x\in\mathrm{Sym}(n)$. Then $x\in\mathrm{Alt}(n)$ if and only if $\sum_{\text{$i$ even}}n_{i}(x)$ is even.

The conjugacy classes of $\mathrm{Sym}(n)$ are in bijection with the possible cycle structures of its elements. For any $x,y\in\mathrm{Sym}(n)$, $y\in x^{\mathrm{Sym}(n)}$ if and only if $n_{i}(x)=n_{i}(y)$ for all $i$.

Let $x\in\mathrm{Alt}(n)$. If $x$ has either a cycle of even length or two cycles of equal odd length (including length $1$), then $x^{\mathrm{Alt}(n)}=x^{\mathrm{Sym}(n)}$. Otherwise, $x^{\mathrm{Sym}(n)}$ is the union of two distinct classes $x^{\mathrm{Alt}(n)}$ and $(x^{\mathrm{Alt}(n)})^{g}$, where we can choose any $g\in\mathrm{Sym}(n)\setminus\mathrm{Alt}(n)$.
\end{proposition}

In particular it makes sense to talk about $n_{i}(C)$. Define for brevity
\begin{align}\label{eq:cdelta}
\begin{aligned}
\mathcal{C}_{n} & :=\{\,\text{conjugacy classes $C$ of $\mathrm{Sym}(n)$}\,\}, \\
\mathcal{C}_{n}(\delta) & :=\{\,\text{conjugacy classes $C$ of $\mathrm{Sym}(n)$ with $|C|\geq|\mathrm{Sym}(n)|^{1-\delta}$}\,\}, \\
\mathcal{A}_{n} & :=\{\,\text{conjugacy classes $C$ of $\mathrm{Alt}(n)$}\,\}, \\
\mathcal{A}_{n}(\delta) & :=\{\,\text{conjugacy classes $C$ of $\mathrm{Alt}(n)$ with $|C|\geq|\mathrm{Alt}(n)|^{1-\delta}$}\,\}.
\end{aligned}
\end{align}
We know how to relate size and number of cycles.

\begin{proposition}\label{pr:gmcycles}
For any $\delta_{1},\delta_{2}>0$ there is $n_{0}$ such that, for any conjugacy class $C$ of $G=\mathrm{Alt}(n)$ with $n\geq n_{0}$ we have the following:
\begin{enumerate}[(a)]
\item\label{pr:gmcycles-1} if $|C|\geq|G|^{1-\delta_{1}}$, then the number of cycles of $C$ is at most $(\delta_{1}+\delta_{2})n$;
\item\label{pr:gmcycles-2} if the number of cycles of $C$ is at most $\delta_{1}n$, then $|C|\geq|G|^{1-\delta_{1}-\delta_{2}}$.
\end{enumerate}
\end{proposition}

\begin{proof}
See \cite[Prop.~2.3]{DMP24}, which up to straightforward manipulations proves \eqref{pr:gmcycles-2} and refers to \cite[Lemma~2.3]{GM21} for \eqref{pr:gmcycles-1}.
\end{proof}

The result above has several useful consequences.

\begin{corollary}\label{co:classbig}
There is $\delta_{0}>0$ such that for all $0<\delta<\delta_{0}$, all $n$ large enough depending on $\delta$, and all $C\in\mathcal{A}_{n}(\delta)$ the following facts hold.
\begin{enumerate}[(a)]
\item\label{co:classbig-one} There is some $i\geq 1000$ with $n_{i}(C)\geq 1$.
\item\label{co:classbig-fewsmall} $\sum_{i\leq 100}in_{i}(C)<\frac{1}{100}n$.
\item\label{co:classbig-cut} Let $k=\sum_{i\leq n}n_{i}(C)$ be the number of cycles of $C$, and let $l_{1},\ldots,l_{k}$ be their lengths. Take any class $C'\in\mathcal{A}_{n'}$ whose cycles have lengths $l'_{1},\ldots,l'_{k}$ with $l'_{j}\leq l_{j}$ for all $j$; we allow $l'_{j}=0$, in which case there is no $j$-th cycle in $C'$. If $n-n'=\sum_{j\leq k}(l_{j}-l'_{j})\leq\frac{1}{10}n$, then $C'\in\mathcal{A}_{n'}(2\delta)$.
\end{enumerate}
\end{corollary}

\begin{proof}
In all cases, the ``$n$ large enough'' is taken so that $n$ in \eqref{co:classbig-one} and \eqref{co:classbig-fewsmall} or $\frac{9}{10}n$ in \eqref{co:classbig-cut} is larger than $n_{0}$ in Proposition~\ref{pr:gmcycles}, for the choices of $\delta_{1},\delta_{2}$ that we are going to make below depending on $\delta$.

If there are no cycles of length $\geq 1000$ then there are $>\frac{1}{1000}n$ cycles in $C$, so by Proposition~\ref{pr:gmcycles}\eqref{pr:gmcycles-1} with $(\delta_{1},\delta_{2})=(\delta,\delta)$ and $\delta<\frac{1}{2000}$ we obtain \eqref{co:classbig-one}. If the cycles of length $\leq 100$ cover at least $\frac{1}{100}n$ then there are $\geq\frac{1}{10000}n$ cycles in $C$, and again by Proposition~\ref{pr:gmcycles}\eqref{pr:gmcycles-1} with $(\delta_{1},\delta_{2})=(\delta,\delta)$ and $\delta<\frac{1}{20000}$ we obtain \eqref{co:classbig-fewsmall}. Finally, Proposition~\ref{pr:gmcycles}\eqref{pr:gmcycles-1} with $(\delta_{1},\delta_{2})=(\delta,\frac{1}{2}\delta)$ implies that $k\leq\frac{3}{2}\delta n$, and if $k'$ is the number of cycles of $C'$ then $k'\leq k$ (strict inequality is possible if there are $l'_{j}=0$). The hypothesis implies that $k'\leq\frac{5}{3}\delta n'$, and Proposition~\ref{pr:gmcycles}\eqref{pr:gmcycles-2} with $(\delta_{1},\delta_{2})=(\frac{5}{3}\delta,\frac{1}{3}\delta)$ gives \eqref{co:classbig-cut}.
\end{proof}

\section{Proof that Thm.~\ref{th:main2} implies Thms.~\ref{th:main}--\ref{th:mainmore}}\label{se:maintomain}

To prove that $C_{1}C_{2}C_{3}=\mathrm{Alt}(n)$, we show that $C_{1}C_{2}C_{3}$ contains every conjugacy class of $\mathrm{Alt}(n)$. Theorem~\ref{th:main2} will take care of the small classes, while classical methods cover the classes that are not small: in this section we do the latter.

A result from the 1980s lets us reach any class of large enough $m$-cycles.

\begin{proposition}\label{pr:dvir}
Let $C_{1},C_{2}\in\mathcal{C}_{n}$, and for $j\in\{1,2\}$ let $k_{j}=\sum_{i}n_{i}(C_{j})$. Call $O_{m}\in\mathcal{C}_{n}$ the class whose elements are made of one $m$-cycle and $n-m$ $1$-cycles. If $C_{1},C_{2},O_{m}\subseteq\mathrm{Alt}(n)$ and $m\geq k_{1}+k_{2}$, then $C_{1}C_{2}\supseteq O_{m}$.
\end{proposition}

\begin{proof}
See \cite[Thm.~5.1(iii)]{Dvi85}. The proof is elementary and constructive: follow \cite[\S 3 and \S 5]{Dvi85}.
\end{proof}

For $\mathrm{Alt}(n)$ we need a small variation, since by Proposition~\ref{pr:cyclestr} a class $C\in\mathcal{C}_{n}$ with $C\subseteq\mathrm{Alt}(n)$ might be the union of two distinct classes in $\mathrm{Alt}(n)$.

\begin{proposition}\label{pr:dvir2}
Let $C_{1},C_{2}\in\mathcal{A}_{n}$, and let $m$ be the unique odd number in $\{n-3,n-2\}$. Then, for any $\delta>0$ small enough and $n$ large enough depending on $\delta$, if $|C_{1}||C_{2}|\geq|\mathrm{Alt}(n)|^{1+\delta}$ then $C_{1}C_{2}\supseteq O_{m}$.
\end{proposition}

\begin{proof}
By Proposition~\ref{pr:gmcycles}\eqref{pr:gmcycles-1}, if $k_{j}$ is the number of cycles of $C_{j}$ then the condition $m\geq k_{1}+k_{2}$ from Proposition~\ref{pr:dvir} holds. Fix $g\in\mathrm{Sym}(n)\setminus\mathrm{Alt}(n)$. Then by Proposition~\ref{pr:dvir} and Remark~\ref{re:littleplay} we must have at least one of the four inclusions
\begin{align}\label{eq:4incl}
C_{1}C_{2} & \supseteq O_{m}, & C_{1}(C_{2})^{g} & \supseteq O_{m}, & (C_{1})^{g}C_{2} & \supseteq O_{m}, & (C_{1})^{g}(C_{2})^{g} & \supseteq O_{m}.
\end{align}
Not only $O_{m}\in\mathcal{C}_{n}$, but by Proposition~\ref{pr:cyclestr} since $m$ is odd $O_{m}\subseteq\mathrm{Alt}(n)$, and since its elements have at least two $1$-cycles $O_{m}\in\mathcal{A}_{n}$. In particular $O_{m}=(O_{m})^{g}$, so the first inclusion in \eqref{eq:4incl} is equivalent to the fourth and the second is equivalent to the third. Moreover, if either $C_{1}=(C_{1})^{g}$ or $C_{2}=(C_{2})^{g}$ then the first is also equivalent to either the second or the third. Thus, unless $C_{j}\neq(C_{j})^{g}$ for both $j\in\{1,2\}$, any of the inclusions in \eqref{eq:4incl} implies the other three.

We are left with the case in which $C_{1}\neq(C_{1})^{g}$ and $C_{2}\neq(C_{2})^{g}$. A character-theoretic proof of this case is essentially contained in \cite[Lemma~5.3(i)]{GM21}, with the difference that $m\in\{n-1,n\}$ therein, but the technique works out similarly. See also \cite[Thm.~3]{LT23} for the case $C_{1}=C_{2}$. If, in line with the rest of the paper, one desires a constructive proof, just use Propositions~\ref{pr:cyclestr}--\ref{pr:gmcycles}\eqref{pr:gmcycles-2} to show that $|C_{1}|,|C_{2}|,|O_{m}|\geq|\mathrm{Alt}(n)|^{1-\delta}$ and then apply Theorem~\ref{th:main2}.
\end{proof}

We combine Proposition~\ref{pr:dvir2} and Theorem~\ref{th:main2} to obtain Theorems~\ref{th:main}--\ref{th:mainmore}.

\begin{proof}[Proof that Thm.~\ref{th:main2} $\Rightarrow$ Thm.~\ref{th:main}]
Let $\delta>0$ be small enough to make both Proposition~\ref{pr:dvir2} and Theorem~\ref{th:main2} hold, and take $n$ large. We prove Theorem~\ref{th:main} with $\delta/3$, by showing that $C_{1}C_{2}C_{3}\supseteq C$ for any $C\in\mathcal{A}_{n}$. 

First, let $|C|\geq|\mathrm{Alt}(n)|^{2\delta/3}$. By Proposition~\ref{pr:dvir2}, if $m\in\{n-3,n-2\}$ is odd then both $C_{1}C_{2}\supseteq O_{m}$ and $CC_{3}^{-1}\supseteq O_{m}$. Via Remark~\ref{re:littleplay}, we conclude that $C_{1}C_{2}C_{3}\supseteq C$.

Now let $|C|<|\mathrm{Alt}(n)|^{2\delta/3}$. For any $\gamma\in C$ and any $\gamma_{3}\in C_{3}$, the class $C'\ni c_{3}c^{-1}$ must have $|C'|\geq|\mathrm{Alt}(n)|^{1-\delta}$, because otherwise using Remark~\ref{re:littleplay} we get a contradiction by $|C'||C|\geq|C'C|\geq|C_{3}|\geq|\mathrm{Alt}(n)|^{1-\delta/3}>|C'||C|$. But then Theorem~\ref{th:main2} implies that $C_{1}C_{2}\supseteq(C')^{-1}$, and then Remark~\ref{re:littleplay} yields again $C_{1}C_{2}C_{3}\supseteq(C')^{-1}C_{3}\supseteq C$.
\end{proof}

\begin{proof}[Proof that Thm.~\ref{th:main} $\Rightarrow$ Thm.~\ref{th:mainmore}]
By \cite[Thm.~1.3]{MP21}, the result holds for every non-abelian finite simple group $G\neq\mathrm{Alt}(n)$ (see also \cite[Thm.~7.4]{LST24} for $G$ classical).

Let $G=\mathrm{Alt}(n)$, and let $\delta,n_{0}$ be as in Theorem~\ref{th:main}. The number of conjugacy classes in $\mathrm{Alt}(n)$ is at most $2p(n)$ (the partition function), so it is bounded from above by $e^{c\sqrt{n}}$ for some absolute $c$. Therefore, any normal subset $S\subseteq G$ contains a class $C$ of size $|C|\geq e^{-c\sqrt{n}}|S|$. There is also some $n_{1}$ such that for every $n\geq n_{1}$ we have $e^{c\sqrt{n}}\leq|G|^{\frac{1}{2}\delta}$. Set $n_{2}:=\max\{n_{0},n_{1}\}$.

Assume first that $n\geq n_{2}$. For any three $S_{i}$ with $|S_{i}|\geq|G|^{1-\frac{1}{2}\delta}$ there are $C_{i}\subseteq S_{i}$ with $|C_{i}|\geq|G|^{1-\delta}$. Hence, we obtain $G=C_{1}C_{2}C_{3}\subseteq S_{1}S_{2}S_{3}\subseteq G$ by Theorem~\ref{th:main}, and thus $S_{1}S_{2}S_{3}=G$. Assume then that $n<n_{2}$. Then the conclusion of Theorem~\ref{th:mainmore} holds vacuously for an appropriate exponent, say $|G|^{1-(n_{2}\log n_{2})^{-1}}>|G|-1$, because the condition forces $S_{i}=G$.
\end{proof}

\section{Notation}\label{se:notation}

The rest of the paper is devoted to proving Theorem~\ref{th:main2}. We use footnotes to explain the motivation of choices that, at the moment of their occurrence, may seem arbitrary or non-optimal. We dedicate this section to introduce the necessary notation.

Thanks to Remark~\ref{re:littleplay}, the goal is show that, given $C_{1},C_{2},C_{3}$ with $|C_{i}|\geq|G|^{1-\delta}$ for some small enough $\delta$, there exist elements $\alpha_{i}\in C_{i}$ satisfying $\alpha_{1}\alpha_{2}=\alpha_{3}$. Inside $\mathrm{Sym}(n)$, by Proposition~\ref{pr:cyclestr} taking a conjugacy class $C_{i}$ is the same as fixing a cycle structure, and the same is almost true for $\mathrm{Alt}(n)$ as well. Our method involves creating a more rigid structure that fixes also an ordering of the cycles, together with additional information, and the same structure will be used for the elements $\alpha_{i}\in C_{i}$ satisfying $\alpha_{1}\alpha_{2}=\alpha_{3}$.

\subsection{Classes}\label{se:notation-classes}

Here we introduce the language referring to classes and cycle structure, rather than the class elements themselves.

Fix $n$, let $P_{n}:=\{1,2,\ldots,n\}=\mathbb{Z}\cap[1,n]$ be the set of \textit{positions}, and let $\tilde{P}_{n}:=\left\{\frac{1}{2},\frac{3}{2},\ldots,n+\frac{1}{2}\right\}=\left(\frac{1}{2}\mathbb{Z}\cap(0,n+1)\right)\setminus P_{n}$ be the set of \textit{half-positions}. Let $\mathcal{B}=\{\square,\blacksquare\}$ be an alphabet of two symbols, which will denote whether or not we end a cycle and start a new one ($\blacksquare$ is a \textit{cycle break}). A \textit{class string} is a function $\phi:\tilde{P}_{n}\rightarrow\mathcal{B}$ with $\phi(\frac{1}{2})=\phi(n+\frac{1}{2})=\blacksquare$. As we will often do to give a clearer idea of what we are doing, we can visually represent a class string for example as
\begin{align}\label{eq:stringex}
\begin{aligned}
\tilde{P}_{7} & & : & & {\scriptstyle \frac{1}{2} } & & {\scriptstyle \frac{3}{2} } & & {\scriptstyle \frac{5}{2} } & & {\scriptstyle \frac{7}{2} } & & {\scriptstyle \frac{9}{2} } & & {\scriptstyle \frac{11}{2} } & & {\scriptstyle \frac{13}{2} } & & {\scriptstyle \frac{15}{2} } & \\
\phi & & : & & \blacksquare & & \square & & \square & & \blacksquare & & \square & & \square & & \square & & \blacksquare &
\end{aligned}
\end{align}
Let $\Phi_{n}$ be the set of class strings $\phi$. There is a natural function $\mu:\Phi_{n}\rightarrow\mathcal{C}_{n}$: every $\phi$ yields a unique class $C$, defined by taking a cycle of size $r$ whenever $\phi(a)=\phi(a+r)=\blacksquare$ and $\phi(b)=\square$ for all $a<b<a+r$. The string \eqref{eq:stringex} corresponds to the class of $\mathrm{Sym}(7)$ whose elements are the products of two (disjoint) cycles of length $3$ and $4$.

Throughout the paper we define several types of \textit{labels}, to be collected in five sets $\mathcal{N},\mathcal{L},\mathcal{T},\mathcal{S},\mathcal{P}$ and to be assigned to half-positions (an extra symbol ``$\emptyset$'' is used for an empty label). A \textit{label function} is a function $\lambda:\tilde{P}_{n}\rightarrow\mathcal{Y}$, where $\mathcal{Y}$ is the union of some or all the sets of labels above. Let $\Lambda_{n,\mathcal{Y}}$ be the set of label functions $\lambda:\tilde{P}_{n}\rightarrow\mathcal{Y}$. A \textit{string triple} is an element of
\begin{equation*}
\mathfrak{S}=\bigcup_{m=1}^{\infty}\Phi_{m}^{\times 3}\times\Lambda_{m,\{\emptyset\}\cup\mathcal{N}\cup\mathcal{L}\cup\mathcal{T}\cup\mathcal{S}\cup\mathcal{P}},
\end{equation*}
where $X^{\times 3}$ is shorthand for $X\times X\times X$ (so the string ``triple'' is a quadruple, made of three class strings and one label function).

We start in $\mathcal{A}_{n}(\delta)^{\times 3}$ but work mostly in $\mathfrak{S}$, building functions $\theta_{k}:\mathcal{X}_{k-1}\rightarrow\mathcal{X}_{k}$ where $\mathcal{X}_{k}$ is often a subset of $\mathfrak{S}$: with each successive $k$, inside $(\phi_{1},\phi_{2},\phi_{3},\lambda)\in\mathcal{X}_{k}$ the function $\lambda$ becomes more and more complicated while the classes $\mu(\phi_{1}),$ $\mu(\phi_{2}),\mu(\phi_{3})$ yield an easier and easier problem, in the spirit of \S\ref{se:example}. Each time we pass from $\mathcal{X}_{k-1}$ to $\mathcal{X}_{k}$, we call it a \textit{reduction} of the problem; there shall be seven of them in \S\S\ref{se:re1}--\ref{se:re7}, marked by Roman numerals.

\subsection{Elements}\label{se:notation-elements}

Here we deal with language referring to class elements.

Let $V_{n}:=\{1,2,\ldots,n\}=\mathbb{Z}\cap[1,n]$ be the set of \textit{values}; although it is the same as $P_{n}$, we use different names to highlight their different role. If $\Pi_{n}$ is the set of bijective functions $\eta:P_{n}\rightarrow V_{n}$ (thus, $\Pi_{n}\simeq\mathrm{Sym}(n)$), an \textit{element string} is a pair $(\eta,\phi)\in\Pi_{n}\times\Phi_{n}$. Again, there is a natural function $\mu:\Pi_{n}\times\Phi_{n}\rightarrow\mathrm{Sym}(n)$, by which $(\eta,\phi)$ is sent to the product of cycles of the form $(x_{1}\ \ldots\ x_{r})$ whenever there is $a\in P_{n}$ such that
\begin{align*}
\phi\left(a+i+\frac{1}{2}\right) & =\begin{cases} \blacksquare & (i\in\{-1,r\}), \\ \square & (0\leq i<r), \end{cases} & \eta(a+i) & =x_{i} \ \ (1\leq i\leq r).
\end{align*}
This $\mu$ is compatible with the previous $\mu$, in the sense that the element $\mu(\eta,\phi)$ is contained in the class $\mu(\phi)$. For instance, the pair made of the string \eqref{eq:stringex} and the function $\eta$ defined by the sequence $(5,2,3,7,6,1,4)$ is sent to the element $\alpha=(5\,2\,3)(7\,6\,1\,4)$, which we will visually represent as
\begin{align}\label{eq:elemex}
\begin{aligned}
P_{7}\cup\tilde{P}_{7} & & : & & {\scriptstyle \frac{1}{2} } & & {\scriptstyle 1 } & & {\scriptstyle 2 } & & {\scriptstyle 3 } & & {\scriptstyle \frac{7}{2} } & & {\scriptstyle 4 } & & {\scriptstyle 5 } & & {\scriptstyle 6 } & & {\scriptstyle 7 } & & {\scriptstyle \frac{15}{2} } & \\
(\eta,\phi) & & : & & \blacksquare & & 5 & & 2 & & 3 & & \blacksquare & & 7 & & 6 & & 1 & & 4 & & \blacksquare &
\end{aligned}
\end{align}
The values $\square$ are ignored for simplicity, as our focus will be chiefly on how the $\blacksquare$ are distributed and aligned among the three strings.

\subsection{Solutions}\label{se:notation-solutions}

Given a class triple $(\phi_{1},\phi_{2},\phi_{3},\lambda)\in\mathfrak{S}$, a septuple $(\eta_{1},\eta_{2},\eta_{3},\phi_{1},\phi_{2},\phi_{3},\lambda)$ in which the class strings $(\eta_{i},\phi_{i})$ satisfy $\mu(\eta_{1},\phi_{1})\mu(\eta_{2},\phi_{2})=\mu(\eta_{3},\phi_{3})$ is called a \textit{solution}. We often refer to it as a solution for the reduction in which the string triple lies, and we also refer to the triple of elements $\alpha_{i}$ as a solution, especially in Reductions~\ref{re:1}--\ref{re:2} whose setting is not $\mathfrak{S}$.

In the reductions, having $\mu(\eta_{1},\phi_{1})\mu(\eta_{2},\phi_{2})=\mu(\eta_{3},\phi_{3})$ is not the only requirement: the solution must have also some particularly nice structure. Let us define the relevant concepts here.

A \textit{cycle} of a string $\phi\in\Phi_{n}$, or of a string element $(\eta,\phi)\in\Pi_{n}\times\Phi_{n}$, is a restriction of $\phi$ or $(\eta,\phi)$ to some interval $[a,b]$ with $a,b\in\tilde{P}_{n}$ such that $\phi(a)=\phi(b)=\blacksquare$ and $\phi(c)=\square$ for all $a<c<b$. Cycles from different strings or string elements, each corresponding to an interval $[a_{i},b_{i}]$, are said to \textit{share (at least) $k$ common positions} if $\min_{i}b_{i}-\max_{i}a_{i}\geq k$, and they \textit{share a break} (resp., \textit{two breaks}) if for some $a,b$ either $a_{i}=a$ for all $i$ or $b_{i}=b$ for all $i$ (resp., both $a_{i}=a$ and $b_{i}=b$ for all $i$).

As a trivial observation, note that for a fixed $\phi\in\Phi_{n}$ there can be multiple $\eta\in\Pi_{n}$ giving the same $\alpha=\mu(\eta,\phi)\in\mathrm{Sym}(n)$. For example, we can cyclically permute the values of $\eta$ between two consecutive $\blacksquare$, which gives rise to different representations of the same $\alpha$ (just as $(1\, 2\, 3)=(2\, 3\, 1)=(3\, 1\, 2)$ in $\mathrm{Sym}(3)$). We say that a property holds for a triple of element strings $(\eta_{i},\phi_{i})$ \textit{up to cycling around} if the property holds for $(\eta'_{i},\phi_{i})$ with $\eta'_{i}$ obtained by cyclically permuting $\eta_{i}$ in the way we described; we do not necessarily use the same permutation for all $i$.

Having established this language, we can characterize nice solutions.

\begin{definition}\label{de:aligned}
Let $(\eta_{1},\eta_{2},\eta_{3},\phi_{1},\phi_{2},\phi_{3},\lambda)$ be a solution, for each $i$ let $\gamma_{i}$ be a cycle of $(\eta_{i},\phi_{i})$ restricted to $[a_{i},b_{i}]$, and let $k\geq 1$ such that $\gamma_{1},\gamma_{2},\gamma_{3}$ share $\geq k$ common positions, say $x+1,\ldots,x+k$. Fix the following properties, which may or may not be satisfied for $\gamma_{1},\gamma_{2},\gamma_{3}$:
\begin{align*}
\begin{aligned}\ \\ \mathfrak{C}_{1}: \\ (k\geq 1)\end{aligned} & \begin{cases}&\text{up to cycling around, there is a value $z_{1}$} \\ &\text{for which $\eta_{1}(x+1)=\eta_{2}(x+1)=\eta_{3}(x+1)=z_{1}$.} \end{cases} \\
\begin{aligned}\ \\ \mathfrak{C}_{2}: \\ (k\geq 2)\end{aligned} & \begin{cases}&\text{up to cycling around, there are values $z_{1},z_{2}$} \\ &\text{for which $\eta_{1}(x+1)=\eta_{3}(x+1)=z_{1}$} \\ &\text{and $\eta_{1}(x+2)=\eta_{2}(x+2)=z_{2}$.} \end{cases} \\
\begin{aligned}\ \\ \mathfrak{C}_{3}: \\ (k\geq 2)\end{aligned} & \begin{cases}&\text{up to cycling around, there are values $z_{1},z_{2}$} \\ &\text{for which $\eta_{1}(x+1)=\eta_{2}(x+1)=\eta_{3}(x+1)=z_{1}$} \\ &\text{and $\eta_{1}(x+2)=\eta_{2}(x+2)=z_{2}$.} \end{cases} \\
\begin{aligned}\ \\ \mathfrak{C}_{4}: \\ (k\geq 2)\end{aligned} & \begin{cases}&\text{up to cycling around, there are values $z_{1},z_{2}$} \\ &\text{for which $\eta_{1}(x+1)=\eta_{2}(x+1)=z_{1}$} \\ &\text{and $\eta_{1}(x+2)=\eta_{2}(x+2)=\eta_{3}(x+2)=z_{2}$.} \end{cases} \\
\begin{aligned}\ \\ \mathfrak{C}_{5}: \\ (k\geq 2)\end{aligned} & \begin{cases}&\text{up to cycling around, there are values $z_{1},z_{2}$} \\ &\text{for which $\eta_{1}(x+1)=\eta_{2}(x+1)=\eta_{3}(x+1)=z_{2}$} \\ &\text{and $\eta_{1}(x+2)=\eta_{3}(x')=z_{1}$ for some position $x'$ in $\gamma_{3}$.} \end{cases} \\
\begin{aligned}\ \\ \mathfrak{C}_{6}: \\ (k\geq 2)\end{aligned} & \begin{cases}&\text{up to cycling around, there are values $z_{1},z_{2}$} \\ &\text{for which $\eta_{1}(x+1)=\eta_{3}(x')=z_{1}$ for some position $x'$ in $\gamma_{3}$} \\ &\text{and $\eta_{1}(x+2)=\eta_{2}(x+2)=\eta_{3}(x+2)=z_{2}$.} \end{cases}
\end{align*}
The various instances of ``cycling around'' need not be the same, nor does $x'$ have to be shared with $\gamma_{1},\gamma_{2}$. Visually, the properties look like
\begin{align*}
\mathfrak{C}_{1}:& \begin{cases}\begin{aligned}
\hphantom{\square}\hphantom{\square}\blacksquare\ z_{1} \ \hphantom{\square} \blacksquare \hphantom{\square} \\
\hphantom{\square}\blacksquare\hphantom{\square}\ z_{1} \ \blacksquare \hphantom{\square} \hphantom{\square} \\
\blacksquare\hphantom{\square}\hphantom{\square}\ z_{1} \ \hphantom{\square} \hphantom{\square} \blacksquare 
\end{aligned}\end{cases} & 
\!\!\!\mathfrak{C}_{2}:& \begin{cases}\begin{aligned}
\hphantom{\square}\hphantom{\square}\blacksquare\ z_{1}\,z_{2} \ \hphantom{\square} \blacksquare \hphantom{\square} \\
\hphantom{\square}\blacksquare\hphantom{\square}\ \hphantom{z_{1}}\,z_{2} \ \blacksquare \hphantom{\square} \hphantom{\square} \\
\blacksquare\hphantom{\square}\hphantom{\square}\ z_{1}\,\hphantom{z_{2}} \ \hphantom{\square} \hphantom{\square} \blacksquare 
\end{aligned}\end{cases} & 
\!\!\!\mathfrak{C}_{3}:& \begin{cases}\begin{aligned}
\hphantom{\square}\hphantom{\square}\blacksquare\ z_{1}\,z_{2} \ \hphantom{\square} \blacksquare \hphantom{\square} \\
\hphantom{\square}\blacksquare\hphantom{\square}\ z_{1}\,z_{2} \ \blacksquare \hphantom{\square} \hphantom{\square} \\
\blacksquare\hphantom{\square}\hphantom{\square}\ z_{1}\,\hphantom{z_{2}} \ \hphantom{\square} \hphantom{\square} \blacksquare 
\end{aligned}\end{cases} \\
\mathfrak{C}_{4}:& \begin{cases}\begin{aligned}
\hphantom{\square}\hphantom{\square}\blacksquare\ z_{1}\,z_{2} \ \hphantom{\square} \blacksquare \hphantom{\square} \\
\hphantom{\square}\blacksquare\hphantom{\square}\ z_{1}\,z_{2} \ \blacksquare \hphantom{\square} \hphantom{\square} \\
\blacksquare\hphantom{\square}\hphantom{\square}\ \hphantom{z_{1}}\,z_{2} \ \hphantom{\square} \hphantom{\square} \blacksquare 
\end{aligned}\end{cases} & 
\!\!\!\mathfrak{C}_{5}:& \begin{cases}\begin{aligned}
\hphantom{\square}\hphantom{\square}\blacksquare\ z_{1}\,\hphantom{z_{2}}\hphantom{\square} \hphantom{\square} \blacksquare \hphantom{\square} \\
\hphantom{\square}\blacksquare\hphantom{\square}\ z_{1}\, z_{2}\hphantom{\square} \blacksquare \hphantom{\square} \hphantom{\square} \\
\blacksquare\, z_{2}\ \hphantom{\square} z_{1}\,\hphantom{z_{2}\square} \hphantom{\square} \hphantom{\square} \blacksquare 
\end{aligned}\end{cases} & 
\!\!\!\mathfrak{C}_{6}:& \begin{cases}\begin{aligned}
\hphantom{\square}\hphantom{\square}\blacksquare\ z_{1}\,z_{2}\hphantom{\square} \hphantom{\square} \blacksquare \hphantom{\square} \\
\hphantom{\square}\blacksquare\hphantom{\square}\ \hphantom{z_{1}}\, z_{2}\hphantom{\square} \blacksquare \hphantom{\square} \hphantom{\square} \\
\blacksquare\, z_{1}\ \hphantom{\square z_{1}}\,z_{2}\hphantom{\square} \hphantom{\square} \hphantom{\square} \blacksquare 
\end{aligned}\end{cases}
\end{align*}
We can compose properties: if $\mathfrak{C},\mathfrak{C}'$ are two properties, we write $\mathfrak{C}\cap\mathfrak{C}'$ to indicate that both of them hold, and we write $\mathfrak{C}\mathfrak{C}'$ to indicate that they hold using disjoint sets of values. Again, the instances of ``cycling around'' for $\mathfrak{C},\mathfrak{C}'$ in their composition need not be the same.

The solution $(\eta_{1},\eta_{2},\eta_{3},\phi_{1},\phi_{2},\phi_{3},\lambda)$ is \textit{aligned} if, for all triples of cycles $\gamma_{i}\in(\eta_{i},\phi_{i})$, the following properties hold:
\begin{enumerate}[(1)]
\item\label{de:aligned-al1} if $\gamma_{1},\gamma_{2},\gamma_{3}$ share $\geq 1$ common positions, then $\mathfrak{C}_{1}$ holds for $\gamma_{1},\gamma_{2},\gamma_{3}$;
\item\label{de:aligned-alp} if $\gamma_{1},\gamma_{2},\gamma_{3}$ share $\geq 2$ common positions, and they share a break ($a$ or $b$) with $\lambda(a)\in\mathcal{P}$ and $\lambda(a-1)\notin\mathcal{L}$, or $\lambda(b)\in\mathcal{P}$ and $\lambda(b+1)\notin\mathcal{L}$, then $\mathfrak{C}_{1}^{2}\cap\mathfrak{C}_{1}\mathfrak{C}_{2}$ holds for $\gamma_{1},\gamma_{2},\gamma_{3}$;
\item\label{de:aligned-all} if $\gamma_{1},\gamma_{2},\gamma_{3}$ share $\geq 2$ common positions, and they share a break ($a$ or $b$) with $\lambda(a)\in\mathcal{L}$, or $\lambda(b)\in\mathcal{L}$, then $\mathfrak{C}_{1}^{3}\mathfrak{C}_{3}^{2}\mathfrak{C}_{4}^{2}\mathfrak{C}_{5}\mathfrak{C}_{6}$ holds for $\gamma_{1},\gamma_{2},\gamma_{3}$;
\item\label{de:aligned-alpl} if $\gamma_{1},\gamma_{2},\gamma_{3}$ share $\geq 2$ common positions, and they share a break ($a$ or $b$) with $(\lambda(a-1),\lambda(a))\in\mathcal{L}\times\mathcal{P}$, or $(\lambda(b),\lambda(b+1))\in\mathcal{P}\times\mathcal{L}$, then $\mathfrak{C}_{1}^{4}\mathfrak{C}_{2}\mathfrak{C}_{3}^{2}\mathfrak{C}_{4}^{2}\mathfrak{C}_{5}\mathfrak{C}_{6}$ holds for $\gamma_{1},\gamma_{2},\gamma_{3}$;
\item\label{de:aligned-disj} if $(\gamma_{1},\gamma_{2},\gamma_{3})\neq(\gamma'_{1},\gamma'_{2},\gamma'_{3})$, the two sets of values used in \eqref{de:aligned-al1}--\eqref{de:aligned-alp}--\eqref{de:aligned-all}--\eqref{de:aligned-alpl} for the two triples are disjoint from each other\footnote{This is not implicit in the act of taking distinct triples. For instance the three element strings $\blacksquare 1 2 \blacksquare 3 4 \blacksquare$, $\blacksquare 1 4 2 3 \blacksquare$, $\blacksquare 1 3 2 4 \blacksquare$ satisfy $\mu(\eta_{1},\phi_{1})\mu(\eta_{2},\phi_{2})=\mu(\eta_{3},\phi_{3})$ and for the four cycles $\gamma_{1},\gamma'_{1},\gamma_{2},\gamma_{3}$ we have that $(\gamma_{1},\gamma_{2},\gamma_{3}),(\gamma'_{1},\gamma_{2},\gamma_{3})$ both satisfy $\mathfrak{C}_{5}$, but not with disjoint sets of values: we must choose $(z_{1},z_{2})\in\{(1,4),(2,3)\}$ for the first triple and $(z_{1},z_{2})\in\{(3,1),(4,2)\}$ for the second triple.}.
\end{enumerate}
Visually, the properties look like
\begin{align*}
\eqref{de:aligned-al1}:\ & \begin{cases}\begin{aligned}
\hphantom{\square}\hphantom{\square}\blacksquare\ \hphantom{\square \mathfrak{C}_{1}\square} \ \hphantom{\square} \blacksquare \hphantom{\square} \\
\hphantom{\square}\blacksquare\hphantom{\square}\ \hphantom{\square} \mathfrak{C}_{1}\hphantom{\square} \ \blacksquare \hphantom{\square} \hphantom{\square} \\
\blacksquare\hphantom{\square}\hphantom{\square}\ \hphantom{\square \mathfrak{C}_{1}\square} \ \hphantom{\square} \hphantom{\square} \blacksquare \\
\text{(any $\lambda$)}\hphantom{nmmm}
\end{aligned}\end{cases} &
\eqref{de:aligned-alp}:\ & \begin{cases}\begin{aligned}
\hphantom{\square}\hphantom{\square}\blacksquare\ \hphantom{\square \mathfrak{C}_{1}^{2}\cap\mathfrak{C}_{1}\mathfrak{C}_{2} \square} \ \hphantom{\square} \blacksquare \ \blacksquare \\
\hphantom{\square}\blacksquare\hphantom{\square}\ \hphantom{\square} \mathfrak{C}_{1}^{2}\cap\mathfrak{C}_{1}\mathfrak{C}_{2} \hphantom{\square} \hphantom{\square} \ \blacksquare \ \blacksquare \\
\blacksquare\hphantom{\square}\hphantom{\square}\ \hphantom{\square \mathfrak{C}_{1}^{2}\cap\mathfrak{C}_{1}\mathfrak{C}_{2} \square} \ \hphantom{\square} \blacksquare \ \blacksquare \\
\mathcal{P} \ \xcancel{\mathcal{L}}
\end{aligned}\end{cases} \\
\eqref{de:aligned-all}:\ & \begin{cases}\begin{aligned}
\hphantom{\square}\hphantom{\square}\blacksquare\ \hphantom{\square \mathfrak{C}_{1}^{3}\mathfrak{C}_{3}^{2}\mathfrak{C}_{4}^{2}\mathfrak{C}_{5}\mathfrak{C}_{6} \square} \ \blacksquare \\
\hphantom{\square}\blacksquare\hphantom{\square}\ \hphantom{\square} \mathfrak{C}_{1}^{3}\mathfrak{C}_{3}^{2}\mathfrak{C}_{4}^{2}\mathfrak{C}_{5}\mathfrak{C}_{6} \hphantom{\square} \ \blacksquare \\
\blacksquare\hphantom{\square}\hphantom{\square}\ \hphantom{\square \mathfrak{C}_{1}^{3}\mathfrak{C}_{3}^{2}\mathfrak{C}_{4}^{2}\mathfrak{C}_{5}\mathfrak{C}_{6} \square} \ \blacksquare \\
\mathcal{L}
\end{aligned}\end{cases} &
\eqref{de:aligned-alpl}:\ & \begin{cases}\begin{aligned}
\hphantom{\square}\hphantom{\square}\blacksquare\ \hphantom{\square \mathfrak{C}_{1}^{4}\mathfrak{C}_{2}\mathfrak{C}_{3}^{2}\mathfrak{C}_{4}^{2}\mathfrak{C}_{5}\mathfrak{C}_{6} \square} \ \blacksquare \ \blacksquare \\
\hphantom{\square}\blacksquare\hphantom{\square}\ \hphantom{\square} \mathfrak{C}_{1}^{4}\mathfrak{C}_{2}\mathfrak{C}_{3}^{2}\mathfrak{C}_{4}^{2}\mathfrak{C}_{5}\mathfrak{C}_{6} \hphantom{\square} \ \blacksquare \ \blacksquare \\
\blacksquare\hphantom{\square}\hphantom{\square}\ \hphantom{\square \mathfrak{C}_{1}^{4}\mathfrak{C}_{2}\mathfrak{C}_{3}^{2}\mathfrak{C}_{4}^{2}\mathfrak{C}_{5}\mathfrak{C}_{6} \square} \ \blacksquare \ \blacksquare \\
\mathcal{P} \ \mathcal{L}
\end{aligned}\end{cases} 
\end{align*}
\end{definition}

\subsection{Other notation}\label{se:notation-other}

We collect here some additional notation. If $C_{1},C_{2},C_{3}\in\mathcal{C}_{n}$, define
\begin{equation}\label{eq:nu}
\nu_{C_{1},C_{2},C_{3}}:=\left| \vphantom{\tilde{P}_{n}} \{i\in\{1,2,3\}\,|\,C_{i}\subseteq\mathrm{Alt}(n)\}\right|\in\{0,1,2,3\}.
\end{equation}
Clearly, if $\nu_{C_{1},C_{2},C_{3}}$ is even then we cannot have any $\alpha_{i}\in C_{i}$ with $\alpha_{1}\alpha_{2}=\alpha_{3}$. For a triple of strings $\phi_{1},\phi_{2},\phi_{3}\in\Phi_{n}$ and a half-position $a\in\tilde{P}_{n}$ (mostly known by context), we will denote by $c^{\blacksquare}(\phi_{i})$ and $c^{\blacksquare}(a)$ two functions that count the number of values $\blacksquare$, i.e.
\begin{align}\label{eq:functbcount}
c^{\blacksquare}(\phi_{i}) & :=\left|\{x\in\tilde{P}_{n}\,|\,\phi_{i}(x)=\blacksquare\}\right|, & c^{\blacksquare}(a) & :=\left|\{j\,|\,\phi_{j}(a)=\blacksquare\}\right|.
\end{align}

In future discussions, for brevity we shall occasionally use shortened notations for properties that the string triples must respect. For instance we may ask for $(\phi_{1},\phi_{2},\phi_{3},\lambda)\in\mathfrak{S}$ to satisfy the property ``$\mathfrak{P}(\text{labels})$'' defined as follows:
\begin{equation}\label{eq:prop-labels}
\mathfrak{P}(\mathrm{labels}): \begin{cases}&\text{for all $a$, if $\lambda(a)\in\mathcal{N}$ then $c^{\blacksquare}(a)=0$,} \\ &\text{and if $\lambda(a)\in\mathcal{L}\cup\mathcal{T}\cup\mathcal{S}\cup\mathcal{P}$ then $c^{\blacksquare}(a)=3$.}\end{cases}
\end{equation}
Other properties will be defined when appropriate.

Again for brevity, starting from Lemma~\ref{le:2to1} and occasionally for the rest of the paper, when writing down solutions or elements in an explicit way we may write ``$[a \overset{\text{odds}}{\cdots} b]$'' to indicate a sequence made of all the odd numbers from $a$ to $b$ (included), and similarly ``$[a \overset{\text{evens}}{\cdots} b]$'' and ``$[a \overset{\text{all}}{\cdots} b]$''. We extend the same notation to indices, meaning that ``$[x_{a} \overset{\text{odds}}{\cdots} x_{b}]$'' denotes a sequence made of all $x_{i}$ with $a\leq i\leq b$ odd, and analogously for the other cases.

Another way in which we may abbreviate elements is the following. If $\alpha\in\mathrm{Sym}(n)$ and $r\in V_{n}$ is a value of some significance to us, we may write $\alpha=\vec{\beta}(r\ \vec{\rho})$: here, $\vec{\rho}$ is the sequence of elements in the same cycle as $r$, say $\vec{\rho}$ is made of elements $\rho_{1},\ldots,\rho_{k}\in V_{n}$ with $\alpha(r)=\rho_{1}$, $\alpha(\rho_{i})=\rho_{i+1}$ for $1\leq i<k$, and $\alpha(\rho_{k})=r$, whereas $\vec{\beta}$ is the product of all the disjoint cycles of $\alpha$ not containing $r$. Both $\vec{\rho}$ and $\vec{\beta}$ might be empty and, if we have $\alpha_{i}=\vec{\beta}_{i}(r\ \vec{\rho}_{i})$ for several indices $i$, the $\vec{\rho}_{i},\vec{\beta}_{i}$ do not necessarily have the same length for all $i$.

\section{Reduction~\ref{re:1}: passing to $\mathrm{Sym}(n)$}\label{se:re1}

In the first reduction, we simply pass from a problem in $\mathrm{Alt}(n)$ to a problem in $\mathrm{Sym}(n)$.

Fix some $g\in\mathrm{Sym}(n)\setminus\mathrm{Alt}(n)$. If $C_{i}\in\mathcal{A}_{n}$ then $C'_{i}:=C_{i}\cup(C_{i})^{g}\in\mathcal{C}_{n}$ by Proposition~\ref{pr:cyclestr}, independently from the choice of $g$. Moreover, taking $n$ large depending on $\delta$, if $C_{i}\in\mathcal{A}_{n}(\delta)$ then $C'_{i}\in\mathcal{C}_{n}(2\delta)$. Finally, $\nu_{C'_{1},C'_{2},C'_{3}}=3$ by construction.

Let
\begin{equation*}
\mathcal{X}_{1}(n,\delta):=\{\, (D_{1},D_{2},D_{3})\in\mathcal{C}_{n}^{\times 3}(\delta)\ |\ \text{$\nu_{D_{1},D_{2},D_{3}}$ odd}\, \},
\end{equation*}
and let
\begin{align*}
\theta_{1} & :\mathcal{A}_{n}(\delta)^{\times 3}\rightarrow\mathcal{X}_{1}(n,2\delta), & \theta_{1}(C_{1},C_{2},C_{3})=(C_{1}\cup(C_{1})^{g},C_{2}\cup(C_{2})^{g},C_{3}\cup(C_{3})^{g})
\end{align*}
for any fixed $g\in\mathrm{Sym}(n)\setminus\mathrm{Alt}(n)$. By the previous discussion, the definition of $\theta_{1}$ makes sense for all $n$ large enough with respect to $\delta$.

\begin{reduction}\label{re:1}
Find, for any $(C_{1},C_{2},C_{3})\in\mathcal{X}_{1}(n,\delta_{1})$, elements $\alpha_{i}\in C_{i}$ with $\alpha_{1}\alpha_{2}=\alpha_{3}$. Furthermore, the $\alpha_{i}$ must satisfy the following property: there are $x_{1},\ldots,x_{5}\in P_{n}$ such that $(x_{i})^{\alpha_{1}}=(x_{i})^{\alpha_{2}}=x_{i+1}$ for $1\leq i\leq 4$.
\end{reduction}

\begin{lemma}\label{le:1to0}
If there is a solution for Reduction~\ref{re:1}, for all $\delta_{1}>0$ small enough and all $n$ large enough depending on $\delta_{1}$, then Theorem~\ref{th:main2} holds, for all $\delta>0$ small enough and all $n$ large enough depending on $\delta$.
\end{lemma}

\begin{proof}
Let $C_{1},C_{2},C_{3}\in\mathcal{A}_{n}(\delta)$ with $\delta,n$ such that the function $\theta_{1}$ is well-defined and there is a solution for Reduction~\ref{re:1} with $\delta_{1}=2\delta$. Thus, let $\alpha_{i}$ be as in Reduction~\ref{re:1} for $\theta_{1}(C_{1},C_{2},C_{3})$.

Since $C_{i},(C_{i})^{g}$ are two (possibly identical) conjugacy classes in $\mathrm{Alt}(n)$ giving the same $C_{i}\cup(C_{i})^{g}$, it is enough to start from our triple of $\alpha_{i}$ and produce solutions for the eight possible class combinations. Without loss of generality, say that $\alpha_{i}\in C_{i}$ for all $i$, and call $C_{i}^{\perp}=(C_{i})^{g}$. Rename $x_{1},\ldots,x_{5}$ in Reduction~\ref{re:1} as $1,\ldots,5$ for simplicity, and as in \S\ref{se:notation-other} write
\begin{align*}
\alpha_{1} & =(1\ 2\ 3\ 4\ 5\ \vec{\rho}_{1})\vec{\beta}_{1}, & \alpha_{2} & =(1\ 2\ 3\ 4\ 5\ \vec{\rho}_{2})\vec{\beta}_{2}.
\end{align*}

Let $\sigma=(1\ 3)(2\ 4)=\sigma^{-1}$ and $\tau=(1\ 3\ 5\ 2\ 4)$. Then, in addition to the already known solution with $(\alpha_{1},\alpha_{2},\alpha_{3})\in(C_{1},C_{2},C_{3})$, three more solutions are given as follows:
\begin{align*}
& \begin{aligned}
\alpha'_{1} & =\alpha_{1}\sigma=(\alpha_{1})^{(1\ 3)} \\ \alpha'_{2} & =\sigma\alpha_{2}=(\alpha_{2})^{(2\ 4)}
\end{aligned}
& & \Longrightarrow & \alpha'_{1}\alpha'_{2} & =\alpha_{3}, & (\alpha'_{1},\alpha'_{2},\alpha_{3}) & \in(C_{1}^{\perp},C_{2}^{\perp},C_{3}), \\
& \begin{aligned}
\alpha''_{1} & =\alpha_{1}\tau=(\alpha_{1})^{(1\ 3\ 4\ 2)} \\ \alpha''_{2} & =\tau^{-1}\alpha_{2}=(\alpha_{2})^{(2\ 5)(3\ 4)}
\end{aligned} 
& & \Longrightarrow & \alpha''_{1}\alpha''_{2} & =\alpha_{3}, & (\alpha''_{1},\alpha''_{2},\alpha_{3}) & \in(C_{1}^{\perp},C_{2},C_{3}), \\
& \begin{aligned}
\alpha'''_{1} & =\alpha_{1}\tau^{-1}=(\alpha_{1})^{(1\ 4)(2\ 5)} \\ \alpha'''_{2} & =\tau\alpha_{2}=(\alpha_{2})^{(2\ 4\ 5\ 3)}
\end{aligned} 
& & \Longrightarrow & \alpha'''_{1}\alpha'''_{2} & =\alpha_{3}, & (\alpha'''_{1},\alpha'''_{2},\alpha_{3}) & \in(C_{1},C_{2}^{\perp},C_{3}).
\end{align*}
The remaining four solutions with $C_{3}^{\perp}$ can be found by conjugating the four known solutions by $(1\ 2)\notin\mathrm{Alt}(n)$. The result follows.
\end{proof}

\section{Reduction~\ref{re:2}: short even cycles are even}\label{se:re2}

In the second reduction, we make sure that in each class there is an even number of short cycles of even length. We fix\footnote{The choice of $31$ is tight for this procedure. Reduction~\ref{re:3} removes cycles of length $\leq 31$ and, although they can be reintroduced later in a controlled way, at the end we have properties $\mathfrak{P}(\mathrm{sub}')$ and $\mathfrak{P}(\mathcal{L}',19)$ in $\mathcal{X}_{7}$ (see \eqref{eq:prop-lprime} and \eqref{eq:prop-subprime}): here, $19$ comes from shortening $31$ by $12$ places. Then, to make the solution aligned in Proposition~\ref{pr:7good}, in the worst case we might need $19$ values to play with.} the meaning of ``short'' as being $\leq 31$.

For $C\in\mathcal{C}_{n}$, let $s_{k}(C)$ be the number of cycles of $C$ of length $\leq k$ and even, i.e.\ $s_{k}(C):=\sum_{i=1}^{\lfloor k/2\rfloor}n_{2i}(C)$, and let $t(C)$ be the length of the largest cycle in $C$, i.e.\ $t(C):=\max\{i|n_{i}(C)\geq 1\}$. Given a triple $C_{1},C_{2},C_{3}$, let $J_{k}\subseteq\{1,2,3\}$ be the set of indices $j$ for which $s_{k}(C_{j})$ is odd: if $j\in J_{k}$ then $C_{j}$ has a cycle of even length $\leq k$, and in that case call $r_{k}(C_{j})$ the length of the largest such cycle, whereas if $j\notin J_{k}$ set $r_{k}(C_{j})=0$. Define\footnote{In passing from $n$ to $n'$ below, every time we remove the $r_{k}(C_{i})$-cycle, we also choose to remove $1$ extra point to keep $\nu$ odd, and we remove $8$ extra points in order to have $5$ consecutive common values as required by Reduction~\ref{re:1}.}
\begin{align*}
R_{k} & =R_{k}(C_{1},C_{2},C_{3}):=8+\sum_{j\in J_{k}}(r_{k}(C_{j})+1), & n' & :=n-R_{k}.
\end{align*}
For fixed $k$ the range of values in which $n'$ can vary is $[n-3k-11,n-8]$, regardless of the choice of $C_{j}$.

For each $C_{i}$, let $C'_{i}\in\mathcal{C}_{n'}$ be the class obtained from $C_{i}$ by removing the fixed $r_{k}(C_{i})$-cycle (if any) and replacing a $t(C_{i})$-cycle with a $t'_{i}$-cycle for $t'_{i}:=t(C_{i})+r_{k}(C_{i})-R_{k}$. For $C_{i}\in\mathcal{C}_{n}(\delta_{1})$ and $n$ large depending on $\delta_{1}$ we have $t(C_{i})\geq 1000$ by Corollary~\ref{co:classbig}\eqref{co:classbig-one}, so in that case the definition of $C'_{i}$ makes sense for $k=31$, and in particular $t'_{i}\geq 1000-2\cdot 31-11>31$. By Corollary~\ref{co:classbig}\eqref{co:classbig-cut}, if $C_{i}\in\mathcal{C}_{n}(\delta_{1})$ then also $C'_{i}\in\mathcal{C}_{n}(2\delta_{1})$. Checking the changes in cycle lengths case by case, one verifies also that if $\nu_{C_{1},C_{2},C_{3}}$ is odd then $\nu_{C'_{1},C'_{2},C'_{3}}$ is odd. Finally, by the procedure just described, each $C'_{i}$ has now an even number of cycles of length $\leq k$ and even, i.e.\ $J_{k}=\emptyset$ for this new triple of classes $C'_{i}$.

Let
\begin{align*}
\mathcal{X}_{2}(m,\delta):= \ & \left\{\left. (D_{1},D_{2},D_{3},t_{1},t_{2},t_{3})\in\mathcal{C}_{m}^{\times 3}(\delta)\times\mathbb{N}_{>31}^{\times 3} \ \right|\right. \\
 & \text{$\nu_{D_{1},D_{2},D_{3}}$ odd, }J_{31}=\emptyset, \ \text{$D_{i}$ has a $t_{i}$-cycle} \ \},
\end{align*}
and let
\begin{align*}
\theta_{2} & :\mathcal{X}_{1}(n,\delta_{1})\rightarrow\bigcup_{n'=n-104}^{n-8}\mathcal{X}_{2}(n',2\delta_{1}), & \theta_{2}(C_{1},C_{2},C_{3})=(C'_{1},C'_{2},C'_{3},t'_{1},t'_{2},t'_{3}),
\end{align*}
where the $C'_{i}$ and the $t'_{i}$ are constructed as above. By the previous discussion, the definition of $\theta_{2}$ makes sense for all $n$ large enough depending on $\delta_{1}$.

\begin{reduction}\label{re:2}
Find, for any $(C_{1},C_{2},C_{3},t_{1},t_{2},t_{3})\in\mathcal{X}_{2}(n,\delta_{2})$, elements $\alpha_{i}\in C_{i}$ with $\alpha_{1}\alpha_{2}=\alpha_{3}$. Furthermore, the $\alpha_{i}$ must satisfy the following property: there is a triple of cycles, of length $t_{1},t_{2},t_{3}$ in $\alpha_{1},\alpha_{2},\alpha_{3}$ respectively, sharing at least one value $r$.
\end{reduction}

\begin{lemma}\label{le:2to1}
If there is a solution for Reduction~\ref{re:2}, for all $\delta_{2}>0$ small enough and all $n$ large enough depending on $\delta_{2}$, then there is a solution for Reduction~\ref{re:1}, for all $\delta_{1}>0$ small enough and all $n$ large enough depending on $\delta_{1}$.
\end{lemma}

\begin{proof}
 Let $(C_{1},C_{2},C_{3})\in\mathcal{X}_{1}:=\mathcal{X}_{1}(n,\delta_{1})$, so that $C_{i}\in\mathcal{C}_{n}(\delta_{1})$ for each $i$, and let $\theta_{2}$ be well-defined thanks to our choice of $\delta_{1},n$. Given $(C'_{1},C'_{2},C'_{3},t'_{1},t'_{2},t'_{3})=\theta_{2}(C_{1},C_{2},C_{3})$, let $\alpha'_{1},\alpha'_{2},\alpha'_{3}$ be a corresponding solution for Reduction~\ref{re:2}. Set $J:=J_{31}$ and $r_{j}:=r_{31}(C_{j})$.

To walk back the procedure and obtain the $\alpha_{i}$, we first add the $r_{j}$-cycles. By hypothesis, there is a triple of $t'_{i}$-cycles and there is a value $r$ such that each $t'_{i}$-cycle in that triple is of the form $(r\ \vec{\rho}_{i})$, so write $\alpha_{i}=(r\ \vec{\rho}_{i})\vec{\beta}_{i}$ as in \S\ref{se:notation-other}. Fix one $j\in J$: we introduce $r_{j}+1$ new points $s,x_{1},\ldots,x_{r_{j}}$ by naturally embedding the $\alpha_{i}$ into the pointwise stabilizer $\mathrm{Sym}(n+r_{j}+1)_{(\{s,x_{1},\ldots,x_{r_{j}}\})}$. If $j=1$, take
\begin{align*}
\alpha'_{1} & =(r\ s)(x_{1}\ \cdots\ x_{r_{j}})\alpha_{1}=(x_{1}\ \cdots\ x_{r_{j}})(r\ s\ \vec{\rho}_{1})\vec{\beta}_{1}, \\
\alpha'_{2} & =\alpha_{2}(x_{1}\ \cdots\ x_{r_{j}-1}\ r\ x_{r_{j}}\ s)=(x_{1}\ \cdots\ x_{r_{j}-1}\ r\ \vec{\rho}_{2}\ x_{r_{j}}\ s)\vec{\beta}_{2}, \\
\alpha'_{3} & =\alpha'_{1}\alpha'_{2}=([x_{1}\overset{\text{odds}}{\cdots}x_{r_{j}-1}]\ s\ \vec{\rho}_{3}\ x_{r_{j}}\ [x_{2}\overset{\text{evens}}{\cdots}x_{r_{j}-2}]\ r)\vec{\beta}_{3}
\end{align*}
(see \S\ref{se:notation-other} for the notation above). If $j=2$, take
\begin{align*}
\alpha'_{1} & =(x_{1}\ s\ x_{2}\ \cdots\ x_{r_{j}}\ r)\alpha_{1}=(x_{1}\ s\ x_{2}\ \cdots\ x_{r_{j}}\ \vec{\rho}_{1}\ r)\vec{\beta}_{1}, \\
\alpha'_{2} & =\alpha_{2}(x_{1}\ \cdots\ x_{r_{j}})(r\ s)=(x_{1}\ \cdots\ x_{r_{j}})(s\ r\ \vec{\rho}_{2})\vec{\beta}_{2}, \\
\alpha'_{3} & =\alpha'_{1}\alpha'_{2}=(x_{1}\ r\ [x_{2}\overset{\text{evens}}{\cdots}x_{r_{j}}]\ \vec{\rho}_{3}\ s\ [x_{3}\overset{\text{odds}}{\cdots}x_{r_{j}-1}])\vec{\beta}_{3}.
\end{align*}
If $j=3$, take
\begin{align*}
\alpha'_{1} & =(x_{1}\ s\ x_{2}\ \cdots\ x_{r_{j}}\ r)\alpha_{1}=(x_{1}\ s\ x_{2}\ \cdots\ x_{r_{j}}\ \vec{\rho}_{1}\ r)\vec{\beta}_{1}, \\
\alpha'_{2} & =\alpha_{2}(x_{1}\ \cdots\ x_{r_{j}-1}\ s\ x_{r_{j}}\ r)=(x_{1}\ \cdots\ x_{r_{j}-1}\ s\ x_{r_{j}}\ r\ \vec{\rho}_{2})\vec{\beta}_{2}, \\
\alpha'_{3} & =\alpha'_{1}\alpha'_{2}=([x_{2}\overset{\text{evens}}{\cdots}x_{r_{j}-2}]\ s\ [x_{3}\overset{\text{odds}}{\cdots}x_{r_{j}-1}]\ r)(x_{1}\ x_{r_{j}}\ \vec{\rho}_{3})\vec{\beta}_{3}.
\end{align*}
In all cases $\alpha'_{1}\alpha'_{2}=\alpha'_{3}$, and the triple of cycles containing the $\vec{\rho}_{i}$ has a common value, say $r$ for $j\in\{1,2\}$ and $x_{1}$ for $j=3$, so we can repeat the process for all $j\in J$.

It remains to add the last $8$ points. Embed again the $\alpha_{i}=(r\ \vec{\rho}_{i})\vec{\beta}_{i}$ resulting from the process above into the pointwise stabilizer $\mathrm{Sym}(n+8)_{(\{x_{1},\ldots,x_{8}\})}$, and take
\begin{align*}
\alpha'_{1} & =(x_{1}\ \cdots\ x_{7}\ r\ x_{8})\alpha_{1}=(x_{1}\ \cdots\ x_{7}\ \vec{\rho}_{1}\ r\ x_{8})\vec{\beta}_{1}, \\
\alpha'_{2} & =\alpha_{2}(x_{1}\ \cdots\ x_{6}\ r\ x_{8}\ x_{7})=(x_{1}\ \cdots\ x_{6}\ r\ \vec{\rho}_{2}\ x_{8}\ x_{7})\vec{\beta}_{2}, \\
\alpha'_{3} & =\alpha'_{1}\alpha'_{2}=(x_{1}\ x_{3}\ x_{5}\ r\ x_{7}\ \vec{\rho}_{3}\ x_{8}\ x_{2}\ x_{4}\ x_{6})\vec{\beta}_{3}.
\end{align*}
We have $\alpha'_{1}\alpha'_{2}=\alpha'_{3}$, and the $x_{i}$ satisfy the required property of Reduction~\ref{re:1}.
\end{proof}

\section{Reduction~\ref{re:3}: no short cycles}\label{se:re3}

In the third reduction we fix an ordering of the cycles (thus passing to string triples for the first time) and eliminate all the cycles of short length. We do so step by step.

First, for $C\in\mathcal{C}_{n}$ and $t\in\mathbb{N}_{>31}$ such that $C$ has a $t$-cycle, we order the cycles of $C$ in the following way: if $C$ has $k$ cycles and $\ell_{j}$ is the length of the $j$-th cycle, we set $\ell_{1}=t$ and put the others in decreasing order\footnote{There are two reasons for putting $t$ first: it allows to recover $t$ from reading the string, which makes $\theta_{3}$ injective in Lemma~\ref{le:theta3}, and it allows the specified triple of $t_{i}$-cycles to share a value, since Reduction~\ref{re:3} will require the triple to satisfy $\mathfrak{C}_{1}$. Putting the other cycles in decreasing order is an arbitrary choice, although we do need to make a choice to univocally define $\theta_{3}$.}, i.e.\ $\ell_{j}\geq\ell_{j+1}$ for $2\leq j<k$. Now let $h$ be the largest index for which $\ell_{h}>31$: it exists since $t>31$, and $\ell_{h'}>31$ if and only if $h'\leq h$ by construction. For $n':=\sum_{h'\leq h}\ell_{h'}\leq n$, we create a class string $\phi_{0}\in\Phi_{n'}$ given by
\begin{align*}
\phi_{0}\left(\frac{1}{2}\right) & =\phi_{0}\left(\ell_{1}+\frac{1}{2}\right)=\phi_{0}\left(\ell_{1}+\ell_{2}+\frac{1}{2}\right)=\ldots=\phi_{0}\left(n'+\frac{1}{2}\right)=\blacksquare
\end{align*}
and by $\phi_{0}\left(x+\frac{1}{2}\right)=\square$ for all other $x$. If we have a triple $(C_{1},C_{2},C_{3})\in\mathcal{X}_{2}(n,\delta_{2})$ and appropriate values $t_{i}$, we do so three times and create three initial strings $\phi_{i,0}$: the three values of $n'$ are possibly distinct, but for $\delta_{2}$ small and $n$ large we have the uniform bound $n'\geq\frac{99}{100}n$ by Corollary~\ref{co:classbig}\eqref{co:classbig-fewsmall}.

In the creation of the initial strings so far the short cycles have been ignored, but not eliminated in any meaningful way. For $i\in\{1,2,3\}$, let $S_{i}$ be the set of cycles of $C_{i}$ of length $\leq 31$, say intended as the multiset of lengths $\ell_{h'}$ with $h'>h$. The elimination will proceed in steps, in which we manipulate the strings and also create a first label function $\lambda:\tilde{P}_{n'}\rightarrow\{\emptyset\}\cup\mathcal{N}$ (for some value of $n'$): if $P_{31}^{<\infty}$ is the set of all finite sequences taking integer values between $1$ and $31$, we take $\mathcal{N}:=\{1,2,3\}\times P_{31}^{<\infty}$ as the set of \textit{nesting labels}.

The setup is as follows. We work with a pair of counters $(i,j)$, with $i\in\{1,2,3\}$ and $j$ uniformly bounded in some way: $j\leq\max_{i}\sum_{s\in S_{i}}s\leq\frac{1}{100}n$ will be enough by Corollary~\ref{co:classbig}\eqref{co:classbig-fewsmall}. For each $(i,j)$, at the end of the $(i,j)$-th step we have at hand three strings $\phi_{k,(i,j)}$ ($k\in\{1,2,3\}$), three sets $S_{k,(i,j)}\subseteq S_{k}$, and some $r_{(i,j)}\in\tilde{P}_{n}$ such that $\lambda$ is defined for all half-positions $\leq r_{(i,j)}$. For brevity, call $C_{k,(i,j)}$ the class given by $\mu(\phi_{k,(i,j)})\times S_{k,(i,j)}$: here $\mu$ is the natural function of \S\ref{se:notation} sending strings to classes, to which we attach one short cycle of length $s$ for each $s\in S_{k,(i,j)}$. We introduce an ordering ``$\leq$'', with $(i_{1},j_{1})\leq(i_{2},j_{2})$ if either $i_{1}<i_{2}$, or $i_{1}=i_{2}$ and $j_{1}\leq j_{2}$. Before we start, we define $S_{k,(1,0)}:=S_{k}$, $\phi_{k,(1,0)}:=\phi_{k,0}$, $r_{(1,0)}:=\frac{1}{2}$, and $\lambda\left(\frac{1}{2}\right):=\emptyset$, which in particular implies that $C_{k,(1,0)}=C_{k}$.

We ensure that the following properties are satisfied throughout the whole process:
\begin{enumerate}[(a)]
\item\label{eq:re3-s} if $(i_{1},j_{1})\leq(i_{2},j_{2})$ then $S_{k,(i_{1},j_{1})}\supseteq S_{k,(i_{2},j_{2})}$ for all $k$;
\item\label{eq:re3-r} if $(i_{1},j_{1})\leq(i_{2},j_{2})$ then $r_{(i_{1},j_{1})}\leq r_{(i_{2},j_{2})}$;
\item\label{eq:re3-t} $S_{k,(i,j)}$ has only cycles of length $\leq 31$, and an even number of cycles of even length, for all $i,j,k$;
\item\label{eq:re3-nu} $\nu_{C_{1,(i,j)},C_{2,(i,j)},C_{3,(i,j)}}$ is odd for all $i,j$;
\item\label{eq:re3-n} $\phi_{k,(i,j)}$ has length $\geq\frac{97}{100}n$ for all $i,j,k$;
\item\label{eq:re3-existsa} there is some $a_{(i,j)}\in\tilde{P}_{n}$ with $\max\{r_{(i,j)},\frac{i}{10}n+10\}\leq a_{(i,j)}\leq\frac{i+1}{10}n-70$ such that $\phi_{k,(i,j)}(a_{(i,j)}+r)=\square$ for all $k$ and all integers\footnote{We need at least $60$ consecutive $\square$ to make sure that we can remove some $T\neq\emptyset$ below: at worst, $T$ is made of a pair of $30$-cycles. Then, the extra $10$ places on either side contribute to the injectivity of the maps $\theta_{k}$ because, for every reduction $\theta_{k}:\mathcal{X}_{k-1}\rightarrow\mathcal{X}_{k}$ (which transforms $\lambda$ into some $\lambda'$) and every $a\in\tilde{P}_{n}$ that is affected in some way by that reduction, we must have $\lambda(a)=\emptyset$.} $-10\leq r\leq 70$.
\end{enumerate}
The properties hold for $(i,j)=(1,0)$. In fact, the only one that is not obvious and has not already been proved before is \eqref{eq:re3-existsa}, and if it did not hold then there would be $\geq\left\lfloor\frac{1}{810}n\right\rfloor-1$ half-positions $a$ with $\phi_{k,(1,0)}(a)=\blacksquare$ for at least one $k$, which would mean that at least one $C_{k}$ has $\geq\frac{1}{3}\left(\left\lfloor\frac{1}{810}n\right\rfloor-1\right)$ cycles, contradicting Proposition~\ref{pr:gmcycles}\eqref{pr:gmcycles-1}.

The process starts at the $(1,1)$-th step. For given $i$, whenever we reach $j$ such that $S_{i,(i,j)}=\emptyset$, we jump to the next value of the counter $i$: in other words, we set $S_{k,(i+1,0)}:=S_{k,(i,j)}$, $\phi_{k,(i+1,0)}:=\phi_{k,(i,j)}$, $r_{(i+1,0)}:=r_{(i,j)}$ and continue the process from the $(i+1,1)$-th step. By \eqref{eq:re3-s}, if $S_{k,(i,j)}=\emptyset$ then $S_{k,(i',j')}=\emptyset$ for all $(i',j')\geq(i,j)$; once we reach $S_{1,(3,j)}=S_{2,(3,j)}=S_{3,(3,j)}=\emptyset$ for some $j$, the process ends.

Finally, we explain what happens at the $(i,j)$-th step, assuming that $S_{i,(i,j-1)}\neq\emptyset$. Take the smallest $a_{(i,j-1)}$ satisfying \eqref{eq:re3-existsa}. Consider all subsets $T\subseteq S_{i,(i,j-1)}$ such that the number of even-length cycles in $T$ is even and the sum of the lengths of all cycles of $T$ is $\leq 60$: since \eqref{eq:re3-t} holds, there exists at least one such $T$. Order all such $T$ lexicographically by listing the cycle lengths\footnote{This is an arbitrary choice, although we do need to make a choice to univocally define $\theta_{3}$.}, and take the first one. Denote this $T$ by $T_{(i,j)}$; we can naturally see $T_{(i,j)}$ as an element of $P_{31}^{<\infty}$ as well. Define $t_{(i,j)}:=\sum_{t\in T_{(i,j)}}t$. Then set $\phi_{i,(i,j)}:=\phi_{i,(i,j-1)}$, and for $k\neq i$ if we have $\phi_{k,(i,j-1)}:\tilde{P}_{n'}\rightarrow\mathcal{B}$ then let
\begin{align*}
\phi_{k,(i,j)} & :\tilde{P}_{n'-t_{(i,j)}}\rightarrow\mathcal{B}, & \phi_{k,(i,j)}(r) & =\begin{cases} \phi_{k,(i,j)}(r) & (r\leq a_{(i,j-1)}), \\ \phi_{k,(i,j)}(r+t_{(i,j)}) & (r>a_{(i,j-1)}). \end{cases}
\end{align*}
In other words, we shrink by $t_{(i,j)}$ positions the length of the cycles of the two non-$i$-th strings in correspondence of the half-position $a_{(i,j-1)}$, where by \eqref{eq:re3-existsa} we know that there are enough $\square$ to do so. Correspondingly, we set $S_{i,(i,j)}:=S_{i,(i,j-1)}\setminus T_{(i,j)}$, removing the cycles of $T_{(i,j)}$ from the $i$-th class, which also account for $t_{(i,j)}$ positions in total. For $k\neq i$, we just define $S_{k,(i,j)}:=S_{k,(i,j-1)}$. Lastly, as for the labelling, we set $r_{(i,j)}:=a_{(i,j-1)}$: if $r_{(i,j)}>r_{(i,j-1)}$ then define $\lambda(r):=\emptyset$ for all $r_{(i,j-1)}<r<r_{(i,j)}$ and $\lambda(r_{(i,j)}):=(i,T_{(i,j)})$; if $r_{(i,j)}=r_{(i,j-1)}$, so that we already have $\lambda(r_{(i,j)})=(i,T')$ for a previous $T'$, redefine it to be $(i,T'\oplus T_{(i,j)})$ (``$\oplus$'' refers to concatenating the two finite sequences).

A visual representation of the $(i,j)$-th step is given below, based on \eqref{eq:stringex}. For visual simplicity, we assume $i=1$, $r_{(i,j-1)}=a_{(i,j-1)}-11$, and $n'+\frac{1}{2}=a_{(i,j-1)}+72$ for all three strings. Call $a:=a_{(i,j-1)}$ and $r:=r_{(i,j-1)}$, take $T_{(i,j)}=\{\text{$4$-cycle, $4$-cycle}\}$, and depict the strings as
\begin{align*}
\begin{aligned}
\overbrace{\hphantom{mmmmmmmmmmmmmmmmmm}}^{\text{only $\square$ in all strings}} \hphantom{mmmmmm} \\
\begin{aligned}
\tilde{P}_{n'}: \,\cdots & & {\scriptstyle r-1 } & & {\scriptstyle r } & & {\scriptstyle a-10 } & & \cdots & & {\scriptstyle a } & & \cdots & & {\scriptstyle a+9 } & & \cdots & & {\scriptstyle a+70 } & & {\scriptstyle a+71 } & & {\scriptstyle a+72 } \\
\phi_{1,(1,j-1)}: \,\cdots & & \blacksquare & & \square & & \square & & \square & & \square & & \square & & \square & & \square & & \square & & \square & & \blacksquare \\
\phi_{2,(1,j-1)}: \,\cdots & & \square & & \square & & \square & & \square & & \square & & \square & & \square & & \square & & \square & & \square & & \blacksquare \\
\phi_{3,(1,j-1)}: \,\cdots & & \blacksquare & & \square & & \square & & \square & & \square & & \square & & \square & & \square & & \square & & \blacksquare & & \blacksquare \\
\lambda: \,\cdots & & (\lambda) & & (\lambda) & &  & &  & &  & & [\ \ ]\!  & &  & &  & &  & &  & &  \\
\end{aligned}
\end{aligned}
\end{align*}
where ``$(\lambda)$'' indicates values of $\lambda$ that were already defined in previous steps, and ``$[\ \ ]$'' points out which half-positions will disappear at this step in the second and third string: there are $8$ of them since $t_{(i,j)}=8$. Then we obtain
\begin{align*}
\begin{aligned}
\tilde{P}_{n'}: \,\cdots & & {\scriptstyle r-1 } & & {\scriptstyle r } & & {\scriptstyle a-10 } & & \cdots & & {\scriptstyle a } & & {\scriptstyle a+1 } & & \cdots & & {\scriptstyle a+62 } & & {\scriptstyle a+63 } & & {\scriptstyle a+64 } & & \cdots & & {\scriptstyle a+72 } \\
\phi_{1,(1,j)}: \,\cdots & & \blacksquare & & \square & & \square & & \square & & \square & & \square & & \square & & \square & & \square & & \square & & \cdots & & \blacksquare \\
\phi_{2,(1,j)}: \,\cdots & & \square & & \square & & \square & & \square & & \square & & \square & & \square & & \square & & \square & & \blacksquare & & & & \\
\phi_{3,(1,j)}: \,\cdots & & \blacksquare & & \square & & \square & & \square & & \square & & \square & & \square & & \square & & \blacksquare & & \blacksquare & & & & \\
\lambda: \,\cdots & & (\lambda) & & (\lambda) & & \emptyset & & \emptyset & & x & &  & &  & &  & &  & &  \\
\end{aligned}
\end{align*}
with $x=(1,(4,4))$. If we started with $a_{(i,j-1)}=r_{(i,j-1)}$ and for instance we already had $\lambda(a_{(i,j-1)})=(1,(30,30,1))$ from the previous step, then $x=(1,(30,30,1,4,4))$.

By definition we have $S_{i,(i,j)}\subsetneq S_{i,(i,j-1)}$, so as discussed before the process terminates, as long as the properties \eqref{eq:re3-s}--\eqref{eq:re3-existsa} are preserved throughout. Given our definitions we have \eqref{eq:re3-s} immediately, \eqref{eq:re3-r} by the fact that \eqref{eq:re3-existsa} holds at the previous step, and \eqref{eq:re3-t} by the choice of $T_{(i,j)}$. Whether the $i$-th class is in the corresponding alternating group is not changed by the removal of $T_{(i,j)}$, whereas the other two can only change together, since we remove the same number of points from a single cycle of both classes. Therefore, \eqref{eq:re3-nu} at the $(i,j)$-th step follows from \eqref{eq:re3-nu} at the previous step. Corollary~\ref{co:classbig}\eqref{co:classbig-fewsmall} yields $\sum_{i}\sum_{s\in S_{i}}s\leq\frac{3}{100}n$, giving \eqref{eq:re3-n} at all steps. Combining \eqref{eq:re3-n} with Proposition~\ref{pr:gmcycles}\eqref{pr:gmcycles-1} and Corollary~\ref{co:classbig}\eqref{co:classbig-cut}, and arguing as we did for the $(1,0)$-th step, we obtain \eqref{eq:re3-existsa} at all steps. Thus, the process works as intended.

Denote by $\phi_{k}$ the strings $\phi_{k,(3,j)}$ obtained at the last step of the process, and by $C'_{k}$ the corresponding classes $\mu(\phi_{k})$. The final $n'$ is the same for all three strings, namely
\begin{equation}\label{eq:np3}
n':=n-\sum_{i=1}^{3}\sum_{j=1}^{31}jn_{j}(C_{i})\geq\frac{97n}{100}
\end{equation}
by \eqref{eq:re3-n}. Set $\lambda(r):=\emptyset$ for all the remaining values $r_{(3,j)}<r\leq n'+\frac{1}{2}$.

Now that we have $\phi_{i},C'_{i},\lambda$, we can define $\theta_{3}$. Recall $\mathfrak{P}(\text{labels})$ from \eqref{eq:prop-labels}; we shall start defining some more properties. If $c^{\blacksquare}(a)$ is as in \eqref{eq:functbcount}, write
\begin{align}
\mathfrak{P}(\blacksquare,k): & \begin{cases}&\text{for all $i$, for all $a$ s.t.\ $\phi_{i}(a)=\blacksquare$,} \\ &\text{for all $1\leq|r|\leq k$ we have $\phi_{i}(a+r)=\square$,}\end{cases} \label{eq:prop-cycle} \\
\mathfrak{P}(\mathcal{N},k): & \begin{cases}&\text{for all $a$ s.t.\ $\lambda(a)\in\mathcal{N}$,} \\ &\text{for all $0\leq|r|\leq k$ we have $c^{\blacksquare}(a+r)=0$.}\end{cases} \label{eq:prop-n}
\end{align}
The two properties above represent respectively having no short cycles for any string and having nesting labels at large distance from any $\blacksquare$.

\begin{lemma}\label{le:theta3}
Let
\begin{align*}
\mathcal{X}_{3}(m,\delta):= \ & \left\{\, (\phi_{1},\phi_{2},\phi_{3},\lambda)\in\Phi_{m}^{\times 3}\times\Lambda_{m,\{\emptyset\}\cup\mathcal{N}}\ \left| \  \right.\right. \\
 & \left. \text{if $C'_{i}:=\mu(\phi_{i})$ then $C'_{i}\in\mathcal{C}_{m}(\delta)$, $\nu_{C'_{1},C'_{2},C'_{3}}$ odd;} \right. \\
 & \left. \text{$\mathfrak{P}(\mathrm{labels})$, $\mathfrak{P}(\blacksquare,31)$, $\mathfrak{P}(\mathcal{N},10)$} \,\right\},
\end{align*}
using the definitions in \eqref{eq:cdelta}--\eqref{eq:nu}--\eqref{eq:prop-labels}--\eqref{eq:prop-cycle}--\eqref{eq:prop-n}, and let
\begin{align*}
\theta_{3} & :\mathcal{X}_{2}(n,\delta_{2})\rightarrow\!\!\!\!\bigcup_{n'=\left\lceil\frac{97n}{100}\right\rceil}^{n}\!\!\!\!\mathcal{X}_{3}(n',2\delta_{2}), & \theta_{3}(C_{1},C_{2},C_{3},t_{1},t_{2},t_{3})=(\phi_{1},\phi_{2},\phi_{3},\lambda),
\end{align*}
following the construction described above.

Then, for all $\delta_{2}>0$ small enough and all $n$ large enough depending on $\delta_{2}$, $\theta_{3}$ is a well-defined function (i.e., for any sextuple in $\mathcal{X}_{2}(n,\delta_{2})$, its image is a uniquely constructed element contained in one of the $\mathcal{X}_{3}(n',2\delta_{2})$ in the union above) and is injective.
\end{lemma}

\begin{proof}
For $\delta_{2}$ small and $n$ large, the process described above makes sense, and the image of a sextuple of $\mathcal{X}_{2}(n,\delta_{2})$ through $\theta_{3}$ is a unique string triple in $\mathfrak{S}$.

The resulting strings have length bounded from below by \eqref{eq:np3}. Each $C'_{i}$ can be obtained by shortening or eliminating cycles from $C_{i}$, so by \eqref{eq:np3} and Corollary~\ref{co:classbig}\eqref{co:classbig-cut} if $C_{i}\in\mathcal{C}_{n}(\delta_{2})$ then also $C'_{i}\in\mathcal{C}_{n'}(2\delta_{2})$. The process preserves the fact that $\nu$ is odd at every step, so the final $\nu_{C'_{1},C'_{2},C'_{3}}$ is odd, and we picked the values $r_{(i,j)}$ inside disjoint intervals for different indices $i$, so $\lambda(\tilde{P}_{n'})\subseteq\{\emptyset\}\cup\mathcal{N}$. The construction yields $\mathfrak{P}(\mathrm{labels})$ (since $c^{\blacksquare}(a)=0$ whenever $\lambda(a)\in\mathcal{N}$), $\mathfrak{P}(\blacksquare,31)$, and $\mathfrak{P}(\mathcal{N},10)$. Hence, the image of $\theta_{3}$ is indeed contained in the union of the $\mathcal{X}_{3}(n',2\delta_{2})$.

To show injectivity, we just need to observe that the resulting string triple can be unravelled to recover the original sextuple. Starting with $S_{1}=S_{2}=S_{3}=\emptyset$, we read off each $\lambda(a)\neq\emptyset$ starting from the largest such $a\in\tilde{P}_{n'}$, and if $\lambda(a)=(i_{1},T)$ we add $T$ to $S_{i_{1}}$ and lengthen $\phi_{i_{2}},\phi_{i_{3}}$ by inserting $\ell$ occurrences of the value $\square$ between $a$ and $a+1$ where $\ell=\sum_{s\in T}s$. Since we start from the largest $a$, inserting values does not disrupt the labelling $\lambda$ of the smaller half-positions $a$. After exhausting all $a$, we have recovered the original $\phi_{i,(1,0)}$ and $S_{i,(1,0)}$: thus,
\begin{align*}
C_{i} & =\mu(\phi_{i,(1,0)})\times S_{i,(1,0)}, & t_{i} & =\min\left\{r-\frac{1}{2}\,\left|\,r\in\tilde{P}_{n}\setminus\left\{\frac{1}{2}\right\},\phi_{i,(1,0)}(r)=\blacksquare\right.\right\},
\end{align*}
proving that $\theta_{3}$ is injective.
\end{proof}

\begin{reduction}\label{re:3}
Find, for any $(\phi_{1},\phi_{2},\phi_{3},\lambda)\in\mathcal{X}_{3}(n,\delta_{3})$, a solution septuple $(\eta_{1},\eta_{2},\eta_{3},\phi_{1},\phi_{2},\phi_{3},\lambda)$ such that $\alpha_{1}\alpha_{2}=\alpha_{3}$ for $\alpha_{i}=\mu(\eta_{i},\phi_{i})$. Furthermore, the solution must satisfy the following property: $\mathfrak{C}_{1}$ holds for any triple of cycles sharing $\geq 2$ common positions.
\end{reduction}

\begin{lemma}\label{le:3to2}
If there is a solution for Reduction~\ref{re:3}, for all $\delta_{3}>0$ small enough and all $n$ large enough depending on $\delta_{3}$, then there is a solution for Reduction~\ref{re:2}, for all $\delta_{2}>0$ small enough and all $n$ large enough depending on $\delta_{2}$.
\end{lemma}

\begin{proof}
By Lemma~\ref{le:theta3}, every sextuple $(C_{1},C_{2},C_{3},t_{1},t_{2},t_{3})$ of $\mathcal{X}_{2}:=\mathcal{X}_{2}(n,\delta_{2})$ is the unique preimage of some $(\phi_{1},\phi_{2},\phi_{3},\lambda)\in\mathcal{X}_{3}:=\mathcal{X}_{3}(n',2\delta_{2})$ for some $\frac{97}{100}n\leq n'\leq n$ via $\theta_{3}$. Let $(\eta_{1},\eta_{2},\eta_{3},\phi_{1},\phi_{2},\phi_{3},\lambda)$ be a solution for Reduction~\ref{re:3}. We need to produce $\alpha'_{i}\in C_{i}$ such that $\alpha'_{1}\alpha'_{2}=\alpha'_{3}$.

We work on each $a\in\tilde{P}_{n'}$ with $\lambda(a)\in\mathcal{N}$, starting with the rightmost one and moving left. To walk back the procedure, at the start of each step we assume that we have the following objects at hand:
\begin{itemize}
\item three element strings $(\eta_{i},\phi_{i})$, possibly of different lengths $n_{i}$,
\item a label function $\lambda:\tilde{P}_{m}\rightarrow\{\emptyset\}\cup\mathcal{N}$ for some $m\leq\min\{n_{i}|i\in\{1,2,3\}\}$ with $\lambda(m)\in\mathcal{N}$, and
\item\label{eq:3to2-steps} three sets $S_{i}$ of cyclic permutations, where each $\sigma\in S_{i}$ has values disjoint from those of the other elements of $S_{i}$ and from those of $\mu(\eta_{i},\phi_{i})$,
\end{itemize}
and the objects above satisfy the following properties:
\begin{itemize}
\item the permutations $\alpha_{i}:=\mu(\eta_{i},\phi_{i})\times\prod_{\sigma\in S_{i}}\sigma$ belong to the same $\mathrm{Sym}(n)$,
\item the equality $\alpha_{1}\alpha_{2}=\alpha_{3}$ holds, and
\item for every $a\in\tilde{P}_{m}$ with $\lambda(a)\in\mathcal{N}$, the property $\mathfrak{C}_{1}$ holds for the three cycles $\gamma_{i}$ of $(\eta_{i},\phi_{i})$ containing $a$ (by $\mathfrak{P}(\mathrm{labels})$ we have $c^{\blacksquare}(a)=0$, so there are cycles ``containing'' $a$).
\end{itemize}
The initial solution for Reduction~\ref{re:3} satisfies the above, with $S_{i}=\emptyset$ and $m:=\max\{a\,|\,\lambda(a)\in\mathcal{N}\}$: notably, $\mathfrak{P}(\mathcal{N},10)$ in $\mathcal{X}_{3}$ imples that any triple of cycles $\gamma_{i}$ containing a label in $\mathcal{N}$ must share $\geq 2$ common positions, so $\mathfrak{C}_{1}$ holds wherever necessary.

For $j\in\{1,2,3\}$, let $\gamma_{j}$ be the cycle of $(\eta_{j},\phi_{j})$ containing $m$: by $\mathfrak{C}_{1}$ there is a value $r$ common to the $\gamma_{j}$, so up to cycling around we can assume that $\eta_{j}\left(m-\frac{1}{2}\right)=r$ for all $j$, and that $\gamma_{j}$ is written as ``$\blacksquare \ \vec{\rho}_{j1} \ r \ \vec{\rho}_{j2} \ \blacksquare$''. Let $\alpha_{j}=\vec{\beta}_{j}(\vec{\rho}_{j1} \ r \ \vec{\rho}_{j2})$.

The label $\lambda(m)$ is of the form $(i,T)\in\{1,2,3\}\times P_{31}^{<\infty}$, and $T$ has an even number of even-length cycles. From $T$, extract either one $d$-cycle with $d$ odd or an $e$-cycle and an $f$-cycle with $e,f$ even; in either case, eliminate those cycles from $T$ and relabel $m$ accordingly. Suppose first that we extract a $d$-cycle, and introduce $d$ new points $x_{1},\ldots,x_{d}$ by embedding each $\alpha_{j}$ into the pointwise stabilizer $\mathrm{Sym}(n_{j}+d)_{(\{x_{1},\ldots,x_{d}\})}$. If $i=1$, take $\sigma=(x_{1}\ x_{2}\ \cdots\ x_{d})$ and
\begin{align*}
\alpha'_{1} & =(x_{1}\ x_{2}\ \cdots\ x_{d})\alpha_{1}=\vec{\beta}_{1}(\vec{\rho}_{11}\ r\ \vec{\rho}_{12})\sigma, \\
\alpha'_{2} & =\alpha_{2}(r\ x_{1}\ x_{2}\ \cdots\ x_{d})=\vec{\beta}_{2}(\vec{\rho}_{21}\ x_{1}\ x_{2}\ \cdots\ x_{d}\ r\ \vec{\rho}_{22}), \\
\alpha'_{3} & =\alpha'_{1}\alpha'_{2}=\vec{\beta}_{3}(\vec{\rho}_{31}\ [x_{1}\overset{\text{odds}}{\cdots}x_{d}]\ [x_{2}\overset{\text{evens}}{\cdots}x_{d-1}]\ r\ \vec{\rho}_{32}).
\end{align*}
If $i=2$, take $\sigma=(x_{1}\ x_{2}\ \cdots\ x_{d})$ and
\begin{align*}
\alpha'_{1} & =(r\ x_{1}\ x_{2}\ \cdots\ x_{d})\alpha_{1}=\vec{\beta}_{1}(\vec{\rho}_{11}\ r\ x_{1}\ x_{2}\ \cdots\ x_{d}\ \vec{\rho}_{12}), \\
\alpha'_{2} & =\alpha_{2}(x_{1}\ x_{2}\ \cdots\ x_{d})=\vec{\beta}_{2}(\vec{\rho}_{21}\ r\ \vec{\rho}_{22})\sigma, \\
\alpha'_{3} & =\alpha'_{1}\alpha'_{2}=\vec{\beta}_{3}(\vec{\rho}_{31}\ r\ [x_{2}\overset{\text{evens}}{\cdots}x_{d-1}]\ [x_{1}\overset{\text{odds}}{\cdots}x_{d}]\ \vec{\rho}_{32}).
\end{align*}
If $i=3$, take $\sigma=([x_{1}\overset{\text{odds}}{\cdots}x_{d}]\ [x_{2}\overset{\text{evens}}{\cdots}x_{d-1}])$ and
\begin{align*}
\alpha'_{1} & =\alpha_{1}(r\ x_{1}\ x_{2}\ \cdots\ x_{d})=\vec{\beta}_{1}(\vec{\rho}_{11}\ x_{1}\ x_{2}\ \cdots\ x_{d}\ r\ \vec{\rho}_{12}), \\
\alpha'_{2} & =(r\ x_{2}\ \cdots\ x_{d}\ x_{1})\alpha_{2}=\vec{\beta}_{2}(\vec{\rho}_{21}\ r\ x_{2}\ \cdots\ x_{d}\ x_{1}\ \vec{\rho}_{22}), \\
\alpha'_{3} & =\alpha'_{1}\alpha'_{2}=\vec{\beta}_{3}(\vec{\rho}_{31}\ r\ \vec{\rho}_{32})\sigma.
\end{align*}
Suppose instead that we extract an $e$-cycle and an $f$-cycle, introduce $e+f$ new points $y_{1},\ldots,y_{e},z_{1},\ldots,z_{f}$, and embed $\alpha_{j}$ into the pointwise stabilizer $\mathrm{Sym}(n_{j}+e+f)_{(\{y_{1},\ldots,y_{e},z_{1},\ldots,z_{f}\})}$. If $i=1$, take $\sigma=(y_{1}\ \cdots\ y_{e})(z_{1}\ \cdots\ z_{f})$ and
\begin{align*}
\alpha'_{1} & =(y_{1}\ \cdots\ y_{e})(z_{1}\ \cdots\ z_{f})\alpha_{1}=\vec{\beta}_{1}(\vec{\rho}_{11}\ r\ \vec{\rho}_{12})\sigma, \\
\alpha'_{2} & =\alpha_{2}(r\ z_{2}\ y_{2}\ \cdots\ y_{e}\ z_{1}\ y_{1}\ z_{3}\ \cdots\ z_{f}) \\
 & =\vec{\beta}_{2}(\vec{\rho}_{21}\ z_{2}\ y_{2}\ \cdots\ y_{e}\ z_{1}\ y_{1}\ z_{3}\ \cdots\ z_{f}\ r\ \vec{\rho}_{22}), \\
\alpha'_{3} & =\alpha'_{1}\alpha'_{2}=\vec{\beta}_{3}(\vec{\rho}_{31}\ [z_{2}\overset{\text{evens}}{\cdots}z_{f}]\ [y_{1}\overset{\text{odds}}{\cdots}y_{e-1}]\ z_{1}\ [y_{2}\overset{\text{evens}}{\cdots}y_{e}]\ [z_{3}\overset{\text{odds}}{\cdots}z_{f-1}]\ r\ \vec{\rho}_{32}).
\end{align*}
 If $i=2$, take $\sigma=(y_{1}\ \cdots\ y_{e})(z_{1}\ \cdots\ z_{f})$ and
\begin{align*}
\alpha'_{1} & =(r\ z_{2}\ y_{2}\ \cdots\ y_{e}\ z_{1}\ y_{1}\ z_{3}\ \cdots\ z_{f})\alpha_{1} \\
 & =\vec{\beta}_{1}(\vec{\rho}_{11}\ r\ z_{2}\ y_{2}\ \cdots\ y_{e}\ z_{1}\ y_{1}\ z_{3}\ \cdots\ z_{f}\ \vec{\rho}_{12}), \\
\alpha'_{2} & =\alpha_{2}(y_{1}\ \cdots\ y_{e})(z_{1}\ \cdots\ z_{f})=\vec{\beta}_{2}(\vec{\rho}_{21}\ r\ \vec{\rho}_{22})\sigma, \\
\alpha'_{3} & =\alpha'_{1}\alpha'_{2}\!=\!\vec{\beta}_{3}(\vec{\rho}_{31}\ r\ [z_{3}\overset{\text{odds}}{\cdots}z_{f-1}]\ z_{1}\ [y_{2}\overset{\text{evens}}{\cdots}y_{e}]\ z_{2}\ [y_{3}\overset{\text{odds}}{\cdots}y_{e-1}]\ y_{1}\ [z_{4}\overset{\text{evens}}{\cdots}z_{f}]\ \vec{\rho}_{32}).
\end{align*}
Finally, if $i=3$, say for instance that $e\leq f$. Call $g=(e+f)/2$, and define the following strings for brevity:
\begin{align*}
\tau_{j} & =\begin{cases}y_{2+\frac{j+1}{2}} & \text{$j$ odd,} \\ z_{1+\frac{j}{2}} & \text{$j$ even,}\end{cases} & \vec{\tau} & =(\tau_{j})_{j=1}^{2(e-2)}=(y_{3},z_{2},y_{4},z_{3},\ldots,y_{e},z_{e-1}), \\
\upsilon_{j} & =\begin{cases}z_{g+\frac{j+1}{2}} & \text{$j$ odd,} \\ z_{e+\frac{j}{2}} & \text{$j$ even,}\end{cases} & \vec{\upsilon} & =(\upsilon_{j})_{j=1}^{f-e}=(z_{g+1},z_{e+1},z_{g+2},z_{e+2},\ldots,z_{f},z_{g}).
\end{align*}
Then take $\sigma=(y_{1}\ \cdots\ y_{e})(z_{1}\ \cdots\ z_{f})$ and
\begin{align*}
\alpha'_{1} & =\alpha_{1}(r\ y_{1}\ z_{1}\ \vec{\tau}\ y_{2}\ z_{e}\ \vec{\upsilon})=\vec{\beta}_{1}(\vec{\rho}_{11}\ y_{1}\ z_{1}\ \vec{\tau}\ y_{2}\ z_{e}\ \vec{\upsilon}\ r\ \vec{\rho}_{12}), \\
\alpha'_{2} & =(r\ \vec{\upsilon}\ z_{1}\ y_{2}\ z_{e}\ \vec{\tau}\ y_{1})\alpha_{2}=\vec{\beta}_{2}(\vec{\rho}_{21}\ r\ \vec{\upsilon}\ z_{1}\ y_{2}\ z_{e}\ \vec{\tau}\ y_{1}\ \vec{\rho}_{22}), \\
\alpha'_{3} & =\alpha'_{1}\alpha'_{2}=\vec{\beta}_{3}(\vec{\rho}_{31}\ r\ \vec{\rho}_{32})\sigma.
\end{align*}
In all six cases above, add $\sigma$ to $S_{i}$ and lenghten the two element strings $(\eta_{j},\phi_{j})$ with $j\neq i$ by lengthening the cycles $\gamma_{j\neq i}$ as described above. If after this operation we have $T=\emptyset$, rename $m$ to be the next largest half-position with a label in $\mathcal{N}$.

Observe that, for all $j\in\{1,2,3\}$, every cycle that was already in $S_{j}$ and every position and half-position $b\in P_{n_{j}}\cup\tilde{P}_{n_{j}}$ with $b\leq m-1$ of every string remains unchanged. All the properties in our assumption at the start of the step are valid for the new choice of strings, labels, $m$, and $S_{j}$: in particular, if after the previous step we still have $T\neq\emptyset$, then the new cycles $\gamma'_{j}$ still satisfy $\mathfrak{C}_{1}$ using the same value $r$ as before. Repeat the above for a given $m$ until $T=\emptyset$, and for all $m$ until there are no more labels in $\mathcal{N}$. At the end, call $\alpha'_{i}$ the final $\mu(\eta_{i},\phi_{i})\times\prod_{\sigma\in S_{i}}\sigma$, denote by $C_{i}$ its conjugacy class, and let $t_{i}$ be the length of the first cycle in $(\eta_{i},\phi_{i})$, i.e.\ $t_{i}:=\min\left\{a>0\,\left|\,\phi_{i}\left(a+\frac{1}{2}\right)=\blacksquare\right.\right\}$.

By what we said above, we have $\alpha'_{1}\alpha'_{2}=\alpha'_{3}$. Furthermore, the triple of cycles at the beginning of the $\phi_{i}$, which are of length $t_{1},t_{2},t_{3}$ at the end of the procedure, must satisfy $\mathfrak{C}_{1}$: in fact, at each step $\mathfrak{C}_{1}$ is preserved for all triples whose cycles either contain $m$ or sit at the left of $m$ (or both happen in distinct strings).
\end{proof}

Starting from Reduction~\ref{re:3}, we are working in some $\mathcal{X}_{k}\subseteq\mathfrak{S}$, i.e.\ with string triples rather than triples of classes.

\section{Reduction~\ref{re:4}: no ledges}

\parbox{0.65\textwidth}{
In the fourth reduction we get rid of a technical annoyance, namely of all the half-positions $a\in\tilde{P}_{n}$ where we have $\phi_{i_{1}}(a)=\phi_{i_{2}}(a)=\phi_{i_{3}}(a+\varepsilon)=\blacksquare$ for some choice of $i_{1},i_{2},i_{3}$ and of sign $\varepsilon\in\{\pm 1\}$. We call \textit{ledge} such an $a$. An example of a ledge at $a$ is given on the right, with $i_{3}=1$ and $\varepsilon=1$.
}
\hfill
\parbox{0.3\textwidth}{
\begin{align*}
\begin{aligned}
{\scriptstyle a-2} & & {\scriptstyle a-1} & & {\scriptstyle a} & & {\scriptstyle a+1} & & {\scriptstyle a+2} & \\
\square & & \square & & \square & & \blacksquare & & \square & \\
\square & & \square & & \blacksquare & & \square & & \square & \\
\square & & \square & & \blacksquare & & \square & & \square &
\end{aligned}
\end{align*}
}

Let $(\phi_{1},\phi_{2},\phi_{3},\lambda)\in\mathcal{X}_{3}(n,\delta_{3})$, and let $a\in\tilde{P}_{n}$ be a ledge for $\phi_{1},\phi_{2},\phi_{3}$. By $\mathfrak{P}(\blacksquare,31)$ in $\mathcal{X}_{3}$ we must have $c^{\blacksquare}(b)=0$ for $b=a\pm 2$ and $b=a-\varepsilon$, and by $\mathfrak{P}(\mathrm{labels})$ and $\mathfrak{P}(\mathcal{N},10)$ in $\mathcal{X}_{3}$ we must have $\lambda(b)=\emptyset$ for all $b$ with $|b-a|\leq 2$. Let $\mathcal{L}=\{1,2,3\}\times\{\pm 1\}$ be the set of \textit{ledge labels}. Define new strings $\phi'_{i}\in\Phi_{n-2}$ and a new label function $\lambda':\tilde{P}_{n-2}\rightarrow\{\emptyset\}\cup\mathcal{N}\cup\mathcal{L}$ by
\begin{align*}
\phi'_{i}(b) & =\begin{cases} \phi_{i}(b) & (b\leq a-2), \\ \blacksquare & (b=a-1), \\ \phi_{i}(b+2) & (b\geq a), \end{cases} \\
\lambda'(b) & =\begin{cases} \lambda(b)\in\{\emptyset\}\cup\mathcal{N} & (b\leq a-2), \\ (i,\varepsilon)\in\mathcal{L} & (b=a-1,\ \phi_{i}(a+\varepsilon)=\blacksquare), \\ \lambda(b+2)\in\{\emptyset\}\cup\mathcal{N} & (b\geq a). \end{cases}
\end{align*}
Visually, we represent the process below.
\begin{align}
\overbrace{\hphantom{mmmmmm}}^{\text{will shrink into one}} \hphantom{nmmmmmmmmmmmmmmmm} \nonumber \\
\begin{aligned}
\tilde{P}_{n}: & & {\scriptstyle a-2} & & {\scriptstyle a-1} & & {\scriptstyle a} & & {\scriptstyle a+1} & & {\scriptstyle a+2} & & & & \tilde{P}_{n-2}: & & {\scriptstyle a-2}\!\! & & {\scriptstyle a-1}\!\! & & {\scriptstyle a} & \\
\phi_{1}: & & \square & & \square & & \square & & \blacksquare & & \square & & & & \phi'_{1}: & & \square & & \blacksquare & & \square & \\
\phi_{2}: & & \square & & \square & & \blacksquare & & \square & & \square & & \longrightarrow & & \phi'_{2}: & & \square & & \blacksquare & & \square & \\
\phi_{3}: & & \square & & \square & & \blacksquare & & \square & & \square & & & & \phi'_{3}: & & \square & & \blacksquare & & \square & \\
\lambda: & & \emptyset & & \emptyset & & \emptyset & & \emptyset & & \emptyset & & & & \lambda': & & \emptyset & & (1,1)\!\!\!\! & & \emptyset &
\end{aligned} \label{eq:graphiclstep}
\end{align}
By construction the new string triple has one fewer ledge, so we repeat the process until there are no ledges\footnote{Ledges create problems when performing Reduction~\ref{re:6}. In fact, they create a situation in which three cycles of length $l,1,1$ with $l>1$ share a common position, which then makes it impossible for that triple to satisfy $\mathfrak{C}_{1}$.}. Let $(\phi'_{1},\phi'_{2},\phi'_{3},\lambda')$ be the final string triple with $\phi'_{i}\in\Phi_{n'}$, and set $C'_{i}:=\mu(\phi'_{i})$.

Now we define $\theta_{4}$. Write
\begin{align}
\mathfrak{P}(\mathcal{L},k): & \begin{cases}&\text{for all $a$, if $\lambda(a)\in\mathcal{L}$ then $c^{\blacksquare}(a+r)=0$ for all $1\leq|r|\leq k$,} \end{cases} \label{eq:prop-lk} \\
\mathfrak{P}(\mathcal{L}): & \begin{cases}&\text{for all $a$, if $c^{\blacksquare}(a)=2$ then $c^{\blacksquare}(a+r)=0$ for $r\in\{\pm 1\}$.} \label{eq:prop-l} \end{cases}
\end{align}

\begin{lemma}\label{le:theta4}
Let
\begin{align*}
\mathcal{X}_{4}(m,\delta):= \ & \left\{\, (\phi_{1},\phi_{2},\phi_{3},\lambda)\in\Phi_{m}^{\times 3}\times\Lambda_{m,\{\emptyset\}\cup\mathcal{N}\cup\mathcal{L}}\ \right| \\
 & \left. \text{if $C'_{i}:=\mu(\phi_{i})$ then $C'_{i}\in\mathcal{C}_{m}(\delta)$, $\nu_{C'_{1},C'_{2},C'_{3}}$ odd;} \right. \\
 & \left. \vphantom{\Phi_{m}^{\times 3}} \text{$\mathfrak{P}(\mathrm{labels})$, $\mathfrak{P}(\blacksquare,27)$, $\mathfrak{P}(\mathcal{N},8)$, $\mathfrak{P}(\mathcal{L},27)$, $\mathfrak{P}(\mathcal{L})$} \,\right\},
\end{align*}
using \eqref{eq:cdelta}--\eqref{eq:nu}--\eqref{eq:prop-labels}--\eqref{eq:prop-cycle}--\eqref{eq:prop-n}--\eqref{eq:prop-lk}--\eqref{eq:prop-l}, and let
\begin{align*}
\theta_{4} & :\mathcal{X}_{3}(n,\delta_{3})\rightarrow\!\!\!\!\bigcup_{n'=\left\lceil\frac{29n}{31}\right\rceil}^{n}\!\!\!\!\mathcal{X}_{4}(n',2\delta_{3}), & \theta_{4}(\phi_{1},\phi_{2},\phi_{3},\lambda)=(\phi'_{1},\phi'_{2},\phi'_{3},\lambda'),
\end{align*}
following the construction described above.

Then $\theta_{4}$ is a well-defined injective function for all $\delta_{3}>0$ small enough and all $n$ large enough depending on $\delta_{3}$.
\end{lemma}

\begin{proof}
By $\mathfrak{P}(\blacksquare,31)$ in $\mathcal{X}_{3}$, there are $\leq\frac{1}{31}n$ ledges for $\phi_{1},\phi_{2},\phi_{3}$. Therefore the final $n'$ satisfies $n'\geq\frac{29}{31}n$,
and by Corollary~\ref{co:classbig}\eqref{co:classbig-cut} we have $C'_{i}\in\mathcal{C}_{n'}(2\delta_{3})$. At each step of the process, say for $\lambda(a)=(i_{1},\varepsilon)$, we replace one cycle of $C_{i_{1}}$, say of length $\ell$, with a cycle of length $\ell-2\geq 2$, and we replace two cycles in both $C_{i_{2}},C_{i_{3}}$, say of length $\ell_{1}$ and $\ell_{2}$ (not necessarily the same lengths in both classes), with two cycles of length $\ell_{1}-1\geq 2$ and $\ell_{2}-1\geq 2$. This implies $\nu_{C_{1},C_{2},C_{3}}=\nu_{C'_{1},C'_{2},C'_{3}}$ at each step, so the final $\nu_{C'_{1},C'_{2},C'_{3}}$ is odd.

We still have $\mathfrak{P}(\mathrm{labels})$. Informally speaking, the construction shortens the distance between consecutive values $\blacksquare$ by $\leq 2$ positions from the left and $\leq 2$ positions from the right. Therefore $\mathfrak{P}(\blacksquare,31)$ and $\mathfrak{P}(\mathcal{N},10)$ in $\mathcal{X}_{3}$ yield $\mathfrak{P}(\blacksquare,27)$, $\mathfrak{P}(\mathcal{N},8)$, and $\mathfrak{P}(\mathcal{L},27)$ in $\mathcal{X}_{4}$. Finally, if $c^{\blacksquare}(a)=2$ for $\phi'_{1},\phi'_{2},\phi'_{3}$, say we have $\phi'_{i_{1}}(a)=\phi'_{i_{2}}(a)=\blacksquare$ and $\phi'_{i_{3}}(a)=\square$. Then $\phi'_{i_{1}}(a\pm 1)=\phi'_{i_{2}}(a\pm 1)=\square$ by $\mathfrak{P}(\blacksquare,27)$, and $\phi'_{i_{3}}(a\pm 1)=\square$ because there are no ledges, thus giving us $\mathfrak{P}(\mathcal{L})$.

One look at \eqref{eq:graphiclstep} should make injectivity obvious. As in Lemma~\ref{le:theta4}, we unravel the string triple from right to left at all $a$ with $\lambda'(a)\in\mathcal{L}$ and recover the original $\phi_{i}$, thanks to the fact that $\lambda'(a)$ encodes all the necessary information; as for $\lambda$, the affected half-positions must have value $\emptyset$ (as $\mathfrak{P}(\mathrm{labels})$ and $\mathfrak{P}(\mathcal{N},10)$ hold in $\mathcal{X}_{3}$).
\end{proof}

\begin{reduction}\label{re:4}
Find, for any $(\phi_{1},\phi_{2},\phi_{3},\lambda)\in\mathcal{X}_{4}(n,\delta_{4})$, a solution septuple $(\eta_{1},\eta_{2},\eta_{3},\phi_{1},\phi_{2},\phi_{3},\lambda)$ such that $\alpha_{1}\alpha_{2}=\alpha_{3}$ for $\alpha_{i}=\mu(\eta_{i},\phi_{i})$. Furthermore, the solution must be aligned.
\end{reduction}

From Reduction~\ref{re:4} onward, we always ask for aligned solutions. The ledges that still existed up to Reduction~\ref{re:3} prevented us from claiming the whole point \eqref{de:aligned-al1} of Definition~\ref{de:aligned}, but such a condition was not necessary before. We also did not ask for disjoint values as in point \eqref{de:aligned-disj} until now, but only because it was a redundant condition: if $\mathfrak{C}_{1}$ holds for two distinct triples, they must necessarily use two distinct values $z_{1}$.

\begin{lemma}\label{le:4to3}
If there is a solution for Reduction~\ref{re:4}, for all $\delta_{4}>0$ small enough and all $n$ large enough depending on $\delta_{4}$, then there is a solution for Reduction~\ref{re:3}, for all $\delta_{3}>0$ small enough and all $n$ large enough depending on $\delta_{3}$.
\end{lemma}

\begin{proof}
By Lemma~\ref{le:theta4}, every string triple of $\mathcal{X}_{3}:=\mathcal{X}_{3}(n,\delta_{3})$ is the unique preimage of some $(\phi_{1},\phi_{2},\phi_{3},\lambda)\in\mathcal{X}_{4}:=\mathcal{X}_{4}(n',2\delta_{3})$ for some $\frac{29}{31}n\leq n'\leq n$ via $\theta_{4}$. Let $(\eta_{1},\eta_{2},\eta_{3},\phi_{1},\phi_{2},\phi_{3},\lambda)$ be a solution for Reduction~\ref{re:4}. We need to produce $\eta'_{i}$ such that $(\eta'_{1},\eta'_{2},\eta'_{3},\phi'_{1},\phi'_{2},\phi'_{3},\lambda')$ is the required solution for Reduction~\ref{re:3}, where we set $(\phi'_{1},\phi'_{2},\phi'_{3},\lambda')=\theta_{4}^{-1}(\phi_{1},\phi_{2},\phi_{3},\lambda)$.

We work on each $a\in\tilde{P}_{n'}$ with $\lambda(a)\in\mathcal{L}$, starting with the rightmost one and moving left. At the start of each step, we assume that $\mathfrak{C}_{1}^{3}\mathfrak{C}_{3}^{2}\mathfrak{C}_{4}^{2}\mathfrak{C}_{5}\mathfrak{C}_{6}$ holds for the triple of cycles immediately to the left of $a$, and that $\mathfrak{C}_{1}^{2}\mathfrak{C}_{3}\mathfrak{C}_{4}\mathfrak{C}_{5}$ holds for the triple immediately to its right: this is true at the start of the whole process, by Definition~\ref{de:aligned}\eqref{de:aligned-all}.

There are six possibilities for $\lambda(a)$. Suppose first that $\lambda(a)=(1,1)$. Lengthen the strings by $2$ positions in correspondence of $a$, by introducing two new points $u,v$ and embedding naturally the $\alpha_{i}=\mu(\eta_{i},\phi_{i})\in\mathrm{Sym}(n)$ into the pointwise stabilizer $\mathrm{Sym}(n+2)_{(\{u,v\})}$: the strings and the solution are transformed as
\begin{align}\label{eq:4to3lengthen}
\begin{aligned}
& \!\!\! & & \,{\scriptstyle a} & &  & &  & &  \!\!\! & & \,{\scriptstyle a} \ \ \ {\scriptstyle a+1} \ \ {\scriptstyle a+2} & &  & \\
& \cdots \ x_{-2} \ x_{-1} \!\!\! & & \blacksquare & & x_{1} \ x_{2} \ \cdots & &  & & \cdots \ x_{-2} \ x_{-1} \!\!\! & & \blacksquare \ u \ \blacksquare \ v \ \blacksquare & & \!\!\! x_{1} \ x_{2} \ \cdots & \\
& \cdots \ y_{-2} \ y_{-1} \!\!\! & & \blacksquare & & y_{1} \ y_{2} \ \cdots & & \longrightarrow & & \cdots \ y_{-2} \ y_{-1} \!\!\! & & \blacksquare \ u \ \blacksquare \ v \ \blacksquare & & \!\!\! y_{1} \ y_{2} \ \cdots & \\
& \cdots \ z_{-2} \ z_{-1} \!\!\! & & \blacksquare & & z_{1} \ z_{2} \ \cdots & &  & & \cdots \ z_{-2} \ z_{-1} \!\!\! & & \blacksquare \ u \ \blacksquare \ v \ \blacksquare & & \!\!\! z_{1} \ z_{2} \ \cdots & \\
& \!\!\! & & \!\!\!\!\!(1,1)\!\!\!\!\! & &  & &  & & \!\!\! & & \,\emptyset \ \ \ \ \,\emptyset \ \ \ \ \,\emptyset & &  & 
\end{aligned}
\end{align}
Then we use $\mathfrak{C}_{1}$ on the left and $\mathfrak{C}_{4}$ on the right, and we take
\begin{align}\label{eq:4to311}
\begin{aligned}
\alpha_{1}\! & =\!\vec{\beta}_{1}(\vec{\rho}_{1}\, r)(u)(v)(s\, t\, \vec{\sigma}_{1}), & \alpha'_{1}\! & =\!(r\, u)(t\, v)\alpha_{1}(r\, t\, u\, v\, s)\!=\!\vec{\beta}_{1}(\vec{\rho}_{1}\, t\, s\, u)(r\, v\, \vec{\sigma}_{1}), \\
\alpha_{2}\! & =\!\vec{\beta}_{2}(\vec{\rho}_{2}\, r)(u)(v)(s\, t\, \vec{\sigma}_{2}), & \alpha'_{2}\! & =\!(r\, s\, v\, u\, t)\alpha_{2}\!=\!\vec{\beta}_{2}(\vec{\rho}_{2}\, r\, t)(s\, v\, u\, \vec{\sigma}_{2}), \\
\alpha_{3}\! & =\!\vec{\beta}_{3}(\vec{\rho}_{3}\, r)(u)(v)(t\, \vec{\sigma}_{3}), & \alpha'_{3}\! & =\!(r\, u)(t\, v)\alpha_{3}\!=\!\vec{\beta}_{3}(\vec{\rho}_{3}\, r\, u)(t\, v\, \vec{\sigma}_{3}).
\end{aligned}
\end{align}
For the new triple on the left we still have $\mathfrak{C}_{1}^{2}\mathfrak{C}_{3}^{2}\mathfrak{C}_{4}^{2}\mathfrak{C}_{5}\mathfrak{C}_{6}$ (and in particular $\mathfrak{C}_{1}^{2}\mathfrak{C}_{3}\mathfrak{C}_{4}\mathfrak{C}_{5}$), so at the next step our previous assumption is still valid. Moreover, for the new triple on the right we have $\mathfrak{C}_{1}^{2}\mathfrak{C}_{3}\mathfrak{C}_{5}$, and in particular $\mathfrak{C}_{1}$.

If $\lambda(a)=(2,1)$, lengthen the strings via \eqref{eq:4to3lengthen} and, using $\mathfrak{C}_{1}$ on the left and $\mathfrak{C}_{3}$ on the right, take
\begin{align}\label{eq:4to321}
\begin{aligned}
\alpha_{1}\! & =\!\vec{\beta}_{1}(\vec{\rho}_{1}\, r)(u)(v)(s\, t\, \vec{\sigma}_{1}), & \alpha'_{1}\! & =\!(r\, u\, v)\alpha_{1}(r\, u\, s)\!=\!\vec{\beta}_{1}(\vec{\rho}_{1}\, u\, v)(r\, s\, t\, \vec{\sigma}_{1}), \\
\alpha_{2}\! & =\!\vec{\beta}_{2}(\vec{\rho}_{2}\, r)(u)(v)(s\, t\, \vec{\sigma}_{2}), & \alpha'_{2}\! & =\!(r\, s\, v\, t\, u)\alpha_{2}\!=\!\vec{\beta}_{2}(\vec{\rho}_{2}\ r\, t\, u)(s\, v\, \vec{\sigma}_{2}), \\
\alpha_{3}\! & =\!\vec{\beta}_{3}(\vec{\rho}_{3}\, r)(u)(v)(s\, \vec{\sigma}_{3}), & \alpha'_{3}\! & =\!(r\, u\, v)\alpha_{1}(s\, u\, v)\alpha_{2}\!=\!\vec{\beta}_{3}(\vec{\rho}_{3}\, r\, v)(s\, u\, \vec{\sigma}_{3}).
\end{aligned}
\end{align}
The new triple on the left satisfies $\mathfrak{C}_{1}^{2}\mathfrak{C}_{3}^{2}\mathfrak{C}_{4}^{2}\mathfrak{C}_{5}\mathfrak{C}_{6}$, and in particular $\mathfrak{C}_{1}^{2}\mathfrak{C}_{3}\mathfrak{C}_{4}\mathfrak{C}_{5}$, while the new triple on the right satisfies $\mathfrak{C}_{1}^{2}\mathfrak{C}_{4}\mathfrak{C}_{5}$, and in particular $\mathfrak{C}_{1}$.

If $\lambda(a)=(3,1)$, lengthen the strings via \eqref{eq:4to3lengthen} and, using $\mathfrak{C}_{6}$ on the left and either $\mathfrak{C}_{3}$ or $\mathfrak{C}_{4}$ on the right, take
\begin{align}\label{eq:4to331}
\begin{aligned}
\alpha_{1}\! & =\!\vec{\beta}_{1}(\vec{\rho}_{1}\, r\, s)(u)(v)(t\, w\, \vec{\sigma}_{1}),\!\! & \alpha'_{1}\! & =\!(s\, u\, t\, v)\alpha_{1}(s\ t)\!=\!\vec{\beta}_{1}(\vec{\rho}_{1}\, r\, t\, v)(s\, u\, w\, \vec{\sigma}_{1}), \\
\alpha_{2}\! & =\!\vec{\beta}_{2}(\vec{\rho}_{2}\, s)(u)(v)(t\, w\, \vec{\sigma}_{2}), & \alpha'_{2}\! & =\!(s\, t\, u\, v\, w)\alpha_{2}\!=\!\vec{\beta}_{2}(\vec{\rho}_{2}\, s\, w)(t\, u\, v\, \vec{\sigma}_{2}), \\
\alpha_{3}\! & =\!\vec{\beta}_{3}(\vec{\rho}_{3}\, r\, \vec{\rho}_{4}\, s)(u)(v), & \alpha'_{3}\! & =\!(s\, u\, t\, v)\alpha_{1}(s\, u\, v\, w)\alpha_{2}\!=\!\vec{\beta}_{3}(\vec{\rho}_{3}\, r\, u\, \vec{\rho}_{4}\, s\, v).
\end{aligned}
\end{align}
The new triple on the left still satisfies $\mathfrak{C}_{1}^{3}\mathfrak{C}_{3}^{2}\mathfrak{C}_{4}^{2}\mathfrak{C}_{5}$, and in particular $\mathfrak{C}_{1}^{2}\mathfrak{C}_{3}\mathfrak{C}_{4}\mathfrak{C}_{5}$, while the new triple on the right satisfies both $\mathfrak{C}_{1}^{2}\mathfrak{C}_{3}\mathfrak{C}_{5}$ and $\mathfrak{C}_{1}^{2}\mathfrak{C}_{4}\mathfrak{C}_{5}$, and in particular $\mathfrak{C}_{1}$.

To understand what to do for $\lambda(a)\in\{(1,-1),(2,-1),(3,-1)\}$, it is enough to make two observations. First, the equation $\alpha_{1}\alpha_{2}=\alpha_{3}$ can be inverted to yield $\alpha_{2}^{-1}\alpha_{1}^{-1}=\alpha_{3}^{-1}$. Second, if the pair $(\eta,\phi)$ represents an element $\alpha$, then by writing the same strings $\eta,\phi$ from right to left we manage to represent $\alpha^{-1}$. Thus, we can retrieve the remaining solutions by inverting \eqref{eq:4to311}--\eqref{eq:4to321}--\eqref{eq:4to331} appropriately: for $\lambda(a)=(1,-1)$ use $\mathfrak{C}_{4}$ on the left, $\mathfrak{C}_{1}$ on the right, and invert \eqref{eq:4to321}; for $\lambda(a)=(2,-1)$ use $\mathfrak{C}_{3}$ on the left, $\mathfrak{C}_{1}$ on the right, and invert \eqref{eq:4to311};  for $\lambda(a)=(3,-1)$ use $\mathfrak{C}_{3}$ or $\mathfrak{C}_{4}$ on the left, $\mathfrak{C}_{5}$ on the right, and invert \eqref{eq:4to331}. As before, in all cases the new triple on the left satisfies $\mathfrak{C}_{1}^{2}\mathfrak{C}_{3}\mathfrak{C}_{4}\mathfrak{C}_{5}$, allowing us to maintain the assumption, and the new triple on the right satisfies $\mathfrak{C}_{1}$.

In all cases $\alpha'_{1}\alpha'_{2}=\alpha'_{3}$, and the resulting $(\eta'_{1},\eta'_{2},\eta'_{3},\phi'_{1},\phi'_{2},\phi'_{3},\lambda')$ is a solution of the problem obtained by walking back the construction in \eqref{eq:graphiclstep}. Repeat the procedure for all half-positions with a label in $\mathcal{L}$, and the final object is a solution for Reduction~\ref{re:3}.

Now we prove the property involving $\mathfrak{C}_{1}$ in Reduction~\ref{re:3}. For any given step of the form \eqref{eq:4to311}--\eqref{eq:4to321}--\eqref{eq:4to331} (or their inverses), call $B_{i}$ the set of cycles of $(\eta_{i},\phi_{i})$ contained in $\vec{\beta}_{i}$. Consider any triple of cycles $\gamma_{i}$ in $(\eta_{i},\phi_{i})$ sharing $\geq 2$ common positions. If at least one $\gamma_{i}$ is in $B_{i}$, then the process above shows that the corresponding triple of cycles in $(\eta'_{i},\phi'_{i})$ also shares $\geq 2$ common positions. Hence, since $\mathfrak{C}_{1}$ holds for the latter (and does not use $r,s,t,u,v,w$), it holds for the former. Now suppose $\gamma_{i}\notin B_{i}$ for all $i$. If this is the triple containing all the $\vec{\rho}_{i}$, then we showed above that $\mathfrak{C}_{1}^{2}\mathfrak{C}_{3}\mathfrak{C}_{4}\mathfrak{C}_{5}$ (and in particular $\mathfrak{C}_{1}$) holds.  If it is the triple containing all the $\vec{\sigma}_{i}$, then we showed above that $\mathfrak{C}_{1}$ holds. Every triple containing some $\vec{\rho}_{i}$ and some $\vec{\sigma}_{i}$ shares $0$ or $1$ common positions, so it is not included among the triples to consider for the extra condition of Reduction~\ref{re:3}. Hence, $\mathfrak{C}_{1}$ is satisfied by all triples that needed to be considered, and we are done.
\end{proof}

\section{Reduction~\ref{re:5}: isolating occurrences of $c^{\blacksquare}(a)=2$}

In the fifth reduction, we ``isolate'' from each other\footnote{The reader might recall that in the example of \S\ref{se:example} we modified two cycle structures without touching the third. In particular, we managed to turn $c^{\blacksquare}(a)=1$ into $c^{\blacksquare}(a)=3$ (which will happen in Reduction~\ref{re:6}). The same reader might then wonder why we do not do the same here and turn $c^{\blacksquare}(a)=2$ into $c^{\blacksquare}(a)=0$, rather than perform this ``isolation''. The reason is that undoing Reduction~\ref{re:6} in Lemma~\ref{le:6to5}, i.e.\ ``glueing back'' cycles together, is easy, while ``breaking back'' cycles is hard: to break an $(m_{1}+m_{2})$-cycle into \textit{precisely} an $m_{1}$-cycle and an $m_{2}$-cycle, one needs much more control on the solutions than the one provided by the conditions of \S\ref{se:solutions}. If the procedure then needs to be repeated multiple times on the same cycles, the complexity becomes essentially untenable.} the half-positions $a\in\tilde{P}_{n}$ with $c^{\blacksquare}(a)=2$, by creating extra $\blacksquare$ so that there will not be two such $a$ without some occurrence of $\blacksquare$ between them.

Let $(\phi_{1},\phi_{2},\phi_{3},\lambda)\in\mathcal{X}_{4}(n,\delta_{4})$, and let $a\in\tilde{P}_{n}$ with $c^{\blacksquare}(a)=2$. By $\mathfrak{P}(\blacksquare,27)$ in $\mathcal{X}_{4}$, for any such $a$ there is at least one sign $\varepsilon_{a}\in\{\pm 1\}$ such that $c^{\blacksquare}(a+\varepsilon_{a}r)=0$ for all integers $1\leq r\leq 8$ (choose $\varepsilon_{a}=1$ if both signs would be valid\footnote{This is an arbitrary choice, although we do need to make a choice to univocally define $\theta_{5}$.}); for any of those half-positions, $\lambda(a+\varepsilon_{a}r)=\emptyset$ as well since $\mathfrak{P}(\mathrm{labels})$ and $\mathfrak{P}(\mathcal{N},8)$ hold. Let $\mathcal{T}=\{\heartsuit\}$ be the set of the (unique) \textit{trapping label} $\heartsuit$. Define new strings $\phi'_{i}\in\Phi_{n}$ and a new label function $\lambda':\tilde{P}_{n}\rightarrow\{\emptyset\}\cup\mathcal{N}\cup\mathcal{L}\cup\mathcal{T}$ by\footnote{First, we must set $c^{\blacksquare}$ to $3$ for two consecutive half-positions in order to have $\nu$ odd. Second, they need to be placed at distance $\geq 5$ from $a$: if they were closer, after Reduction~\ref{re:7} we might have cycles that are too short to have the properties of Definition~\ref{de:aligned}, thus failing to undo the reductions. A bottleneck example comes from the triples of the solution in \eqref{eq:baby2}, which might fail to satisfy $\mathfrak{C}_{1}^{2}\cap\mathfrak{C}_{1}\mathfrak{C}_{2}$ under the hypothesis of Definition~\ref{de:aligned}\eqref{de:aligned-alp}. Third, we need two more spaces with $\lambda$ set to $\emptyset$ on the other side, because we need to reach $\mathfrak{P}(\mathcal{N},1)$ in Lemma~\ref{le:theta7} to guarantee that every $\theta_{k}$ is injective.}
\begin{align*}
\phi'_{i}(b) & =\begin{cases} \phi_{i}(b) & (b\notin\{a+5\varepsilon_{a},a+6\varepsilon_{a}\}), \\ \blacksquare & (b\in\{a+5\varepsilon_{a},a+6\varepsilon_{a}\}), \end{cases} \\
\lambda'(b) & =\begin{cases} \lambda(b)\in\{\emptyset\}\cup\mathcal{N}\cup\mathcal{L} & (b\notin\{a+5\varepsilon_{a},a+6\varepsilon_{a}\}), \\ \heartsuit\in\mathcal{T} & (b\in\{a+5\varepsilon_{a},a+6\varepsilon_{a}\}). \end{cases}
\end{align*}
Visually, the result of the process for one $a$ (with $\varepsilon_{a}=1$) is as below.
\begin{align*}
\begin{aligned}
\overbrace{\hphantom{mmmmmmmmmmmmmmmmmmmmm}}^{\text{all of them had to be $\square$ and $\emptyset$ originally}} \hphantom{mmn} \\
\begin{aligned}
\tilde{P}_{n}: & & {\scriptstyle a-1} & & {\scriptstyle a} & & {\scriptstyle a+1} & & {\scriptstyle a+2} & & {\scriptstyle a+3} & & {\scriptstyle a+4} & & {\scriptstyle a+5} & & {\scriptstyle a+6} & & {\scriptstyle a+7} & & {\scriptstyle a+8} & & {\scriptstyle a+9} & \\
\phi'_{1}: & & \square & & \square & & \square & & \square & & \square & & \square & & \blacksquare & & \blacksquare & & \square & & \square & & \square & \\
\phi'_{2}: & & \square & & \blacksquare & & \square & & \square & & \square & & \square & & \blacksquare & & \blacksquare & & \square & & \square & & \square & \\
\phi'_{3}: & & \square & & \blacksquare & & \square & & \square & & \square & & \square & & \blacksquare & & \blacksquare & & \square & & \square & & \square & \\
\lambda': & & \emptyset & & \emptyset & & \emptyset & & \emptyset & & \emptyset & & \emptyset & & \heartsuit & & \heartsuit & & \emptyset & & \emptyset & & \emptyset &
\end{aligned}
\end{aligned}
\end{align*}
We do so for all $a$, and obtain at the end new strings $\phi'_{i}$, new classes $C'_{i}:=\mu(\phi'_{i})\in\mathcal{C}_{n}$, and a new label function $\lambda'$.

Now we define $\theta_{5}$. Write
\begin{align}
\mathfrak{P}(\blacksquare',k): & \begin{cases}&\text{for all $i$, for all $a\neq a'$ s.t.\ $\phi_{i}(a)=\phi_{i}(a')=\blacksquare$,} \\ &\text{either $|a-a'|>k$ or $\lambda(a)=\heartsuit$ or $\lambda(a')=\heartsuit$,} \\ &\text{and in all cases $|a-a'|\neq 2$.}\end{cases} \label{eq:prop-cyclet} \\
\mathfrak{P}(\mathcal{T},a,\varepsilon,\varepsilon',R): & \begin{cases}&\text{$(c^{\blacksquare},\lambda)(a+\varepsilon)=(3,\heartsuit)$, $(c^{\blacksquare},\lambda)(a+\varepsilon')=(2,\emptyset)$,} \\ &\text{and $(c^{\blacksquare},\lambda)(a+r)=(0,\emptyset)$ for all $r\in R$.}\end{cases} \nonumber \\
\mathfrak{P}(\mathcal{T}): & \begin{cases}&\text{for all $a$ s.t.\ $\lambda(a)=\heartsuit$, we must have one among:} \\ &\text{1) $\mathfrak{P}(\mathcal{T},a,1,-5,\{-4,-3,-2,-1,2,3\})$, or} \\ &\text{2) $\mathfrak{P}(\mathcal{T},a,-1,-6,\{-5,-4,-3,-2,1,2\})$, or} \\ &\text{3) $\mathfrak{P}(\mathcal{T},a,1,6,\{-2,-1,2,3,4,5\})$, or} \\ &\text{4) $\mathfrak{P}(\mathcal{T},a,-1,5,\{-3,-2,1,2,3,4\})$.}\end{cases} \label{eq:prop-t} \\
\mathfrak{P}(\mathcal{T}'): & \begin{cases}&\text{for all $a_{1}<a_{2}$ s.t.\ $c^{\blacksquare}(a_{1})=c^{\blacksquare}(a_{2})=2$,} \\ &\text{there is $a_{1}<b<a_{2}$ s.t.\ $c^{\blacksquare}(b)\geq 1$.}\end{cases} \label{eq:prop-tprime}
\end{align}
The aforementioned ``isolation'' is encoded in $\mathfrak{P}(\mathcal{T}')$.

\begin{lemma}\label{le:theta5}
Let
\begin{align*}
\mathcal{X}_{5}(m,\delta):= \ & \left\{\, (\phi_{1},\phi_{2},\phi_{3},\lambda)\in\Phi_{m}^{\times 3}\times\Lambda_{m,\{\emptyset\}\cup\mathcal{N}\cup\mathcal{L}\cup\mathcal{T}}\ \right| \\
 & \left. \text{if $C'_{i}:=\mu(\phi_{i})$ then $C'_{i}\in\mathcal{C}_{m}(\delta)$, $\nu_{C'_{1},C'_{2},C'_{3}}$ odd;} \right. \\
 & \left. \vphantom{\Phi_{m}^{\times 3}} \text{$\mathfrak{P}(\mathrm{labels})$, $\mathfrak{P}(\blacksquare',27)$, $\mathfrak{P}(\mathcal{N},2)$, $\mathfrak{P}(\mathcal{L},21)$, $\mathfrak{P}(\mathcal{L})$, $\mathfrak{P}(\mathcal{T})$, $\mathfrak{P}(\mathcal{T}')$} \,\right\},
\end{align*}
using \eqref{eq:cdelta}--\eqref{eq:nu}--\eqref{eq:prop-labels}--\eqref{eq:prop-cyclet}--\eqref{eq:prop-n}--\eqref{eq:prop-lk}--\eqref{eq:prop-l}--\eqref{eq:prop-t}--\eqref{eq:prop-tprime}, and let
\begin{align*}
\theta_{5} & :\mathcal{X}_{4}(n,\delta_{4})\rightarrow\mathcal{X}_{5}(n,10\delta_{4}), & \theta_{5}(\phi_{1},\phi_{2},\phi_{3},\lambda)=(\phi'_{1},\phi'_{2},\phi'_{3},\lambda'),
\end{align*}
following the construction described above.

Then $\theta_{5}$ is a well-defined injective function for all $\delta_{4}>0$ small enough and all $n$ large enough depending on $\delta_{4}$.
\end{lemma}

\begin{proof}
By construction, $c^{\blacksquare}(\phi'_{i})\leq 3(c^{\blacksquare}(\phi_{1})+c^{\blacksquare}(\phi_{2})+c^{\blacksquare}(\phi_{3}))$ for each $i$: therefore, $C'_{i}\in\mathcal{C}_{n}(10\delta_{4})$ by Proposition~\ref{pr:gmcycles}\eqref{pr:gmcycles-1}--\eqref{pr:gmcycles-2}. At each step of the process, from each $C_{i}$ we replace one cycle, say of length $\ell$, with three cycles of length $1,\ell_{1},\ell_{2}$ with $1+\ell_{1}+\ell_{2}=\ell$ and $\ell_{1},\ell_{2}\geq 2$. This implies $\nu_{C_{1},C_{2},C_{3}}=\nu_{C'_{1},C'_{2},C'_{3}}$ at each step, so the final $\nu_{C'_{1},C'_{2},C'_{3}}$ is odd.

We still have $\mathfrak{P}(\mathrm{labels})$. Informally speaking, the distance restrictions for previous labels are now relaxed by $6$, i.e.\ $\mathfrak{P}(\mathcal{N},2)$ and $\mathfrak{P}(\mathcal{L},21)$ hold. There are still no ledges, so $\mathfrak{P}(\mathcal{L})$ holds. Every cycle of $C'_{i}$ either has length $>27$ or (its representation in $\phi'_{i}$) starts or ends at some $a\in\tilde{P}_{n}$ with $\lambda'(a)=\heartsuit$, and even in that case the cycle must have length either $=1$ or $\geq 3$, so $\mathfrak{P}(\blacksquare',27)$ holds. By construction, half-positions $a$ with $\lambda(a)=\heartsuit$ come in pairs and know their immediate surroundings, yielding $\mathfrak{P}(\mathcal{T})$. Finally, if for $\phi'_{1},\phi'_{2},\phi'_{3}$ we have $c^{\blacksquare}(a_{1})=c^{\blacksquare}(a_{2})=2$ and $c^{\blacksquare}(b)=0$ for all $a_{1}<b<a_{2}$, then the same is true for $\phi_{1},\phi_{2},\phi_{3}$, but pigeonhole and $\mathfrak{P}(\blacksquare,27)$ in $\mathcal{X}_{4}$ imply that $a_{2}-a_{1}>27$: the construction then forces $c^{\blacksquare}(a_{1}+5)=c^{\blacksquare}(a_{1}+6)=3$ for $\phi'_{1},\phi'_{2},\phi'_{3}$, and we have $\mathfrak{P}(\mathcal{T}')$.

Inverting $\theta_{5}$ is immediate, thus giving injectivity: wherever $(\phi'_{1},\phi'_{2},\phi'_{3},\lambda')(b)=(\blacksquare,\blacksquare,\blacksquare,\heartsuit)$, put $(\phi_{1},\phi_{2},\phi_{3},\lambda)(b)=(\square,\square,\square,\emptyset)$.
\end{proof}

\begin{reduction}\label{re:5}
\,Find, for any $(\phi_{1},\phi_{2},\phi_{3},\lambda)\in\mathcal{X}_{5}(n,\delta_{5})$, a solution septuple $(\eta_{1},\eta_{2},\eta_{3},\phi_{1},\phi_{2},\phi_{3},\lambda)$ such that $\alpha_{1}\alpha_{2}=\alpha_{3}$ for $\alpha_{i}=\mu(\eta_{i},\phi_{i})$. Furthermore, the solution must be aligned.
\end{reduction}

\begin{lemma}\label{le:5to4}
If there is a solution for Reduction~\ref{re:5}, for all $\delta_{5}>0$ small enough and all $n$ large enough depending on $\delta_{5}$, then there is a solution for Reduction~\ref{re:4}, for all $\delta_{4}>0$ small enough and all $n$ large enough depending on $\delta_{4}$.
\end{lemma}

\begin{proof}
By Lemma~\ref{le:theta5}, every string triple of $\mathcal{X}_{4}:=\mathcal{X}_{4}(n,\delta_{4})$ is the unique preimage of some $(\phi_{1},\phi_{2},\phi_{3},\lambda)\in\mathcal{X}_{5}:=\mathcal{X}_{5}(n,10\delta_{4})$ via $\theta_{5}$. Let $(\eta_{1},\eta_{2},\eta_{3},\phi_{1},\phi_{2},\phi_{3},\lambda)$ be a solution for Reduction~\ref{re:5}. We need to produce $\eta'_{i}$ such that $(\eta'_{1},\eta'_{2},\eta'_{3},\phi'_{1},\phi'_{2},\phi'_{3},\lambda')$ is the required solution for Reduction~\ref{re:4}, where $(\phi'_{1},\phi'_{2},\phi'_{3},\lambda')=\theta_{5}^{-1}(\phi_{1},\phi_{2},\phi_{3},\lambda)$.

Let $a,a+1$ be a pair of half-positions with $\lambda(a)=\lambda(a+1)=\heartsuit$. Using $\mathfrak{C}_{1}$, there are values $r,s,t$ for which we can write $\alpha_{i}=\vec{\beta}_{i}(\vec{\rho}_{i}\ r)(s)(t\ \vec{\sigma}_{i})$. Then take
\begin{align}\label{eq:5to4}
\begin{aligned}
\alpha'_{1} & =(r\ t\ s)\alpha_{1}=\vec{\beta}_{1}(\vec{\rho}_{1}\ r\ \vec{\sigma}_{1}\ t\ s), \\
\alpha'_{2} & =\alpha_{2}(r\ t\ s)=\vec{\beta}_{2}(\vec{\rho}_{2}\ t\ \vec{\sigma}_{2}\ s\ r), \\
\alpha'_{3} & =(r\ t\ s)\alpha_{3}(r\ t\ s)=\vec{\beta}_{3}(\vec{\rho}_{3}\ t\ r\ \vec{\sigma}_{3}\ s).
\end{aligned}
\end{align}
Then $\alpha'_{1}\alpha'_{2}=\alpha'_{3}$, which results in a solution of the problem obtained by removing all the cycle breaks at $a,a+1$. Repeating the procedure for all pairs, starting from the rightmost one and moving left, we obtain a solution for Reduction~\ref{re:4}.

To prove that the solution is aligned, note that the triples of cycles in the new solution inherit all the properties from the ones in the old solution, except for the $\mathfrak{C}_{1}$ involving $r,s,t$. Then, so as to cover this last case, we can prove $\mathfrak{C}_{1}$ for the new cycles in \eqref{eq:5to4} using $z_{1}=r$.
\end{proof}

\section{Reduction~\ref{re:6}: no places with $c^{\blacksquare}(a)=1$}\label{se:re6}

In the sixth reduction, we get rid of half-positions $a\in\tilde{P}_{n}$ with $c^{\blacksquare}(a)=1$, by replacing them with $c^{\blacksquare}(a)=3$.

Let $(\phi_{1},\phi_{2},\phi_{3},\lambda)\in\mathcal{X}_{5}(n,\delta_{5})$, and let $a\in\tilde{P}_{n}$ with $c^{\blacksquare}(a)=1$. By $\mathfrak{P}(\mathcal{L})$ in $\mathcal{X}_{5}$, for any $b$ with $c^{\blacksquare}(b)=2$ we must have $|b-a|\geq 2$, and using $\mathfrak{P}(\mathrm{labels})$, $\mathfrak{P}(\blacksquare',27)$, and $\mathfrak{P}(\mathcal{T})$ we obtain that for any $b$ with $c^{\blacksquare}(b)=3$ we must have $|b-a|\geq 3$. Moreover, by the same properties, for any two half-positions $x_{1}<x_{2}$ with $c^{\blacksquare}(x_{i})\in\{1,2\}$ and $x_{2}-x_{1}\leq 27$, any two indices $i_{1},i_{2}$ for which $\phi_{i_{1}}(x_{1})=\phi_{i_{2}}(x_{2})=\blacksquare$ must be distinct. Therefore, if two $a_{1}<a_{2}$ with $a_{2}-a_{1}\leq 2$ have $c^{\blacksquare}(a_{j})=1$ and $b$ has $c^{\blacksquare}(b)\geq 2$, then we must have $|b-a_{j}|\geq 25$ for both $j$, and there can never be four $a_{1}<a_{2}<a_{3}<a_{4}$ with $a_{4}-a_{1}\leq 27$.

Taking all these facts together, any $c^{\blacksquare}(a)=1$ must be sitting in a sequence of values of $c^{\blacksquare}$ for neighbouring half-positions that has one of the following forms\footnote{The reason why we need so many cases in \eqref{eq:graphicsstart} is that, if the next closest cycle break occurs at some $r$ with $c^{\blacksquare}(r)=3$, then $r$ needs to be at distance $\geq 3$ from any cycle break created during this reduction. If it were closer, after Reduction~\ref{re:7} we might have cycles that are too short to have the properties of Definition~\ref{de:aligned}\eqref{de:aligned-alp}, thus failing to undo the reductions (we could end up having triples without $\mathfrak{C}_{1}$, i.e.\ without common values to use in the undoing process). Then, even after taking care of this requirement, we still need to differentiate more subcases in \eqref{eq:graphicsstart}, for instance the 3rd and the 4th, so as to keep $\nu$ odd.} (ignore the underlining for now, it will come into play soon):
\begin{align}\label{eq:graphicsstart}
\begin{aligned}
& (x,0,\underline{1},0,y) \ \text{with $x,y\in\{0,2\}$,} \\
& (0,0,\underline{1},\underline{1},0,0), \\
& (0,0,\underline{1},\underline{0},\underline{1},\underline{0},0,0), \\
& (0,0,\underline{0},\underline{1},\underline{0},\underline{1},0,0), \\
& (0,0,0,\underline{1},\underline{1},\underline{1},0,0,0), \\
& (0,0,0,\underline{1},\underline{1},\underline{0},\underline{1},\underline{0},0,0), \\
& (0,0,\underline{0},\underline{1},\underline{0},\underline{1},\underline{1},0,0,0), \\
& (0,0,0,\underline{1},\underline{0},\underline{1},\underline{0},\underline{1},0,0,0).
\end{aligned}
\end{align}
By $\mathfrak{P}(\mathrm{labels})$ and $\mathfrak{P}(\mathcal{N},2)$, at all the half-positions $x$ underlined in some case of \eqref{eq:graphicsstart} we have $\lambda(x)=\emptyset$.

Let $\mathcal{S}=\{0,1,2,3\}$ be the set of \textit{shutter labels}. Define new strings $\phi'_{i}\in\Phi_{n}$ and a new label function $\lambda':\tilde{P}_{n}\rightarrow\{\emptyset\}\cup\mathcal{N}\cup\mathcal{L}\cup\mathcal{T}\cup\mathcal{S}$ by
\begin{align*}
\phi'_{i}(b) & =\begin{cases} \blacksquare & (\text{$c^{\blacksquare}(b)$ appears underlined in \eqref{eq:graphicsstart} for some choice of $a$}), \\ \phi_{i}(b) & (\text{otherwise}), \end{cases} \\
\lambda'(b) & =\begin{cases} 0\in\mathcal{S} & (\text{$c^{\blacksquare}(b)$ underlined for some $a$, and $c^{\blacksquare}(b)=0$}), \\ i\in\{1,2,3\}\subseteq\mathcal{S} & (\text{$c^{\blacksquare}(b)$ underlined for some $a$, and $\phi_{i}(b)=\blacksquare$}), \\ \lambda(b)\in\{\emptyset\}\cup\mathcal{N}\cup\mathcal{L}\cup\mathcal{T} & (\text{otherwise}), \end{cases}
\end{align*}
We obtain at the end new strings $\phi'_{i}$, new classes $C'_{i}:=\mu(\phi'_{i})\in\mathcal{C}_{n}$, and a new label function $\lambda'$.

Now we define $\theta_{6}$. All the restrictions we have imposed force the string triple to have a rigid structure: they allow us to break it down into substrings (between two consecutive occurrences of $c^{\blacksquare}(a)=3$) of a very particular shape. Write
\begin{equation}\label{eq:prop-sub}
\mathfrak{P}(\mathrm{sub}): \begin{cases}&\text{for all $a_{1}<a_{2}$ s.t.\ $c^{\blacksquare}(a_{1})=c^{\blacksquare}(a_{2})=3$} \\ &\text{and s.t.\ $c^{\blacksquare}(b)\neq 3$ for all $a_{1}<b<a_{2}$:} \\ &\text{$a_{2}-a_{1}\neq 2$, and for $a_{1}<b<a_{2}$ we have $c^{\blacksquare}(b)=0$} \\ &\text{for all but at most one $b=b_{0}$, and if such $b_{0}$ exists} \\ &\text{then $c^{\blacksquare}(b_{0})=2$, \,$\min\{b_{0}-a_{1},a_{2}-b_{0}\}\geq 2$,}  \\ &\text{and $\max\{b_{0}-a_{1},a_{2}-b_{0}\}\geq 5$.} \end{cases}
\end{equation}
Visually, each restriction defined by $a_{1}$ and $a_{2}$ can have only one of the two following shapes:
\begin{align}
\overbrace{\hphantom{mmmmmmmmmmn}}^{\text{either $=1$ or $\geq 3$}}\hphantom{mmmn}\overbrace{\hphantom{mmmmmmmmmmmmmmm}}^{\begin{aligned}\scriptsize\text{$\geq 2$ on each side of $b_{0}$,\ \ \ } \\ \scriptsize\text{ and $\geq 5$ on at least one side}\end{aligned}} \nonumber \\
\begin{aligned}
\tilde{P}_{n}: & & \ \ & & {\scriptstyle a_{1}} & &  & &  & &  & &  & & {\scriptstyle a_{2}} & &  & & {\scriptstyle a_{1}} & &  & &  & &  & & {\scriptstyle b_{0}} & &  & &  & & {\scriptstyle a_{2}} & \\
\phi_{1}: & & \ \ & & \blacksquare & & \square & & \square & & \square & & \square & & \blacksquare & &  & & \blacksquare & & \square & & \square & & \square & & \blacksquare & & \square & & \square & & \blacksquare & \\
\phi_{2}: & & \ \ & & \blacksquare & & \square & & \square & & \square & & \square & & \blacksquare & & \ \text{or} \ & & \blacksquare & & \square & & \square & & \square & & \blacksquare & & \square & & \square & & \blacksquare & \\
\phi_{3}: & & \ \ & & \blacksquare & & \square & & \square & & \square & & \square & & \blacksquare & &  & & \blacksquare & & \square & & \square & & \square & & \square & & \square & & \square & & \blacksquare & 
\end{aligned} \label{eq:graphicsend}
\end{align}

\begin{lemma}\label{le:theta6}
Let
\begin{align*}
\mathcal{X}_{6}(m,\delta):= \ & \left\{\, (\phi_{1},\phi_{2},\phi_{3},\lambda)\in\Phi_{m}^{\times 3}\times\Lambda_{m,\{\emptyset\}\cup\mathcal{N}\cup\mathcal{L}\cup\mathcal{T}\cup\mathcal{S}}\ \right| \\
 & \left. \text{if $C'_{i}:=\mu(\phi_{i})$ then $C'_{i}\in\mathcal{C}_{m}(\delta)$, $\nu_{C'_{1},C'_{2},C'_{3}}$ odd;} \right. \\
 & \left. \vphantom{\Phi_{m}^{\times 3}} \text{$\mathfrak{P}(\mathrm{labels})$, $\mathfrak{P}(\mathcal{N},1)$, $\mathfrak{P}(\mathcal{L},20)$, $\mathfrak{P}(\mathrm{sub})$} \,\right\},
\end{align*}
using \eqref{eq:cdelta}--\eqref{eq:nu}--\eqref{eq:prop-labels}--\eqref{eq:prop-n}--\eqref{eq:prop-lk}--\eqref{eq:prop-sub}, and let
\begin{align*}
\theta_{6} & :\mathcal{X}_{5}(n,\delta_{5})\rightarrow\mathcal{X}_{6}(n,7\delta_{5}), & \theta_{6}(\phi_{1},\phi_{2},\phi_{3},\lambda)=(\phi'_{1},\phi'_{2},\phi'_{3},\lambda'),
\end{align*}
following the construction described above.

Then $\theta_{6}$ is a well-defined injective function for all $\delta_{5}>0$ small enough and all $n$ large enough depending on $\delta_{5}$.
\end{lemma}

\begin{proof}
By construction, $c^{\blacksquare}(\phi'_{i})\leq 2(c^{\blacksquare}(\phi_{1})+c^{\blacksquare}(\phi_{2})+c^{\blacksquare}(\phi_{3}))$ for each $i$: therefore, $C'_{i}\in\mathcal{C}_{n}(7\delta_{5})$ by Proposition~\ref{pr:gmcycles}\eqref{pr:gmcycles-1}--\eqref{pr:gmcycles-2}.

Let us consider the quantity $\nu$. Recall that in $\phi_{1},\phi_{2},\phi_{3}$ the half-positions with $c^{\blacksquare}(a)=1$ can only occur grouped as in \eqref{eq:graphicsstart}. Say that, if some $a$ are grouped together, we perform the process above on all the underlined half-positions in the group at once in a single step. We examine the cases of \eqref{eq:graphicsstart} one by one to see what happens to $\nu$; for ease of notation, say that in each case the first $a$ with $c^{\blacksquare}(a)=1$ has $\phi_{1}(a)=\blacksquare$, the second one (if any) has $\phi_{2}(a)=\blacksquare$, and the third one (if any) has $\phi_{3}(a)=\blacksquare$. The cycle structure from $C_{i}$ to $C'_{i}$ changes as follows:
\begin{align*}
\text{1st case:} & \begin{cases} \text{$C_{1}$: unchanged;} & \\ \text{$C_{2}$: $\ell_{2}$-cycle} & \rightarrow \ \text{$\{\ell'_{2},(\ell_{2}-\ell'_{2})\}$-cycles, both $\geq 2$;} \\ \text{$C_{3}$: $\ell_{3}$-cycle} & \rightarrow \ \text{$\{\ell'_{3},(\ell_{3}-\ell'_{3})\}$-cycles, both $\geq 2$.} \end{cases} \\
\text{2nd case:} & \begin{cases} \text{$C_{1}$: $\ell_{1}$-cycle} & \rightarrow \ \text{$\{1,(\ell_{1}-1)\}$-cycles, $\ell_{1}-1\geq 3$;} \\ \text{$C_{2}$: $\ell_{2}$-cycle} & \rightarrow \ \text{$\{1,(\ell_{2}-1)\}$-cycles, $\ell_{2}-1\geq 3$;} \\ \text{$C_{3}$: $\ell_{3}$-cycle} & \rightarrow \ \text{$\{1,\ell'_{3},(\ell_{3}-\ell'_{3}-1)\}$-cycles, last two $\geq 3$.} \end{cases} \\
\text{3rd case:} & \begin{cases} \text{$C_{1}$: $\ell_{1}$-cycle} & \!\!\!\rightarrow\text{$\{1,1,1,(\ell_{1}\!-\!3)\}$-cycles, $\ell_{1}\!-\!3\!\geq\!3$;} \\ \text{$C_{2}$: $\{\ell_{21},\ell_{22}\}$-cycles} & \!\!\!\rightarrow\text{$\{1,1,1,(\ell_{21}\!-\!2),(\ell_{22}\!-\!1)\}$-cycles, last two $\!\geq\!3$;} \\ \text{$C_{3}$: $\ell_{3}$-cycle} & \!\!\!\rightarrow\text{$\{1,1,1,\ell'_{3},(\ell_{3}\!-\!\ell'_{3}\!-\!3)\}$-cycles, last two $\!\geq\!3$.} \end{cases} \\
\text{4th case:} & \begin{cases} \text{$C_{1}$: $\{\ell_{11},\ell_{12}\}$-cycles} & \!\!\!\rightarrow\text{$\{1,1,1,(\ell_{11}\!-\!1),(\ell_{12}\!-\!2)\}$-cycles, last two $\!\geq\!3$;} \\ \text{$C_{2}$: $\ell_{2}$-cycle} & \!\!\!\rightarrow\text{$\{1,1,1,(\ell_{2}\!-\!3)\}$-cycles, $\ell_{2}\!-\!3\!\geq\!3$;} \\ \text{$C_{3}$: $\ell_{3}$-cycle} & \!\!\!\rightarrow\text{$\{1,1,1,\ell'_{3},(\ell_{3}\!-\!\ell'_{3}\!-\!3)\}$-cycles, last two $\!\geq\!3$.} \end{cases} \\
\text{5th case:} & \begin{cases} \text{$C_{1}$: $\ell_{1}$-cycle} & \!\!\!\rightarrow \ \text{$\{1,1,(\ell_{1}\!-\!2)\}$-cycles, $\ell_{1}\!-\!2\geq 4$;} \\ \text{$C_{2}$: $\{\ell_{21},\ell_{22}\}$-cycles} & \!\!\!\rightarrow \ \text{$\{1,1,(\ell_{21}\!-\!1),(\ell_{22}\!-\!1)\}$-cycles, last two $\geq 4$;} \\ \text{$C_{3}$: $\ell_{3}$-cycle} & \!\!\!\rightarrow \ \text{$\{1,1,(\ell_{3}\!-\!2)\}$-cycles, $\ell_{3}\!-\!2\geq 4$.} \end{cases} \\
\text{6th case:} & \begin{cases} \text{$C_{1}$: $\ell_{1}$-cycle} & \!\!\!\rightarrow\text{$\{1,\!1,\!1,\!1,(\ell_{1}\!-\!4)\}$-cycles, $\ell_{1}\!-\!4\!\geq\!3$;} \\ \text{$C_{2}$: $\{\ell_{21},\ell_{22}\}$-cycles} & \!\!\!\rightarrow\text{$\{1,\!1,\!1,\!1,(\ell_{21}\!-\!1),(\ell_{22}\!-\!3)\}$-cycles,\! last\! two $\!\geq\!3$;} \\ \text{$C_{3}$: $\{\ell_{31},\ell_{32}\}$-cycles} & \!\!\!\rightarrow\text{$\{1,\!1,\!1,\!1,(\ell_{31}\!-\!3),(\ell_{32}\!-\!1)\}$-cycles,\! last\! two $\!\geq\!3$.} \end{cases} \\
\text{7th case:} & \begin{cases} \text{$C_{1}$: $\{\ell_{11},\ell_{12}\}$-cycles} & \!\!\!\rightarrow\text{$\{1,\!1,\!1,\!1,(\ell_{11}\!-\!1),(\ell_{12}\!-\!3)\}$-cycles,\! last\! two $\!\geq\!3$;} \\ \text{$C_{2}$: $\{\ell_{21},\ell_{22}\}$-cycles} & \!\!\!\rightarrow\text{$\{1,\!1,\!1,\!1,(\ell_{21}\!-\!3),(\ell_{22}\!-\!1)\}$-cycles,\! last\! two $\!\geq\!3$;} \\ \text{$C_{3}$: $\ell_{3}$-cycle} & \!\!\!\rightarrow\text{$\{1,\!1,\!1,\!1,(\ell_{3}\!-\!4)\}$-cycles, $\ell_{3}\!-\!4\!\geq\!3$.} \end{cases} \\
\text{8th case:} & \begin{cases} \text{$C_{1}$: $\ell_{1}$-cycle} & \!\!\!\rightarrow\text{$\{1,\!1,\!1,\!1,(\ell_{1}\!-\!4)\}$-cycles, $\ell_{1}\!-\!4\!\geq\!3$;} \\ \text{$C_{2}$: $\{\ell_{21},\ell_{22}\}$-cycles} & \!\!\!\rightarrow\text{$\{1,\!1,\!1,\!1,(\ell_{21}\!-\!2),(\ell_{22}\!-\!2)\}$-cycles,\! last\! two $\!\geq\!3$;} \\ \text{$C_{3}$: $\ell_{3}$-cycle} & \!\!\!\rightarrow\text{$\{1,\!1,\!1,\!1,(\ell_{3}\!-\!4)\}$-cycles, $\ell_{3}\!-\!4\!\geq\!3$.} \end{cases}
\end{align*}
In the 1st, 2nd, 3rd, and 4th cases we have $|\nu_{C_{1},C_{2},C_{3}}-\nu_{C'_{1},C'_{2},C'_{3}}|=2$, whereas in the 5th, 6th, 7th, and 8th cases we have $\nu_{C_{1},C_{2},C_{3}}=\nu_{C'_{1},C'_{2},C'_{3}}$. In either situation, $\nu_{C'_{1},C'_{2},C'_{3}}$ is odd.

We still have $\mathfrak{P}(\mathrm{labels})$. Since we changed only half-positions $b$ at distance $\leq 1$ from some $a$ with $c^{\blacksquare}(a)=1$, we have $\mathfrak{P}(\mathcal{N},1)$ and $\mathfrak{P}(\mathcal{L},20)$. For $\phi'_{1},\phi'_{2},\phi'_{3}$ there are no more $a$ with $c^{\blacksquare}(a)=1$ and, since those that existed for $\phi_{1},\phi_{2},\phi_{3}$ have been changed to $c^{\blacksquare}(a)=3$, by $\mathfrak{P}(\mathcal{T}')$ any two consecutive $b$ with $c^{\blacksquare}(b)=2$ must be separated by some $a$ with $c^{\blacksquare}(a)=3$. Thus, we have reached the situation of \eqref{eq:graphicsend}, apart from the distance bounds. The fact that $a_{2}-a_{1}\neq 2$ on the left side of \eqref{eq:graphicsend} comes from our construction: the first case of \eqref{eq:graphicsstart} creates cycles of length $2$ only for either $x=2$ or $y=2$, therefore such cycles cannot be part of a restriction as in the left side of \eqref{eq:graphicsend}; the other cases of \eqref{eq:graphicsstart} never create cycles of length $2$. The bound on $\min\{b_{0}-a_{1},a_{2}-b_{0}\}$ on the right side of \eqref{eq:graphicsend} comes from $\mathfrak{P}(\mathcal{L})$ in $\mathcal{X}_{5}$. As for the last bound, assume that we have $\max\{b_{0}-a_{1},a_{2}-b_{0}\}<5$: for $\phi_{1},\phi_{2},\phi_{3}$, we have $c^{\blacksquare}(a_{j})\neq 3$ by $\mathfrak{P}(\blacksquare',26)$ and $\mathfrak{P}(\mathcal{T})$, whereas $c^{\blacksquare}(a_{j})=1$ is only possible when $\phi_{i_{1}}(a_{j})=\phi_{i_{2}}(b_{0})=\phi_{i_{3}}(b_{0})=\blacksquare$ for distinct $i_{1},i_{2},i_{3}$ because of $\mathfrak{P}(\blacksquare',26)$; but then the same $i_{1}$ must be used for both $a_{1},a_{2}$, contradicting $\mathfrak{P}(\blacksquare',26)$. Therefore, $\mathfrak{P}(\mathrm{sub})$ holds.

Showing injectivity by inverting $\theta_{6}$ is again easy: for $(\phi'_{1},\phi'_{2},\phi'_{3},\lambda')(b)=(\blacksquare,\blacksquare,\blacksquare,1)$, put $(\phi_{1},\phi_{2},\phi_{3},\lambda)(b)=(\blacksquare,\square,\square,\emptyset)$; when $\lambda'(b)=2$ and $\lambda'(b)=3$ move the $\blacksquare$ appropriately, and put three $\square$ when $\lambda'(b)=0$.
\end{proof}

\begin{reduction}\label{re:6}
Find, for any $(\phi_{1},\phi_{2},\phi_{3},\lambda)\in\mathcal{X}_{6}(n,\delta_{6})$, a solution septuple $(\eta_{1},\eta_{2},\eta_{3},\phi_{1},\phi_{2},\phi_{3},\lambda)$ such that $\alpha_{1}\alpha_{2}=\alpha_{3}$ for $\alpha_{i}=\mu(\eta_{i},\phi_{i})$. Furthermore, the solution must be aligned.
\end{reduction}

\begin{lemma}\label{le:6to5}
If there is a solution for Reduction~\ref{re:6}, for all $\delta_{6}>0$ small enough and all $n$ large enough depending on $\delta_{6}$, then there is a solution for Reduction~\ref{re:5}, for all $\delta_{5}>0$ small enough and all $n$ large enough depending on $\delta_{5}$.
\end{lemma}

\begin{proof}
By Lemma~\ref{le:theta6}, every string triple of $\mathcal{X}_{5}:=\mathcal{X}_{5}(n,\delta_{5})$ is the unique preimage of some $(\phi_{1},\phi_{2},\phi_{3},\lambda)\in\mathcal{X}_{6}:=\mathcal{X}_{6}(n,7\delta_{5})$ via $\theta_{6}$. Let $(\eta_{1},\eta_{2},\eta_{3},\phi_{1},\phi_{2},\phi_{3},\lambda)$ be a solution for Reduction~\ref{re:6}. We need to produce $\eta'_{i}$ such that the septuple $(\eta'_{1},\eta'_{2},\eta'_{3},\phi'_{1},\phi'_{2},\phi'_{3},\lambda')$ is the required solution for Reduction~\ref{re:5}, where we set $(\phi'_{1},\phi'_{2},\phi'_{3},\lambda')=\theta_{6}^{-1}(\phi_{1},\phi_{2},\phi_{3},\lambda)$.

The half-positions $a\in\tilde{P}_{n}$ having $c^{\blacksquare}(a)=1$ in a string triple of $\mathcal{X}_{5}$ must appear in one of the configurations of \eqref{eq:graphicsstart}, and by the process above we have $\lambda'(b)\in\mathcal{S}$ exactly for those $b\in\tilde{P}_{n}$ that appear underlined in one of such configurations. Each configuration is made of: at most three $a$ with $c^{\blacksquare}(a)=1$; either zero or exactly two $b$ with $c^{\blacksquare}(b)=0$, and in the latter case they are at distance $2$ from each other with one $a$ having $c^{\blacksquare}(a)=1$ between them.

To find the $\eta'_{i}$, we work with each configuration separately, starting from the rightmost one and moving left, and inside each configuration we divide the process into two steps. First, we deal with all the $a$ for which the $\phi'_{i}$ have $c^{\blacksquare}(a)=1$, again from right to left. Fix one such $a$, so that $\lambda(a)\in\{1,2,3\}$. The cycles before and after $a$ in the $\phi_{i}$ share $\geq 1$ common positions, so by $\mathfrak{C}_{1}$ they have a common value: say $r$ for the cycles before $a$, and $s$ for the cycles after $a$. Write $\alpha_{i}=\vec{\beta}_{i}(\vec{\rho}_{i}\ r)(s\ \vec{\sigma}_{i})$. If $\lambda(a)=j$, define $\alpha_{i}^{(j)}$ as
\begin{align}\label{eq:6to5remove1}
\begin{aligned}
\alpha^{(1)}_{1} & \!=\alpha_{1}, & \alpha^{(1)}_{2} & \!=\alpha_{2}(r\, s)=\vec{\beta}_{2}(\vec{\rho}_{2}\, s\, \vec{\sigma}_{2}\, r), & \alpha^{(1)}_{3} & \!=\alpha_{3}(r\, s)=\vec{\beta}_{3}(\vec{\rho}_{3}\, s\, \vec{\sigma}_{3}\, r), \\
\alpha^{(2)}_{2} & \!=\alpha_{2}, & \alpha^{(2)}_{1} & \!=(r\, s)\alpha_{1}=\vec{\beta}_{1}(\vec{\rho}_{1}\, r\, \vec{\sigma}_{1}\, s), & \alpha^{(2)}_{3} & \!=(r\, s)\alpha_{3}=\vec{\beta}_{3}(\vec{\rho}_{3}\, r\, \vec{\sigma}_{3}\, s), \\
\alpha^{(3)}_{3} & \!=\alpha_{3}, & \alpha^{(3)}_{1} & \!=\alpha_{1}(r\, s)=\vec{\beta}_{1}(\vec{\rho}_{1}\, s\, \vec{\sigma}_{1}\, r), & \alpha^{(3)}_{2} & \!=(r\, s)\alpha_{2}=\vec{\beta}_{2}(\vec{\rho}_{2}\, r\, \vec{\sigma}_{2}\, s),
\end{aligned}
\end{align}
and define $\lambda^{(j)}$ by changing only the value of $\lambda$ at $a$, setting instead $\lambda^{(j)}(a)=\emptyset$. In all cases $\alpha^{(j)}_{1}\alpha^{(j)}_{2}=\alpha^{(j)}_{3}$, and the resulting $(\eta^{(j)}_{1},\eta^{(j)}_{2},\eta^{(j)}_{3},\phi^{(j)}_{1},\phi^{(j)}_{2},\phi^{(j)}_{3},\lambda^{(j)})$ is a solution of the problem obtained by removing the cycle breaks at $a$ in the $j'$-th and $j''$-th strings with $j',j''\neq j$. Repeat the procedure for all $a$ with $c^{\blacksquare}(a)=1$ in $\phi'_{i}$, and denote by $(\eta''_{1},\eta''_{2},\eta''_{3},\phi''_{1},\phi''_{2},\phi''_{3},\lambda'')$ the solution at the end of this first step.

Let us prove that the intermediate solution above is aligned. Again, we just have to show that \eqref{eq:6to5remove1} is aligned for each $a$. Consider any triple of cycles $\gamma_{i}$ in $(\eta'_{i},\phi'_{i})$. If it satisfies one of the hypotheses of Definition~\ref{de:aligned}\eqref{de:aligned-al1}--\eqref{de:aligned-all}, there is a triple of cycles $\gamma^{*}_{i}$ in $(\eta_{i},\phi_{i})$ satisfying the same hypotheses and such that either $\gamma_{i}=\gamma^{*}_{i}$ or one of six possibilities happens: $(i,j)=(1,2)$, $\gamma_{1}=(\vec{\rho}_{1}\ r\ \vec{\sigma}_{1}\ s)$, and $\gamma^{*}_{1}\in\{(\vec{\rho}_{1}\ r),(s\ \vec{\sigma}_{1})\}$, and analogously for the other five values of $(i,j)$ with $i\neq j$. By Definition~\ref{de:aligned}\eqref{de:aligned-disj}, since the triples of $(\vec{\rho}_{i}\ r)$ and of $(s\ \vec{\sigma}_{i})$ use the values $r,s$ for their $\mathfrak{C}_{1}$, all other triples of cycles in the element strings $(\eta_{i},\phi_{i})$ must use other values. Moreover, every two consecutive values in $\vec{\rho}_{i}$ or $\vec{\sigma}_{i}$ are still consecutive in the $\gamma_{i}$ containing them. Therefore, all properties that do not involve $r,s$ pass to the corresponding triples in $(\eta'_{i},\phi'_{i})$. Furthermore, the property $\mathfrak{C}_{1}$ holds with $z_{1}=r$ and $z_{1}=s$ in the two triples not already fully considered. Hence, the solution resulting from \eqref{eq:6to5remove1} is aligned, and so is $(\eta''_{1},\eta''_{2},\eta''_{3},\phi''_{1},\phi''_{2},\phi''_{3},\lambda'')$.

In the second step, we deal with pairs $b,b+2\in\tilde{P}_{n}$ with $\lambda''(b)=\lambda''(b+2)=0$. After the first step, we must have $c^{\blacksquare}(b+1)=1$ for the $\phi''_{i}$, i.e.\ $\phi''_{j}(b+1)=\blacksquare$ for exactly one $j$. Using $\mathfrak{C}_{1}$, we can fix common values $r,s,t,u$ and write $\alpha''_{j}=\vec{\beta}_{j}(\vec{\rho}_{j}\ r)(s)(t)(u\ \vec{\sigma}_{j})$ and $\alpha''_{i}=\vec{\beta}_{i}(\vec{\rho}_{i}\ r)(s\ t)(u\ \vec{\sigma}_{i})$ for $i\neq j$. Define ${\alpha''_{i}}^{(j)}$ as
\begin{align}\label{eq:6to5remove0}
\begin{aligned}
{\alpha''_{1}}^{(1)} & =(r\ s)(t\ u)\alpha''_{1}=(\vec{\rho}_{1}\ r\ s)(u\ t\ \vec{\sigma}_{1}), \\
{\alpha''_{2}}^{(1)} & =\alpha''_{2}(r\ t\ u)=(\vec{\rho}_{2}\ t\ s\ u\ \vec{\sigma}_{2}\ r), \\
{\alpha''_{3}}^{(1)} & =(r\ s)(t\ u)\alpha''_{3}(r\ t\ u)=(\vec{\rho}_{3}\ t\ \vec{\sigma}_{3}\ r\ u\ s), \\
{\alpha''_{1}}^{(2)} & =(r\ u\ t)\alpha''_{1}=(\vec{\rho}_{1}\ r\ \vec{\sigma}_{1}\ u\ s\ t), \\
{\alpha''_{2}}^{(2)} & =\alpha''_{2}(r\ s)(t\ u)=(\vec{\rho}_{2}\ s\ r)(t\ u\ \vec{\sigma}_{2}), \\
{\alpha''_{3}}^{(2)} & =(r\ u\ t)\alpha''_{3}(r\ s)(t\ u)=(\vec{\rho}_{3}\ s\ u\ r\ \vec{\sigma}_{3}\ t), \\
{\alpha''_{1}}^{(3)} & =(r\ s\ u)\alpha''_{1}=(\vec{\rho}_{1}\ r\ t\ s\ \vec{\sigma}_{1}\ u), \\
{\alpha''_{2}}^{(3)} & =\alpha''_{2}(r\ t\ u)=(\vec{\rho}_{2}\ t\ s\ u\ \vec{\sigma}_{2}\ r), \\
{\alpha''_{3}}^{(3)} & =(r\ s\ u)\alpha''_{3}(r\ t\ u)=(\vec{\rho}_{3}\ t\ u)(r\ s\ \vec{\sigma}_{3}),
\end{aligned}
\end{align}
and define ${\lambda''}^{(j)}$ by setting ${\lambda''}^{(j)}(b)={\lambda''}^{(j)}(b+2)=\emptyset$. In all cases ${\alpha''_{1}}^{(j)}{\alpha''_{2}}^{(j)}={\alpha''_{3}}^{(j)}$, and $({\eta''_{1}}^{(j)},{\eta''_{2}}^{(j)},{\eta''_{3}}^{(j)},{\phi''_{1}}^{(j)},{\phi''_{1}}^{(j)},{\phi''_{1}}^{(j)},{\lambda''}^{(j)})$ is a solution of the problem obtained by removing all the cycle breaks at $b,b+2$. Repeat the procedure for all pairs $b,b+2$, and the final $(\eta'_{1},\eta'_{2},\eta'_{3},\phi'_{1},\phi'_{2},\phi'_{3},\lambda')$ is a solution for Reduction~\ref{re:5}.

To prove that each $({\eta''_{1}}^{(j)},{\eta''_{2}}^{(j)},{\eta''_{3}}^{(j)},{\phi''_{1}}^{(j)},{\phi''_{1}}^{(j)},{\phi''_{1}}^{(j)},{\lambda''}^{(j)})$ is aligned, so that the final solution will be aligned as well, the process is analogous to what we did for the first step. Again, every property except the $\mathfrak{C}_{1}$ that uses $r,s,t,u$ is preserved in the process, and the two values $r,t$ prove that $\mathfrak{C}_{1}$ holds in the remaining triples.
\end{proof}

Effectively, $\mathfrak{P}(\mathrm{sub})$ divides the original problem in $\mathrm{Sym}(n)$ into smaller subproblems, one for each $\mathrm{Sym}(a_{2}-a_{1})$. Each of these individual subproblems has solutions that are easy to find, up to one more technicality: although $\nu_{C_{1},C_{2},C_{3}}$ is odd, not every restricted $\nu$ has to be odd (in which case there would be no solution at all).

\section{Reduction~\ref{re:7}: subproblems have $\nu$ odd}\label{se:re7}

In the seventh and final reduction, we correct the value of $\nu$ for each of the subproblems by shaving off one point where needed.

Let $(\phi_{1},\phi_{2},\phi_{3},\lambda)\in\mathcal{X}_{6}(n,\delta_{6})$, and let $\frac{1}{2}=a_{0}<a_{1}<\ldots<a_{l}=n+\frac{1}{2}$ be the half-positions with $c^{\blacksquare}(a_{j})=3$. By $\mathfrak{P}(\mathrm{sub})$ in $\mathcal{X}_{6}$, if for a given $j$ there is some $x$ with $a_{j-1}<x<a_{j}$ and  $c^{\blacksquare}(x)=2$ then $\max\{a_{j}-x,x-a_{j-1}\}\geq 5$. Define
\begin{equation}\label{eq:defbj}
b_{j}:=\begin{cases} a_{j}-1 & (\text{there is no $x$}), \\  a_{j}-1 & (a_{j}-x\geq x-a_{j-1}), \\ a_{j-1}+1 & (a_{j}-x<x-a_{j-1}), \end{cases}
\end{equation}
so that $b_{j}$ marks an extreme of the interval $[a_{j-1}+1,a_{j}-1]$ that does not sit within distance $2$ from a half-position $b'$ with $c^{\blacksquare}(b')=2$; since $\mathfrak{P}(\mathcal{N},1)$ holds in $\mathcal{X}_{6}$, we must have $\lambda(b_{j})=\emptyset$. Call $\phi_{i}|_{j}\in\Phi_{a_{j}-a_{j-1}}$ the restriction of $\phi_{i}$ to the interval $[a_{j-1},a_{j}]$, and call $C_{i}|_{j}:=\mu(\phi_{i}|_{j})\in\mathcal{C}_{a_{j}-a_{j-1}}$. Since $\nu_{C_{1},C_{2},C_{3}}$ is odd, there is an even number of indices $j$ such that $\nu_{j}:=\nu_{C_{1}|_{j},C_{2}|_{j},C_{3}|_{j}}$ is not odd. Let $J$ be the set of such $j$. Let $\mathcal{P}=\{\spadesuit\}$ be the set of the (unique) \textit{parity label} $\spadesuit$. Define new strings $\phi'_{i}\in\Phi_{n}$ and a new label function $\lambda':\tilde{P}_{n}\rightarrow\{\emptyset\}\cup\mathcal{N}\cup\mathcal{L}\cup\mathcal{T}\cup\mathcal{S}\cup\mathcal{P}$ by
\begin{align*}
\phi'_{i}(b) & =\begin{cases} \blacksquare & (b=b_{j},\, j\in J), \\ \phi_{i}(b) & (\text{otherwise}), \end{cases} & \lambda'(b) & =\begin{cases} \spadesuit\in\mathcal{P} & (b=b_{j},\, j\in J), \\ \lambda(b)\notin\mathcal{P} & (\text{otherwise}). \end{cases}
\end{align*}

Now we define $\theta_{7}$. Write
\begin{align}
\mathfrak{P}(\mathcal{L}',k): & \begin{cases}&\text{for all $a$, if $\lambda(a)\in\mathcal{L}$ then $c^{\blacksquare}(a+r)=0$ for all $1\leq|r|\leq k$,} \\ &\text{except possibly when $|r|=1$ and $\lambda(a+r)=\spadesuit$.} \end{cases} \label{eq:prop-lprime} \\
\mathfrak{P}(\mathrm{sub}'): & \begin{cases}&\text{for all $a_{1}<a_{2}$ s.t.\ $c^{\blacksquare}(a_{1})=c^{\blacksquare}(a_{2})=3$} \\ &\text{and s.t.\ $c^{\blacksquare}(b)\neq 3$ for all $a_{1}<b<a_{2}$:} \\ &\text{$a_{2}-a_{1}\neq 2$, and for $a_{1}<b<a_{2}$ we have $c^{\blacksquare}(b)=0$} \\ &\text{for all but at most one $b=b_{0}$, and if such $b_{0}$ exists} \\ &\text{then $c^{\blacksquare}(b_{0})=2$, \,$\min\{b_{0}-a_{1},a_{2}-b_{0}\}\geq 2$, \,$a_{2}-a_{1}\geq 7$}  \\ &\text{and $\lambda(a_{i})\notin\mathcal{P}$ whenever $|b_{0}-a_{i}|\leq 3$.} \end{cases} \label{eq:prop-subprime} \\
\mathfrak{P}(\mathcal{P}): & \begin{cases}&\text{for all $a_{1}<a_{2}$ s.t.\ $c^{\blacksquare}(a_{1})=c^{\blacksquare}(a_{2})=3$,} \\ &\text{for $\overline{C}_{i}$ the restriction of $C_{i}$ to $[a_{1},a_{2}]$,} \\ &\text{$\nu_{\overline{C}_{1},\overline{C}_{2},\overline{C}_{3}}$ is odd.} \end{cases} \label{eq:prop-p} \\
\mathfrak{P}(\mathcal{P}'): & \begin{cases}&\text{for all $a_{1}<a_{2}$ s.t.\ $\lambda(a_{1})=\lambda(a_{2})=\spadesuit$,} \\ &\text{there is $a_{1}<b<a_{2}$ with $c^{\blacksquare}(b)=3$ and $\lambda(b)\neq\spadesuit$.} \end{cases} \label{eq:prop-pprime}
\end{align}
Visually, instead of \eqref{eq:graphicsend}, $\mathfrak{P}(\mathrm{sub}')$ has
\begin{align}
\overbrace{\hphantom{mmmmmmmmmmn}}^{\text{either $=1$ or $\geq 3$}}\hphantom{mmmn}\overbrace{\hphantom{mmmmmmmmmmmmmmm}}^{\begin{aligned}\scriptsize\text{$\geq 2$ on each side, $\geq 7$ overall\ \ \ } \\ \scriptsize\text{if $\leq 3$ on one side then $\not\in\mathcal{P}$ there}\end{aligned}} \nonumber \\
\begin{aligned}
\tilde{P}_{n}: & & \ \ & & {\scriptstyle a_{1}} & &  & &  & &  & &  & & {\scriptstyle a_{2}} & &  & & {\scriptstyle a_{1}} & &  & &  & &  & & {\scriptstyle b_{0}} & &  & &  & & {\scriptstyle a_{2}} & \\
\phi_{1}: & & \ \ & & \blacksquare & & \square & & \square & & \square & & \square & & \blacksquare & &  & & \blacksquare & & \square & & \square & & \square & & \blacksquare & & \square & & \square & & \blacksquare & \\
\phi_{2}: & & \ \ & & \blacksquare & & \square & & \square & & \square & & \square & & \blacksquare & & \ \text{or} \ & & \blacksquare & & \square & & \square & & \square & & \blacksquare & & \square & & \square & & \blacksquare & \\
\phi_{3}: & & \ \ & & \blacksquare & & \square & & \square & & \square & & \square & & \blacksquare & &  & & \blacksquare & & \square & & \square & & \square & & \square & & \square & & \square & & \blacksquare & 
\end{aligned} \label{eq:graphicsendprime}
\end{align}

\begin{lemma}\label{le:theta7}
Let
\begin{align*}
\mathcal{X}_{7}(m,\delta):= \ & \left\{\, (\phi_{1},\phi_{2},\phi_{3},\lambda)\in\Phi_{m}^{\times 3}\times\Lambda_{m,\{\emptyset\}\cup\mathcal{N}\cup\mathcal{L}\cup\mathcal{T}\cup\mathcal{S}\cup\mathcal{P}}\ \right| \\
 & \left. \vphantom{\Phi_{m}^{\times 3}} \text{if $C'_{i}:=\mu(\phi_{i})$ then $C'_{i}\in\mathcal{C}_{m}(\delta)$;} \right. \\
 & \left. \vphantom{\Phi_{m}^{\times 3}} \text{$\mathfrak{P}(\mathcal{L}',19)$, $\mathfrak{P}(\mathrm{sub}')$, $\mathfrak{P}(\mathcal{P})$, $\mathfrak{P}(\mathcal{P}')$} \,\right\},
\end{align*}
using \eqref{eq:cdelta}--\eqref{eq:nu}--\eqref{eq:prop-labels}--\eqref{eq:prop-lprime}--\eqref{eq:prop-subprime}--\eqref{eq:prop-p}--\eqref{eq:prop-pprime}, and let
\begin{align*}
\theta_{7} & :\mathcal{X}_{6}(n,\delta_{6})\rightarrow\mathcal{X}_{7}(n,3\delta_{6}), & \theta_{7}(\phi_{1},\phi_{2},\phi_{3},\lambda)=(\phi'_{1},\phi'_{2},\phi'_{3},\lambda'),
\end{align*}
following the construction described above.

Then $\theta_{7}$ is a well-defined injective function for all $\delta_{6}>0$ small enough and all $n$ large enough depending on $\delta_{6}$.
\end{lemma}

\begin{proof}
By construction, $c^{\blacksquare}(\phi'_{i})\leq 2c^{\blacksquare}(\phi_{i})$ for each $i$: therefore, $C'_{i}\in\mathcal{C}_{n}(3\delta_{6})$ by Proposition~\ref{pr:gmcycles}\eqref{pr:gmcycles-1}--\eqref{pr:gmcycles-2}. Since $b_{j}$ is always at distance $1$ from some $a$ with $c^{\blacksquare}(a)=3$, $\mathfrak{P}(\mathcal{L},20)$ in $\mathcal{X}_{6}$ implies $\mathfrak{P}(\mathcal{L}',19)$ in $\mathcal{X}_{7}$. There is at most one $b_{j}$ for each interval $[a_{j-1},a_{j}]$, so $\mathfrak{P}(\mathcal{P}')$ holds. At each step of the process, in one of the restrictions for which $\nu_{j}$ is even, from each $C_{i}|_{j}$ we replace one cycle, say of length $\ell$ ($\ell\geq 3$ by the choice of $b_{j}$), with two cycles of length $1,\ell-1$: this implies that $\nu'_{j}:=\nu_{C'_{1}|_{j},C'_{2}|_{j},C'_{3}|_{j}}$ is odd. The property passes to any glueing of intervals, yielding $\mathfrak{P}(\mathcal{P})$.

It remains to prove $\mathfrak{P}(\mathrm{sub}')$. The conditions of $\mathfrak{P}(\mathrm{sub})$ not involving distance bounds and labelling pass directly from $\mathcal{X}_{6}$ to $\mathcal{X}_{7}$. In the left case of \eqref{eq:graphicsend}, if $a_{j}-a_{j-1}=3$ then $\nu_{j}$ is odd, so $j\notin J$ and the condition $a_{2}-a_{1}\neq 2$ is still satisfied inside $\mathcal{X}_{7}$. Similarly, in the right case of \eqref{eq:graphicsend}, $\min\{b_{0}-a_{j-1},a_{j}-b_{0}\}\geq 2$ and $\max\{b_{0}-a_{j-1},a_{j}-b_{0}\}\geq 5$ imply that $a_{j}-a_{j-1}\geq 7$, but if equality holds then $\nu_{j}$ is odd, so the condition $a_{2}-a_{1}\geq 7$ holds inside $\mathcal{X}_{7}$. Moreover, $\max\{b_{0}-a_{j-1},a_{j}-b_{0}\}\geq 5$ and the choice of $b_{j}$ imply together that if $\lambda(a_{i})\in\mathcal{P}$ then $|b_{0}-a_{i}|\geq 4$. Therefore, $\mathfrak{P}(\mathrm{sub}')$ holds.

Injectivity by inverting $\theta_{7}$ is again immediate: wherever $(\phi'_{1},\phi'_{2},\phi'_{3},\lambda')(b)=(\blacksquare,\blacksquare,\blacksquare,\spadesuit)$, put $(\phi_{1},\phi_{2},\phi_{3},\lambda)(b)=(\square,\square,\square,\emptyset)$ since $\mathfrak{P}(\mathcal{N},1)$ holds in $\mathcal{X}_{6}$.
\end{proof}

\begin{reduction}\label{re:7}
Find, for any $(\phi_{1},\phi_{2},\phi_{3},\lambda)\in\mathcal{X}_{7}(n,\delta_{7})$, a solution septuple $(\eta_{1},\eta_{2},\eta_{3},\phi_{1},\phi_{2},\phi_{3},\lambda)$ such that $\alpha_{1}\alpha_{2}=\alpha_{3}$ for $\alpha_{i}=\mu(\eta_{i},\phi_{i})$. Furthermore, the solution must be aligned.
\end{reduction}

\begin{lemma}\label{le:7to6}
If there is a solution for Reduction~\ref{re:7}, for all $\delta_{7}>0$ small enough and all $n$ large enough depending on $\delta_{7}$, then there is a solution for Reduction~\ref{re:6}, for all $\delta_{6}>0$ small enough and all $n$ large enough depending on $\delta_{6}$.
\end{lemma}

\begin{proof}
By Lemma~\ref{le:theta7}, every string triple of $\mathcal{X}_{6}:=\mathcal{X}_{6}(n,\delta_{6})$ is the unique preimage of some $(\phi_{1},\phi_{2},\phi_{3},\lambda)\in\mathcal{X}_{7}:=\mathcal{X}_{7}(n,3\delta_{6})$ via $\theta_{7}$. Let $(\eta_{1},\eta_{2},\eta_{3},\phi_{1},\phi_{2},\phi_{3},\lambda)$ be a solution for Reduction~\ref{re:7}. We need to produce $\eta'_{i}$ such that the septuple $(\eta'_{1},\eta'_{2},\eta'_{3},\phi'_{1},\phi'_{2},\phi'_{3},\lambda')$ is the required solution for Reduction~\ref{re:6}, where we set $(\phi'_{1},\phi'_{2},\phi'_{3},\lambda')=\theta_{7}^{-1}(\phi_{1},\phi_{2},\phi_{3},\lambda)$.

For $(\phi'_{1},\phi'_{2},\phi'_{3},\lambda')\in\mathcal{X}_{6}$, define $a_{j},b_{j},J$ as above, so that $\lambda(x)=\spadesuit\in\mathcal{P}$ exactly at $x=b_{j}$ with $j\in J$. To create the $\eta'_{i}$, we work separately on each pair of values $b_{j_{1}},b_{j_{2}}$ with $j_{1},j_{2}$ consecutive inside $J$. Assume that the $b_{1}=a_{1}-1,b_{2}=a_{2}-1$ (the cases in which $b_{i}=a_{i-1}+1$ work similarly).

Up to cycling around, and using $\mathfrak{C}_{1}$ and $\mathfrak{C}_{2}$ by Definition~\ref{de:aligned}\eqref{de:aligned-al1}--\eqref{de:aligned-alp}, we are in the following situation for some $r,s,t,u,v$:
\begin{align*}
\begin{aligned}
\tilde{P}_{n}: & &  & &  & &  & & b_{j_{1}} \,\ a_{j_{1}}\!  & &  & &  & &  & & b_{j_{2}} \,\ a_{j_{2}}\! & &  & \\
(\eta_{1},\phi_{1}): & & \cdots & & \hphantom{\square} \hphantom{\square} \blacksquare & & \vec{\rho}_{1}\hphantom{\square} & & r \ \blacksquare \ s \ \blacksquare & & \cdots & & \hphantom{\square} \hphantom{\square} \blacksquare & & \vec{\sigma}_{1}\hphantom{\square} & & t \ u \ \blacksquare \ v \ \blacksquare & & \cdots & \\
(\eta_{2},\phi_{2}): & & \cdots & & \blacksquare \hphantom{\square} \hphantom{\square} & & \vec{\rho}_{2}\hphantom{\square} & & r \ \blacksquare \ s \ \blacksquare & & \cdots & & \blacksquare \hphantom{\square} \hphantom{\square} & & \vec{\sigma}_{2}\hphantom{\square} & & u \ \blacksquare \ v \ \blacksquare & & \cdots & \\
(\eta_{3},\phi_{3}): & & \cdots & & \hphantom{\square} \blacksquare \hphantom{\square} & & \vec{\rho}_{3}\hphantom{\square} & & r \ \blacksquare \ s \ \blacksquare & & \cdots & & \hphantom{\square} \blacksquare \hphantom{\square} & & \vec{\sigma}_{3}\hphantom{\square} & & t \ \blacksquare \ v \ \blacksquare & & \cdots & \\
\lambda: & &  & &  & &  & & \spadesuit \ \ \hphantom{a_{j_{1}}}  & &  & &  & &  & & \spadesuit \ \ \hphantom{a_{j_{1}}} & &  &
\end{aligned}
\end{align*}
corresponding as usual to permutations $\alpha_{i}=\mu(\eta_{i},\phi_{i})$ of the form
\begin{align}
\alpha_{1} & =\vec{\beta}_{1}(\vec{\rho}_{1}\ r)(s)(\vec{\sigma}_{1}\ t\ u)(v), \nonumber \\
\alpha_{2} & =\vec{\beta}_{2}(\vec{\rho}_{2}\ r)(s)(\vec{\sigma}_{2}\ u)(v), \label{eq:7to6-alpha} \\
\alpha_{3} & =\alpha_{1}\alpha_{2}=\vec{\beta}_{3}(\vec{\rho}_{3}\ r)(s)(\vec{\sigma}_{3}\ t)(v), \nonumber
\end{align}
where each $\vec{\beta}_{i}$ is the product of the other cycles of $\alpha_{i}$.

Take
\begin{align}
\alpha'_{1} & :=\alpha_{1}(r\ s)(u\ v)=\vec{\beta}_{1}(\vec{\rho}_{1}\ s\ r)(\vec{\sigma}_{1}\ t\ v\ u), \nonumber \\
\alpha'_{2} & :=(r\ u\ s)\alpha_{2}(r\ v\ s)=\vec{\beta}_{2}(\vec{\rho}_{2}\ v\ s)(\vec{\sigma}_{2}\ u\ r), \label{eq:7to6-alphaprime} \\
\alpha'_{3} & :=\alpha'_{1}\alpha'_{2}=\vec{\beta}_{3}(\vec{\rho}_{3}\ v\ r)(\vec{\sigma}_{3}\ t\ s). \nonumber
\end{align}
Then the naturally defined $\eta'_{i},\phi'_{i}$ with $\alpha'_{i}=\mu(\eta'_{i},\phi'_{i})$ give solutions for the problem in which we replace $(\blacksquare,\blacksquare,\blacksquare,\spadesuit)$ with $(\square,\square,\square,\emptyset)$ at the half-positions $b_{j_{1}},b_{j_{2}}$, as becomes apparent by the construction
\begin{align*}
\begin{aligned}
\tilde{P}_{n}: & &  & &  & &  & & b_{j_{1}} \,\ a_{j_{1}}\!  & &  & &  & &  & & b_{j_{2}} \,\ a_{j_{2}}\! & &  & \\
(\eta'_{1},\phi'_{1}): & & \cdots & & \hphantom{\square} \hphantom{\square} \blacksquare & & \vec{\rho}_{1}\hphantom{\square} & & s \ \square \ r \ \blacksquare & & \cdots & & \hphantom{\square} \hphantom{\square} \blacksquare & & \vec{\sigma}_{1}\hphantom{\square} & & t \ v \ \square \ u \ \blacksquare & & \cdots & \\
(\eta'_{2},\phi'_{2}): & & \cdots & & \blacksquare \hphantom{\square} \hphantom{\square} & & \vec{\rho}_{2}\hphantom{\square} & & v \ \square \ s \ \blacksquare & & \cdots & & \blacksquare \hphantom{\square} \hphantom{\square} & & \vec{\sigma}_{2}\hphantom{\square} & & u \ \square \ r \ \blacksquare & & \cdots & \\
(\eta'_{3},\phi'_{3}): & & \cdots & & \hphantom{\square} \blacksquare \hphantom{\square} & & \vec{\rho}_{3}\hphantom{\square} & & v \ \square \ r \ \blacksquare & & \cdots & & \hphantom{\square} \blacksquare \hphantom{\square} & & \vec{\sigma}_{3}\hphantom{\square} & & t \ \square \ s \ \blacksquare & & \cdots & \\
\lambda': & &  & &  & &  & & \emptyset \ \ \hphantom{a_{j_{1}}}  & &  & &  & &  & & \emptyset \ \ \hphantom{a_{j_{1}}} & &  &
\end{aligned}
\end{align*}
Repeating the process for all pairs $b_{j_{1}},b_{j_{2}}$, the final $(\eta'_{1},\eta'_{2},\eta'_{3},\phi'_{1},\phi'_{2},\phi'_{3},\lambda')$ is a solution for Reduction~\ref{re:6}.

Now we need to prove the alignment of the solution. We will show that alignment is preserved at every iteration of going from \eqref{eq:7to6-alpha} to \eqref{eq:7to6-alphaprime}. For each $i$, call $B_{i}$ the set of cycles $\gamma_{i}$ of $(\eta'_{i},\phi'_{i})$ that are contained in $\vec{\beta}_{i}$ in the expressions \eqref{eq:7to6-alphaprime}. For a given $\gamma_{i}$, we denote by $\gamma^{*}_{i}$ the longest cycle of $(\eta_{i},\phi_{i})$ from which $\gamma_{i}$ comes in the expressions \eqref{eq:7to6-alpha}: in other words, if $\gamma_{i}\in B_{i}$ then $\gamma^{*}_{i}=\gamma_{i}$, if $\gamma_{i}=(\vec{\rho}_{1}\ s\ r)$ then $\gamma^{*}_{i}=(\vec{\rho}_{1}\ r)$, and so on.

Consider any triple of cycles $\gamma_{i}$ of $(\eta'_{i},\phi'_{i})$: the $\gamma_{i}$ cannot fall into cases \eqref{de:aligned-alp} and \eqref{de:aligned-alpl} of Definition~\ref{de:aligned}, since there is no label in $\mathcal{P}$ in $\mathcal{X}_{6}$. If $\gamma_{i}\in B_{i}$ for all $i$, then since $\gamma^{*}_{i}=\gamma_{i}$ the properties of Definition~\ref{de:aligned} are preserved. Suppose that $\gamma_{i_{1}}\in B_{i_{1}}$ and $\gamma_{i_{2}}\notin B_{i_{2}}$ for some $i_{1},i_{2}$. If we are in case \eqref{de:aligned-al1} then the $\gamma^{*}_{i}$ are in case \eqref{de:aligned-al1} too, so $\mathfrak{C}_{1}$ holds for them, and the corresponding value $z_{1}$ cannot be $r,s,v$ because none of them sits in a cycle of $B_{i}$ in \eqref{eq:7to6-alpha}; however, for any $z_{1}\notin\{r,s,v\}$, if $z_{1}\in\gamma^{*}_{i}$ then $z_{1}\in\gamma_{i}$ as well, so $\mathfrak{C}_{1}$ holds for the $\gamma_{i}$. If we are in case \eqref{de:aligned-all} then we must have exactly two indices $i_{j}$ with $\gamma_{i_{j}}\in B_{i_{j}}$: we cannot have just one because there would need to be some $b$ with $c^{\blacksquare}(b)=1$, but $\mathfrak{P}(\mathrm{sub})$ holds in $\mathcal{X}_{6}$ so it cannot happen. Therefore $\gamma^{*}_{i_{j}}\in B_{i_{j}}$ in \eqref{eq:7to6-alpha} for two indices, the $\gamma^{*}_{i}$ are in case \eqref{de:aligned-all} too, $\mathfrak{C}_{1}^{3}\mathfrak{C}_{3}^{2}\mathfrak{C}_{4}^{2}\mathfrak{C}_{5}\mathfrak{C}_{6}$ holds for them, and the values $z_{j}$ involved cannot be any of $r,s,t,u,v$ since each of those sits in at least two $\gamma^{*}_{i}\notin B_{i}$ in \eqref{eq:7to6-alpha}. For every other pair of $z_{1},z_{2}$, if they sit consecutively in $\gamma^{*}_{i}$ then they sit consecutively in $\gamma_{i}$ too. Thus, $\mathfrak{C}_{1}^{3}\mathfrak{C}_{3}^{2}\mathfrak{C}_{4}^{2}\mathfrak{C}_{5}\mathfrak{C}_{6}$ holds for the $\gamma_{i}$ using the same $z_{j}$.

Suppose finally that $\gamma_{i}\notin B_{i}$ for all $i$: clearly there are only two triples that share common positions, namely the triple containing all the $\vec{\rho}_{i}$ and the triple containing all the $\vec{\sigma}_{i}$.

Assume that we are in case \eqref{de:aligned-al1}. Then the corresponding cycles $\gamma^{*}_{i}$ are in case \eqref{de:aligned-alp}, and $\mathfrak{C}_{1}^{2}$ and $\mathfrak{C}_{1}\mathfrak{C}_{2}$ hold for them. For the triple containing the $\vec{\rho}_{i}$, by using $z_{1}=r$ for one instance of $\mathfrak{C}_{1}$ inside the property $\mathfrak{C}_{1}^{2}$, there is still some $z_{1}\neq r$ such that $\mathfrak{C}_{1}$ holds for the triple of $\gamma^{*}_{i}$ and that value. Similarly, for the triple containing the $\vec{\sigma}_{i}$, by using $(z_{1},z_{2})=(t,u)$ for $\mathfrak{C}_{2}$, there is some $z_{1}\notin\{t,u\}$ such that $\mathfrak{C}_{1}$ holds for the $\gamma^{*}_{i}$ and that value. In either case $z_{1}\notin\{r,s,t,u,v\}$ and it is contained in the $\vec{\rho}_{i}$ or the $\vec{\sigma}_{i}$, so $\mathfrak{C}_{1}$ holds for the $\gamma_{i}$ as well.

Assume instead that we are in case \eqref{de:aligned-all}. Then the corresponding $\gamma^{*}_{i}$ are in case \eqref{de:aligned-alpl}, and $\mathfrak{C}_{1}^{4}\mathfrak{C}_{2}\mathfrak{C}_{3}^{2}\mathfrak{C}_{4}^{2}\mathfrak{C}_{5}\mathfrak{C}_{6}$ holds for them. In the same fashion as before, we can use $z_{1}=r$ or $(z_{1},z_{2})=(t,u)$ for one instance of $\mathfrak{C}_{1}$ or $\mathfrak{C}_{2}$ in the $\gamma^{*}_{i}$, and use disjoint values $z_{j}\notin\{r,s,t,u,v\}$ for the other properties, implying that $\mathfrak{C}_{1}^{3}\mathfrak{C}_{3}^{2}\mathfrak{C}_{4}^{2}\mathfrak{C}_{5}\mathfrak{C}_{6}$ holds for the $\gamma_{i}$.

In conclusion, the solution \eqref{eq:7to6-alphaprime} is aligned. The process can be repeated every time we go from \eqref{eq:7to6-alpha} to \eqref{eq:7to6-alphaprime}, since by $\mathfrak{P}(\mathcal{P}')$ the half-positions with labels in $\mathcal{P}$ sit in different restrictions so that every cycle is affected by the changes only once during the whole process (and the rest of the time it is part of $\vec{\beta}_{i}$). Hence, the final solution $(\eta'_{1},\eta'_{2},\eta'_{3},\phi'_{1},\phi'_{2},\phi'_{3},\lambda')$ is aligned as well.
\end{proof}

\section{Solutions in the reduced case}\label{se:solutions}

Our final step is to produce an aligned solution for the problem in Reduction~\ref{re:7}. Since the properties $\mathfrak{P}(\mathrm{sub'})$ and $\mathfrak{P}(\mathcal{P})$ as defined in \eqref{eq:prop-subprime} and \eqref{eq:prop-p} hold in $\mathcal{X}_{7}$, our string triple can be divided into pieces as in \eqref{eq:graphicsendprime}, each of which represents a problem in some smaller $\mathrm{Sym}(m)$ with $\nu$ odd (as defined in \eqref{eq:nu}). Let then $(\phi_{1}|_{j},\phi_{2}|_{j},\phi_{3}|_{j},\lambda|_{j})$ be the restriction of the string triple of Reduction~\ref{re:7} to the $j$-th piece, corresponding to a problem in $\mathrm{Sym}(n_{j})$, and let $(\eta_{i,j},\phi_{i}|_{j})\in\Pi_{n_{j}}\times\Phi_{n_{j}}$ be element strings such that the corresponding permutations $\alpha_{i,j}:=\mu(\eta_{i,j},\phi_{i}|_{j})$ satisfy $\alpha_{1,j}\alpha_{2,j}=\alpha_{3,j}$. We can concatenate the solutions by glueing together the element strings appropriately: for example, glueing the solutions at $j=1$ and $j=2$ together yields
\begin{align*}
\eta'_{i} & \in\Pi_{n_{1}+n_{2}}, & \eta'_{i}(x) & =\begin{cases}\eta_{i,1}(x) & (x\leq n_{1}), \\ n_{1}+\eta_{i,2}(x-n_{1}) & (x>n_{1}),\end{cases} \\
\phi'_{i} & \in\Phi_{n_{1}+n_{2}}, & \phi'_{i}(x) & =\begin{cases}(\phi_{i}|_{1})(x) & \left(x\leq n_{1}+\frac{1}{2}\right), \\ (\phi_{i}|_{2})(x-n_{1}) & \left(x\geq n_{1}+\frac{1}{2}\right),\end{cases} \\
\lambda' & \in\Lambda_{n_{1}+n_{2},\{\emptyset\}\cup\mathcal{N}\cup\mathcal{L}\cup\mathcal{T}\cup\mathcal{S}\cup\mathcal{P}} & \lambda'(x) & =\begin{cases}(\lambda|_{1})(x) & \left(x\leq n_{1}+\frac{1}{2}\right), \\ (\lambda|_{2})(x-n_{1}) & \left(x\geq n_{1}+\frac{1}{2}\right).\end{cases}
\end{align*}
Note that we are identifying the last half-position of $j=1$ and the first one of $j=2$, which makes sense since $(\phi_{i}|_{1})\left(n_{1}+\frac{1}{2}\right)=(\phi_{i}|_{2})\left(\frac{1}{2}\right)=\blacksquare$ and $(\lambda|_{1})\left(n_{1}+\frac{1}{2}\right)=(\lambda|_{2})\left(\frac{1}{2}\right)=\emptyset$. Concatenating solutions yields a solution for the original problem since, after renaming the sets on which the $\mathrm{Sym}(n_{j})$ are acting to make them disjoint, we have $\alpha_{1,1}\ldots\alpha_{1,j}\cdot\alpha_{2,1}\ldots\alpha_{2,j}=\alpha_{3,1}\ldots\alpha_{3,j}$.

Now we choose the solutions. If $m$ is odd and $\phi_{1}=\phi_{2}=\phi_{3}$ is made of one $m$-cycle, we choose
\begin{equation}\label{eq:baby1}
\begin{aligned}
(\eta_{1},\phi_{1}) & =(\eta_{2},\phi_{2})=\blacksquare\ 1\ 2\ \cdots\ m\ \blacksquare \\
(\eta_{3},\phi_{3}) & =\blacksquare\ [1\overset{\text{odds}}{\cdots}m]\ [2\overset{\text{evens}}{\cdots}m-1]\ \blacksquare
\end{aligned}
\end{equation}

If $d\geq 3$ is odd and $e$ is even, $\phi_{1}=\phi_{2}$ is made of one $d$-cycle and one $e$-cycle, and $\phi_{3}$ is made of one $(d+e)$-cycle, we choose
\begin{equation}\label{eq:baby2}
\begin{aligned}
(\eta_{1},\phi_{1}) & =\blacksquare\ x_{1}\ x_{2}\ \cdots\ x_{d}\ \blacksquare\ y_{1}\ y_{2}\ \cdots\ y_{e}\ \blacksquare \\
(\eta_{2},\phi_{2}) & =\blacksquare\ y_{1}\ x_{2}\ \cdots\ x_{d}\ \blacksquare\ x_{1}\ y_{2}\ \cdots\ y_{e}\ \blacksquare \\
(\eta_{3},\phi_{3}) & =\blacksquare\ [x_{1}\overset{\text{odds}}{\cdots}x_{d}]\ [y_{2}\overset{\text{evens}}{\cdots}y_{e}]\ [x_{2}\overset{\text{evens}}{\cdots}x_{d-1}]\ [y_{1}\overset{\text{odds}}{\cdots}y_{e-1}]\ \blacksquare
\end{aligned}
\end{equation}
In the above we have replaced $i$ with $x_{i}$ and $d+j$ with $y_{j}$, which makes the structure of the solution more evident; we shall act similarly below. If $\phi_{1}=\phi_{2}$ is made of one $r$-cycle and one $d$-cycle, we swap the order of the $e$-cycle and the $d$-cycle in \eqref{eq:baby2}: this is a necessary remark, since the properties in Definition~\ref{de:aligned} depend on the ordering of cycles in the $\phi_{i}$.

If $d\geq 3$ is odd and $e$ is even, $\phi_{1}=\phi_{3}$ is made of one $d$-cycle and one $e$-cycle, and $\phi_{2}$ is made of one $(d+e)$-cycle, we choose
\begin{equation}\label{eq:baby3}
\begin{aligned}
(\eta_{1},\phi_{1}) & =\blacksquare\ x_{1}\ x_{2}\ \cdots\ x_{d}\ \blacksquare\ y_{1}\ y_{2}\ \cdots\ y_{e}\ \blacksquare \\
(\eta_{2},\phi_{2}) & =\blacksquare\ x_{2}\ x_{3}\ y_{2}\ [x_{4}\overset{\text{all}}{\cdots}x_{d}]\ x_{1}\ y_{1}\ [y_{3}\overset{\text{all}}{\cdots}y_{e}]\ \blacksquare \\
(\eta_{3},\phi_{3}) & =\blacksquare\ x_{1}\ [x_{3}\overset{\text{odds}}{\cdots}x_{d}]\ y_{1}\ [x_{4}\overset{\text{evens}}{\cdots}x_{d-1}]\ \blacksquare\ x_{2}\ [y_{2}\overset{\text{evens}}{\cdots}y_{e}]\ [y_{3}\overset{\text{odds}}{\cdots}y_{e-1}]\ \blacksquare
\end{aligned}
\end{equation}
If $\phi_{1}=\phi_{3}$ is made of one $e$-cycle and one $d$-cycle, swap the order of the $e$-cycle and the $d$-cycle in \eqref{eq:baby3}.

If $d\geq 3$ is odd and $e$ is even, $\phi_{2}=\phi_{3}$ is made of one $d$-cycle and one $e$-cycle, and $\phi_{1}$ is made of one $(d+e)$-cycle, we choose
\begin{equation}\label{eq:baby4}
\begin{aligned}
(\eta_{1},\phi_{1}) & =\blacksquare\ x_{1}\ x_{2}\ \cdots\ x_{d}\ y_{1}\ y_{2}\ \cdots\ y_{e}\ \blacksquare \\
(\eta_{2},\phi_{2}) & =\blacksquare\ y_{1}\ x_{1}\ x_{2}\ [x_{4}\overset{\text{all}}{\cdots}x_{d}]\ \blacksquare\ y_{2}\ x_{3}\ [y_{3}\overset{\text{all}}{\cdots}y_{e}]\ \blacksquare \\
(\eta_{3},\phi_{3}) & =\blacksquare\ x_{1}\ [x_{4}\overset{\text{evens}}{\cdots}x_{d-1}]\ y_{1}\ [x_{3}\overset{\text{odds}}{\cdots}x_{d}]\ \blacksquare\ x_{2}\ [y_{3}\overset{\text{odds}}{\cdots}y_{e-1}]\ [y_{2}\overset{\text{evens}}{\cdots}y_{e}]\ \blacksquare
\end{aligned}
\end{equation}
If $\phi_{2}=\phi_{3}$ is made of one $e$-cycle and one $d$-cycle, swap the order of the $e$-cycle and the $d$-cycle in \eqref{eq:baby4}.

By $\mathfrak{P}(\mathrm{sub}')$ and $\mathfrak{P}(\mathcal{P})$, the four problems for which \eqref{eq:baby1}--\eqref{eq:baby2}--\eqref{eq:baby3}--\eqref{eq:baby4} provide solutions are the only possibilities that can occur for restrictions in $\mathcal{X}_{7}$. It is easy to check that they are solutions, i.e.\ $\mu(\eta_{1},\phi_{1})\mu(\eta_{2},\phi_{2})=\mu(\eta_{3},\phi_{3})$, and concatenating them appropriately solves the problem in Reduction~\ref{re:7}. We can say more, though.

\begin{proposition}\label{pr:7good}
There is an aligned solution for Reduction~\ref{re:7}.
\end{proposition}

\begin{proof}
Cycles share common positions only if they all sit between two consecutive $a_{1},a_{2}\in\tilde{P}_{n}$ with $c^{\blacksquare}(a_{1})=c^{\blacksquare}(a_{2})=3$, and we use distinct values in cycles that do not, so it is enough to examine \eqref{eq:baby1}--\eqref{eq:baby2}--\eqref{eq:baby3}--\eqref{eq:baby4} separately, and the properties of Definition~\ref{de:aligned} will follow for the entire strings.

In \eqref{eq:baby1}, the value $z_{1}=1$ shows that $\mathfrak{C}_{1}$ holds regardless of $m$, and if they share $\geq 2$ common positions (implying that $m\geq 3$) then a second $\mathfrak{C}_{1}$ holds using the value $z_{1}=2$, and $\mathfrak{C}_{2}$ holds using the pair of values $(z_{1},z_{2})=(2,3)$, yielding $\mathfrak{C}_{1}^{2}\cap\mathfrak{C}_{1}\mathfrak{C}_{2}$. Assume that $\lambda\left(m+\frac{1}{2}\right)\in\mathcal{L}$, or that $\lambda\left(m+\frac{1}{2}\right)\in\mathcal{P}$ and $\lambda\left(m+\frac{3}{2}\right)\in\mathcal{L}$. By $\mathfrak{P}(\mathcal{L}',19)$ then $m>18$, so the values from $4$ to $18$ can be used to show the remaining properties and obtain $\mathfrak{C}_{1}^{3}\mathfrak{C}_{3}^{2}\mathfrak{C}_{4}^{2}\mathfrak{C}_{5}\mathfrak{C}_{6}$ or $\mathfrak{C}_{1}^{4}\mathfrak{C}_{2}\mathfrak{C}_{3}^{2}\mathfrak{C}_{4}^{2}\mathfrak{C}_{5}\mathfrak{C}_{6}$.

Consider \eqref{eq:baby2} and the triple of cycles that includes the ones of length $d$. Since $d\geq 3$, use $z_{1}=x_{2}$ to show $\mathfrak{C}_{1}$. If there is a common break with a label in $\mathcal{P}$, by $\mathfrak{P}(\mathrm{sub}')$ we have $d>3$, so use $z_{1}=x_{3}$ to show $\mathfrak{C}_{1}$ and $(z_{1},z_{2})=(x_{3},x_{4})$ to show $\mathfrak{C}_{2}$, giving $\mathfrak{C}_{1}^{2}\cap\mathfrak{C}_{1}\mathfrak{C}_{2}$. If there is a common break with a label in $\mathcal{L}$, or in $\mathcal{P}$ and neighbouring a label in $\mathcal{L}$, then by $\mathfrak{P}(\mathcal{L}',19)$ we have $d>18$ and we can use the values $x_{i}$ with $5\leq i\leq 19$ to prove the remaining properties and obtain the whole $\mathfrak{C}_{1}^{3}\mathfrak{C}_{3}^{2}\mathfrak{C}_{4}^{2}\mathfrak{C}_{5}\mathfrak{C}_{6}$ or $\mathfrak{C}_{1}^{4}\mathfrak{C}_{2}\mathfrak{C}_{3}^{2}\mathfrak{C}_{4}^{2}\mathfrak{C}_{5}\mathfrak{C}_{6}$. Similarly, for the triple that includes the $e$-cycles, use $z_{1}=y_{2}$ to show $\mathfrak{C}_{1}$, if the common break has a label in $\mathcal{P}$ (and by $\mathfrak{P}(\mathrm{sub})$ then $e>3$) use $z_{1}=y_{3}$ and $(z_{1},z_{2})=(y_{3},y_{4})$ to show $\mathfrak{C}_{1}$ and $\mathfrak{C}_{2}$, and in the last case (by $\mathfrak{P}(\mathcal{L}',19)$ then $e>18$) use the $y_{i}$ with $5\leq i\leq 19$ to show the remaining properties. If the $d$-cycles and the $e$-cycles are swapped, the same values apply. For the two triples we used disjoint values, the $x_{i}$ for the $d$-cycles and the $y_{i}$ for the $e$-cycles, so Definition~\ref{de:aligned}\eqref{de:aligned-disj} is respected.

For \eqref{eq:baby3} we argue similarly. This time, when the $d$-cycles are involved, we use $x_{3}$ for $\mathfrak{C}_{1}$, $x_{1}$ and $(x_{1},x_{2})$ for $\mathfrak{C}_{1}$ and $\mathfrak{C}_{2}$, and again $x_{i}$ with $4\leq i\leq 18$ for the other properties. When the $e$-cycles are involved instead, we use $y_{2}$ for $\mathfrak{C}_{1}$, $y_{3}$ and $(y_{3},y_{4})$ for $\mathfrak{C}_{1}$ and $\mathfrak{C}_{2}$, and $y_{i}$ with $5\leq i\leq 19$ for the other properties. If the $d$-cycles and the $e$-cycles are swapped, the same values apply, and Definition~\ref{de:aligned}\eqref{de:aligned-disj} is respected.

For \eqref{eq:baby4} the triple with the $e$-cycles uses the same values as \eqref{eq:baby3}, and for the triple with the $d$-cycles use $y_{1}$ for $\mathfrak{C}_{1}$, $x_{1}$ and $(x_{1},x_{2})$ for $\mathfrak{C}_{1}$ and $\mathfrak{C}_{2}$, and $x_{i}$ with $4\leq i\leq 18$ for the other properties.

Putting together all the cases described above, there is an aligned solution for Reduction~\ref{re:7}.
\end{proof}

As is evident from the proof above, there is no need to ask for $\delta_{7}$ small and $n$ large to obtain Proposition~\ref{pr:7good}. The condition becomes necessary when undoing the reductions, because we need the construction of the various $\theta_{k}$ to make sense.

\begin{proof}[Proof of Thm.~\ref{th:main2}]
By Proposition~\ref{pr:7good} there is an aligned solution for Reduction~\ref{re:7}. Then we use, one after the other, Lemmas~\ref{le:7to6}--\ref{le:6to5}--\ref{le:5to4}--\ref{le:4to3}--\ref{le:3to2}--\ref{le:2to1}--\ref{le:1to0} and obtain Theorem~\ref{th:main2}.
\end{proof}

\section*{Acknowledgements}

The author is grateful to Attila Mar\'oti for attracting his attention to the problem, for discussions about how to tackle it, and for providing helpful comments on an early version of the paper. He also thanks Pham Huu Tiep for suggesting useful references. The author was funded by a Young Researcher Fellowship from the R\'enyi Institute.

\bibliography{Bibliography}

\begin{thebibliography}{KKM24}

\bibitem[Ber72]{Ber72}
E.~Bertram.
\newblock Even permutations as a product of two conjugate cycles.
\newblock {\em J. Combin. Theory Ser. A}, 12(3):368--380, 1972.

\bibitem[Bre78]{Bre78}
J.~L. Brenner.
\newblock Covering theorems for {FINASIGS VIII} - {A}lmost all conjugacy classes in ${A}_{n}$ have exponent $\leq 4$.
\newblock {\em J. Aust. Math. Soc.}, 25:210--214, 1978.

\bibitem[DLR24]{DLR24}
D.~Dona, M.~W. Liebeck, and K.~Rekv\'enyi.
\newblock Involutions in finite simple groups as products of conjugates.
\newblock {\em Comm. Algebra}, 52(9):3750--3761, 2024.

\bibitem[DMP24]{DMP24}
D.~Dona, A.~Mar\'oti, and L.~Pyber.
\newblock Growth of products of subsets in finite simple groups.
\newblock {\em Bull. Lond. Math. Soc.}, 56(8):2704--2710, 2024.

\bibitem[Dvi85]{Dvi85}
Y.~Dvir.
\newblock Covering properties of permutation groups.
\newblock In Z.~Arad and M.~Herzog, editors, {\em Products of Conjugacy Classes in Groups}, pages 197--221. Springer-Verlag, Berlin (Germany), 1985.

\bibitem[FH04]{FH04}
W.~Fulton and J.~Harris.
\newblock {\em Representation theory: a first course}, volume 129 of {\em Graduate Texts in Mathematics}.
\newblock Springer, New York (USA), 2004.

\bibitem[FM25]{FM25}
F.~Fumagalli and A.~Mar\'oti.
\newblock On the {G}owers trick for classical simple groups.
\newblock {\em J. Pure Appl. Algebra}, 229(1), 2025.
\newblock Article no. 107833.

\bibitem[GM21]{GM21}
M.~Garonzi and A.~Mar\'oti.
\newblock Alternating groups as products of four conjugacy classes.
\newblock {\em Arch. Math. (Basel)}, 116(2):121--130, 2021.

\bibitem[HKL04]{HKL04}
M.~Herzog, G.~Kaplan, and A.~Lev.
\newblock Representation of permutations as products of two cycles.
\newblock {\em Discrete Math.}, 285:323--327, 2004.

\bibitem[HKL08]{HKL08}
M.~Herzog, G.~Kaplan, and A.~Lev.
\newblock Covering the alternating groups by products of cycle classes.
\newblock {\em J. Combin. Theory Ser. A}, 115(7):1235--1245, 2008.

\bibitem[Ked09]{Ked09}
K.~S. Kedlaya.
\newblock Product-free subsets of groups, then and now.
\newblock In T.~Y. Chow and D.~C. Isaksen, editors, {\em Communicating Mathematics}, volume 479 of {\em Contemporary Mathematics}, pages 169--177. American Mathematical Society, Providence (USA), 2009.

\bibitem[KKM24]{KKM24}
H.~Kishnani, R.~Kundu, and S.~C. Mishra.
\newblock Alternating groups as products of cycle classes - {II}.
\newblock {\em J. Algebraic Combin.}, 59(3):635--660, 2024.

\bibitem[KLS24]{KLS24}
N.~Keller, N.~Lifshitz, and O.~Sheinfeld.
\newblock Improved covering results for conjugacy classes of symmetric groups via hypercontractivity.
\newblock {\em Forum Math. Sigma}, 12, 2024.
\newblock Article e85.

\bibitem[KM22]{KM22}
E.~I. Khukhro and V.~D. Mazurov.
\newblock Unsolved problems in group theory - {T}he {K}ourovka {N}otebook.
\newblock No.~20, \texttt{https://kourovka-notebook.org} and \texttt{arXiv:1401.0300v33}, 2022.

\bibitem[LM23]{LM23}
N.~Lifshitz and A.~Marmor.
\newblock Bounds for characters of the symmetric group: a hypercontractive approach.
\newblock \texttt{arXiv:2308.08694v3}, 2023.

\bibitem[LS01]{LS01}
M.~W. Liebeck and A.~Shalev.
\newblock Diameters of finite simple groups: sharp bounds and applications.
\newblock {\em Ann. of Math. (2)}, 154:383--406, 2001.

\bibitem[LS08]{LS08}
M.~J. Larsen and A.~Shalev.
\newblock Characters of symmetric groups: sharp bounds and applications.
\newblock {\em Invent. Math.}, 174(3):645--687, 2008.

\bibitem[LST24]{LST24}
M.~J. Larsen, A.~Shalev, and P.~H. Tiep.
\newblock Products of normal subsets.
\newblock {\em Trans. Amer. Math. Soc.}, 377(2):863--885, 2024.

\bibitem[LT23]{LT23}
M.~J. Larsen and P.~H. Tiep.
\newblock Squares of conjugacy classes in alternating groups.
\newblock \texttt{arXiv:2305.04806v1}, 2023.

\bibitem[MP21]{MP21}
A.~Mar\'oti and L.~Pyber.
\newblock A generalization of the diameter bound of {L}iebeck and {S}halev for finite simple groups.
\newblock {\em Acta Math. Hungar.}, 164(2):350--359, 2021.

\bibitem[MS07]{MS07}
T.~W. M{\"u}ller and J.-C. {Schlage-Puchta}.
\newblock Character theory of symmetric groups, subgroup growth of {F}uchsian groups, and random walks.
\newblock {\em Adv. Math.}, 213(2):919--982, 2007.

\bibitem[NP11]{NP11}
N.~Nikolov and L.~Pyber.
\newblock Product decompositions of quasirandom groups and a {J}ordan type theorem.
\newblock {\em J. Eur. Math. Soc. (JEMS)}, 13:1063--1077, 2011.

\bibitem[Rod02]{Rod02}
D.~M. Rodgers.
\newblock Generating and covering the alternating or symmetric group.
\newblock {\em Comm. Algebra}, 30(1):425--435, 2002.

\bibitem[Roi96]{Roi96b}
Y.~Roichman.
\newblock Upper bound on the characters of the symmetric groups.
\newblock {\em Invent. Math.}, 125(3):451--485, 1996.

\bibitem[Vis98]{Vis98}
U.~Vishne.
\newblock Mixing and covering in the symmetric groups.
\newblock {\em J. Algebra}, 205(1):119--140, 1998.

\end{thebibliography}
\bibliographystyle{alpha}

\end{document}